\documentclass[11pt,a4paper]{article}
\usepackage{graphics}
\usepackage{fancyhdr}
\usepackage{color}
\usepackage{mathrsfs}
\usepackage{amsfonts,amssymb,amsmath,amsthm,bm}
\usepackage{authblk}
\usepackage{enumitem}

\usepackage[unicode,pdftex]{hyperref}
\hypersetup{
	colorlinks = true,
	citecolor = green,
	anchorcolor = blue,
	linkcolor = blue}

\newtheorem{thm}{Theorem}[section]
\newtheorem{lemma}[thm]{Lemma}

\newtheorem{prop}[thm]{Proposition}
\newtheorem{conj}[thm]{Conjecture}
\newtheorem{defi}[thm]{Definition}
\newtheoremstyle{rem}{10pt}{10pt}{\rmfamily}{}{\bfseries}{.}{.5em}{} 
\theoremstyle{rem}
\newtheorem{rem}[thm]{Remark}

\numberwithin{equation}{section} 

\title{Fractional and anisotropic Gagliardo--Nirenberg inequalities with applications}
\author{Jie Chen\footnote{E-mail address: jiechern@jmu.edu.cn}}
\affil[]{\scriptsize \textit{School of Science, Jimei University, Xiamen 361021, P.R. China}}
\date{}

\begin{document}
	\maketitle
	\begin{abstract}
		In this article, we first investigate the necessary and sufficient conditions on the ranges of $s,s_1,s_2\in \mathbb{R}$, $1\leq p_1,p_2,q\leq \infty$, $\theta_1,\theta_2\geq 0$ for the validity of the fractional Gagliardo--Nirenberg inequalities 
		$$\|D^su\|_{L^q(\mathbb{R}^d)}\lesssim \|D^{s_1}u\|_{L^{p_1}(\mathbb{R}^d)}^{\theta_1}\|D^{s_2}u\|_{L^{p_2}(\mathbb{R}^d)}^{\theta_2}.$$
		Secondly, we consider the anisotropic Gagliardo--Nirenberg inequalities
		$$\|u\|_{L^q(\mathbb{R}^d)}\lesssim \|u\|_{L^{p_0}(\mathbb{R}^d)}^{\theta_0}\prod_{j=1}^n\|D_{x_j}^{s_j}u\|_{L^{p_j}(\mathbb{R}^d)}^{\theta_j}.$$
		We show the sharp conditions of these inequalities except ``one" case. Finally, we establish the profile decomposition associated with the above inequalities when $p_j = 2$, $0\leq j\leq n$. Using these results, we establish existence of extremizers and construct soliton solutions of relevant dispersive equations.
		\vspace{5pt}
		
		\vspace{5pt}
		
	\end{abstract}

	\section{Introduction}
	The classical Gagliardo--Nirenberg inequalities read: Let $1\leq p_0,p_1,q\leq \infty$, $l,m\in \mathbb{N}_0:=\{0,1,2,\cdots\}$, $l<m$, $l/m\leq \theta\leq 1$, $d((1-\theta)/p_0+\theta/p_1-1/q) = m\theta-l$, $(l,q)\neq (0,p_0)$. If $(\theta,q) = (1,\infty)$, we further assume that $p_1 = 1$. Then for any $u$ in the Schwartz space $\mathcal{S}(\mathbb{R}^d)$, one has
	$$\sum_{|\alpha| = l}\|\partial_x^\alpha u\|_{L^q(\mathbb{R}^d)}\lesssim \|u\|_{L^{p_0}(\mathbb{R}^d)}^{1-\theta}\sum_{|\alpha| = m}\|\partial_x^\alpha u\|_{L^{p_1}(\mathbb{R}^d)}^{\theta},$$
	where $\partial_x^\alpha = \partial_{x_1}^{\alpha_1}\cdots \partial_{x_d}^{\alpha_d}$ for $\alpha = (\alpha_1,\cdots,\alpha_d)\in \mathbb{N}_0^d$, $x = (x_1,\cdots,x_d)\in \mathbb{R}^d$. These inequalities, discovered by Gagliardo \cite{gagliardo1959ulteriori} and Nirenberg \cite{nirenberg1959elliptic}, are very useful in the study of PDEs. There is a vast literature on inequalities of this type since then. For these inequalities in Besov or Triebel--Lizorkin spaces, see Hajaiej--Molinet--Ozawa--Wang \cite{hajaiej2010sufficient}, Triebel \cite{triebel2014gagliardo}, and reference therein. For the complete and ultimate story of these inequalities in Sobolev spaces, we refer to the works of Brezis--Mironescu \cite{brezis2018gagliardo,brezis2019sobolev}.
		
	Naturally, we may also consider inequalities of this type in Bessel potential spaces:
	\begin{equation}\label{gninbessel}
		\|D^su\|_{L^q(\mathbb{R}^d)}\lesssim \|D^{s_1}u\|_{L^{p_1}(\mathbb{R}^d)}^{\theta_1}\|D^{s_2}u\|_{L^{p_2}(\mathbb{R}^d)}^{\theta_2}
	\end{equation}
	where $D^s:=\mathscr{F}^{-1}|\xi|^s\mathscr{F}$ and
	\begin{align*}
		\mathscr{F}f(\xi) = \frac{1}{(2\pi)^{d/2}}\int_{\mathbb{R}^d}f(x)e^{-ix\cdot \xi}~dx.
	\end{align*}
	Without loss of generality, we would assume:
	\begin{equation}\label{basicred}
		\begin{aligned}
			&(s,q),(s_1,p_1),(s_2,p_2) \mbox{ are distinct from each other and}\\
			& 1\leq p_1,p_2,q\leq \infty,~p_1\leq p_2,~ s_1,s_2,s\in \mathbb{R},~ 0<\theta_1 = 1-\theta_2\leq 1.
		\end{aligned}
	\end{equation}
	By scaling arguments and by testing the inequality with functions whose Fourier support at high frequencies, we obtain necessary conditions for the inequality \eqref{gninbessel} to hold:
	\begin{equation}\label{scalingandhighfre}
		d\left(\frac{\theta_1}{p_1}+\frac{\theta_2}{p_2}-\frac{1}{q}\right) = \theta_1s_1+\theta_2s_2-s\geq 0.
	\end{equation}
	In fact, The above conditions are almost sufficient, except for some limiting cases. The first goal of this paper is to obtain necessary and sufficient conditions on the indices for \eqref{gninbessel} to hold. Let
	\begin{equation}\label{formalin}
		C(d,s_1,p_1,s_2,p_2,\theta_1,q):=\sup\frac{\|D^su\|_{L^q(\mathbb{R}^d)}}{\|D^{s_1}u\|_{L^{p_1}(\mathbb{R}^d)}^{\theta_1}\|D^{s_2}u\|_{L^{p_2}(\mathbb{R}^d)}^{\theta_2}}.
	\end{equation}
	where the sup is taking for all nonzero $\mathscr{F}u\in C_0^\infty(\mathbb{R}^d\setminus \{0\})$.
	\begin{thm}\label{firstmain}
		Assume \eqref{basicred}--\eqref{scalingandhighfre}. $C(d,s_1,p_1,s_2,p_2,\theta_1,q)=\infty$ if and only if $(\theta_1,p_1) = (1,1)$, or $(\theta_1,q) = (1,\infty)$, or
		\begin{itemize}
			\item $q = \infty$, $s_1-s_2 = d(1/p_1-1/p_2)$, or
			\item $1 = p_1<p_2<\infty$, $s_1-s_2 = d(1-1/p_2)$, $\theta_1>1-p_2/q$, or
			\item $p_1 = 1$, $p_2 = \infty$, $s_1-s_2 = d$,  $\theta_1>1-1/q$, or
			\item $1 = p_1<q<p_2<\infty$, $0<s_1-s_2< d(1-1/p_2)$, $\theta_1 = \frac{1/q-1/p_2}{1-1/p_2}$, or
			\item $p_1 = 1$, $p_2 = \infty$, $1<q<2$, $0<s_1-s_2<d$, $\theta_1 = 1/q$.
		\end{itemize}
	\end{thm}
	\begin{rem}
		To establish inequality \eqref{gninbessel} for the case $p_1 = 1$, we rely on the remarkable interpolation result of Cohen \cite{cohen1999ondelettes}, which asserts that for $1<q<p$, $\theta = p'/q'$, $t = s\theta$, and either $s>0$ or $s<-d/p'$, we have
		\begin{equation}\label{interesult}
			(L^1,B^s_{p,p})_{\theta,q} = B^t_{q,q},
		\end{equation}
		where $B^{s}_{p,p}$ denotes the Besov space (see, e.g., Triebel \cite{triebel1983theory}).	We adopt the definition of real interpolation spaces from Bergh--L\"{o}fstr\"{o}m \cite{bergh1976interpolation}.
	\end{rem}
	\begin{rem}
		Although our result here seems very different from Theorem 1 in Brezis--Mironescu \cite{brezis2018gagliardo} (where the index conditions are independent of the spatial dimension $d$), the proof of Theorem \ref{firstmain} is highly related to the proof in \cite{brezis2018gagliardo}. In fact, we shall appropriately generalize the important one‑dimensional counterexample constructed in \cite{brezis2018gagliardo} to higher dimensions.
	\end{rem}
	
	The second part of this paper is to consider the following anisotropic Gagliardo--Nirenberg inequalities
	\begin{equation}\label{aniso}
		\|u\|_{L^q(\mathbb{R}^d)}\lesssim \|u\|_{L^{p_0}(\mathbb{R}^d)}^{\theta_0}\prod_{j=1}^n\|D_{x_j}^{s_j}u\|_{L^{p_j}(\mathbb{R}^d)}^{\theta_j},\quad \forall~u\in \mathcal{S}(\mathbb{R}^d)
	\end{equation}
	where $x = (x_1,\cdots,x_n)\in \mathbb{R}^{d_1}\oplus\cdots\oplus \mathbb{R}^{d_n}=\mathbb{R}^d$, $D_{x_j}^{s_j}=\mathscr{F}^{-1}|\xi_j|^{s_j}\mathscr{F}$. Without loss of generality, we assume:
	\begin{equation}\label{anibasicred}
		\begin{aligned}
			&n\in \mathbb{N},~n\geq 2,~ d_j\in \mathbb{N},~s_j\geq 0,~\mbox{for any}~1\leq j\leq n,\\
			&1\leq p_j,q\leq \infty,~\theta_j\geq 0,~\mbox{for any}~0\leq j\leq n,~p_0\neq q,~\mbox{and}~\sum_{j=0}^n\theta_j = 1.
		\end{aligned}
	\end{equation}
	
	Similar to considering \eqref{gninbessel}, we likewise have the following necessary conditions:
	\begin{equation}\label{scalingar}
		d_j\left(\sum_{j=0}^n\frac{\theta_j}{p_j}-\frac{1}{q}\right) = \theta_js_j,\quad \forall~1\leq j\leq n.
	\end{equation}
	\begin{thm}\label{mainani}
		Assume \eqref{anibasicred}--\eqref{scalingar} and $(\theta_0,q)\neq (0,\infty)$. Then \eqref{aniso} holds if and only if it holds for any $u(x) =\prod_{j=1}^nu_j(x_j)$, $u_j\in \mathcal{S}(\mathbb{R}^{d_j})$. Equivalently, \eqref{aniso} holds except $1<q<\infty$, $\theta_0 = 0$, $p_{j_0} = 1$, $s_{j_0} = d_{j_0}/q'$, $\theta_{j_0}>0$ for some $1\leq j_0\leq n$, $p_j = q$ for any $1\leq j\leq n$ with $j\neq j_0$.
	\end{thm}
	\begin{rem}
		For related partial results, see, e.g., Adams \cite{adams1988anisotropic}, Algervik \cite{algervik2010}, and Esfahani \cite{esfahani2015anisotropic}; for the particular case $q = \infty$ and $s_j = 1$, we refer to Porretta \cite{porretta2020note}.
	\end{rem}
	\begin{rem}
		We refer to inequality \eqref{aniso} as \textit{critical} when $\theta_0 = 0$, and as subcritical when \textit{subcritical}.
	\end{rem}
	
	The case $(\theta_0,q) = (0,\infty)$ is different. In fact we have
	\begin{thm}\label{couninfty}
		Assume  \eqref{anibasicred}--\eqref{scalingar} and $(\theta_0,q)= (0,\infty)$. If $1<p_{j_0}<\infty$, $\theta_{j_0}>0$ for some $1\leq j\leq n$, then \eqref{aniso} does not hold.
	\end{thm}
	Essentially, we are missing the case where $p_j$ equals $1$ for any $1\leq j\leq n$. More precisely, we have the following conjecture.
	\begin{conj}\label{conjp1}
		Let $d_j\in \mathbb{N},$ $\theta_j>0$ for any $1\leq j\leq n$ and $\sum_{j=1}^n\theta_j = 1$ and let $d = \sum_{j=1}^n d_j$. The inequality
		\begin{equation}\label{end1}
			\|u\|_{L^\infty(\mathbb{R}^d)}\lesssim \prod_{j=1}^n \|D_{x_j}^{\frac{d_j}{\theta_j}}u\|_{L^1(\mathbb{R}^d)}^{\theta_j}
		\end{equation}
		does not hold.
	\end{conj}
	\begin{rem}
		For $d_1 = d_2 = 1$, $\theta_1 = \theta_2 = 1/2$, Conjecture \ref{conjp1} claims
		$$\|u(x,y)\|^2_{L^\infty(\mathbb{R}^2)}\lesssim \|\partial_{xx}u(x,y)\|_{L^1(\mathbb{R}^2)}\|\partial_{yy}u(x,y)\|_{L^1(\mathbb{R}^2)}$$
		does not hold. Since
		$$|u(x,y)| = \left|\int_{-\infty}^x\int_{-\infty}^y\partial_{xy}u(x',y')~dx'dy'\right|\leq \|\partial_{xy}u\|_{L^1(\mathbb{R}^2)},$$
		a scaling argument then disproves the inequality
		$$\|\partial_{xy}u\|_{L^1(\mathbb{R}^2)}\lesssim \|\partial_{xx}u(x,y)\|_{L^1(\mathbb{R}^2)}+\|\partial_{yy}u(x,y)\|_{L^1(\mathbb{R}^2)}.$$
		This is the well-known Ornstein non-inequality \cite{ornstein1962non}. For further results in this direction, see Kazaniecki–Stolyarov–Wojciechowski \cite{kazaniecki2017anisotropic}. However, to the best of our knowledge, estimate \eqref{end1} remains completely open.
	\end{rem}
	
	Next, we prove the existence of extremizers for \eqref{aniso} when $p_j = 2$ for each $j = 0,\cdots, n$. Let $\bm{d} = (d_1,\cdots,d_n)$, $\bm{s} = (s_1,\cdots,s_n)$. We define
	\begin{align*}
		H_{\bm{d}}^{\bm{s}} = \{u\in L^2(\mathbb{R}^d):D^{s_j}_{x_j}u\in L^2(\mathbb{R}^d), j = 1,\cdots,n\},
	\end{align*}
	with norm
	\begin{equation}\label{defiHds}
		\|u\|_{H_{\bm{d}}^{\bm{s}}}^2:= \|u\|_{L^2}^2+\sum_{j=1}^n\|D_{x_j}^{s_j}u\|_{L^2}^2.
	\end{equation}
	The homogeneous space $\dot{H}_{\bm{d}}^{\bm{s}}$ is then defined as the completion of $H_{\bm{d}}^{\bm{s}}$ under the norm
	\begin{equation}\label{defidotHds}
		\|u\|_{\dot{H}_{\bm{d}}^{\bm{s}}}^2 = \sum_{j=1}^n \|D_{x_j}^{s_j}u\|_{L^2(\mathbb{R}^d)}^2.
	\end{equation}
	\begin{thm}\label{existextri}
		Assume   \eqref{anibasicred}--\eqref{scalingar} and $p_j = 2$, $s_j>0$ for each $j = 0,\cdots, n$. If $\theta_0\neq 0$, there exists a extremizer $Q^{\bm{s},q}_{\bm{d}}\in H^{\bm{s}}_{\bm{d}}$ of the inequality
		\begin{equation}\label{sharp}
			\|u\|_{L^q(\mathbb{R}^d)}\leq C_{op} \|u\|_{L^{2}(\mathbb{R}^d)}^{\theta_0}\prod_{j=1}^n\|D_{x_j}^{s_j}u\|_{L^{2}(\mathbb{R}^d)}^{\theta_j},\quad \forall~u\in H^{\bm{s}}_{\bm{d}}.
		\end{equation}
		If $\theta_0 = 0$, there exists a extremizer $Q^{\bm{s}}_{\bm{d}}\in \dot{H}^{\bm{s}}_{\bm{d}}$ of the inequality
		\begin{equation}\label{sharpc}
			\|u\|_{L^q(\mathbb{R}^d)}\leq C_{op}\prod_{j=1}^n\|D_{x_j}^{s_j}u\|_{L^{2}(\mathbb{R}^d)}^{\theta_j},\quad \forall~u\in \dot{H}^{\bm{s}}_{\bm{d}}.
		\end{equation}
	\end{thm}
	\begin{rem}
		In Section \ref{profiledeco}, we establish the profile decomposition associated with the embeddings $H^{\bm{s}}_{\bm{d}} \hookrightarrow L^q$ and $\dot{H}^{\bm{s}}_{\bm{d}} \hookrightarrow L^q$. This result will be used to prove the existence of extremizers for various Gagliardo–Nirenberg and Sobolev-type inequalities. 
	\end{rem}
	
	This paper is organized as follows. Section \ref{fractionalgn} is devoted to the fractional Gagliardo–Nirenberg inequality and the proof of Theorem \ref{firstmain}; Theorem \ref{mainani} is proved in Subsection \ref{proofofagn}, while the counterexample for the case $(\theta_0,q) = (0,\infty)$ is presented in Subsection \ref{counter}. In Subsection \ref{animixed}, we extend the anisotropic Gagliardo–Nirenberg inequality to mixed derivatives in mixed Lebesgue spaces. In Section \ref{profiledeco}, we study the profile decomposition and apply it in Subsection \ref{solitonsolu} to problems such as the construction of soliton solutions for a class of dispersive equations. Finally, in Subsection \ref{solitonestimate}, we provide estimates for the energy of these solitons and establish nonscattering results in the energy space for the corresponding dispersive equations.
	
	\section{Fractional Gagliardo--Nirenberg inequalities}\label{fractionalgn}
	In this section, we study inequality \eqref{gninbessel} and prove Theorem \ref{firstmain}. Let $\varphi\in C_0^\infty(\mathbb{R})$ be a non-negative radial function satisfying $0\leq \varphi\leq 1$, $\varphi = 1$ on $[-1,1]$, and $\mathrm{supp}~\varphi\subset (-3/2,3/2)$. Let $\psi(\cdot) = \varphi(\cdot)-\varphi(2\cdot)$. For $k\in \mathbb{Z}$, define the Littlewood--Paley projection operators $\Delta_k$ by
	$$\Delta_k = \mathscr{F}^{-1} \psi(|\xi|/2^k)\mathscr{F}.$$
	
	\subsection{Sufficient conditions for \texorpdfstring{\eqref{gninbessel}}{(1.1)}}\label{truepart}
	We begin by examining the sufficient conditions for the validity of \eqref{gninbessel}.
	\begin{prop}\label{subcritical}
		Assume \eqref{basicred}--\eqref{scalingandhighfre}, $\theta_1<1$, and $s_1-d/p_1\neq s_2-d/p_2$. If $p_2\leq q$ or $s_1\theta_1+s_2\theta_2>s$, then \eqref{gninbessel} holds.
	\end{prop}
	\begin{proof}[\textbf{Proof}]
		Since $s_1-d/p_1\neq s_2-d/p_2$, by a scaling argument we may assume that $\|D^{s_j}u\|_{L^{p_j}} = 1$. First, if $p_1\leq p_2\leq q$, then by the Bernstein inequality one has
		$$\|\Delta_ku\|_{L^q}\lesssim 2^{-s_j k+ k(\frac{d}{p_j}-\frac{d}{q})}\|D^{s_j}u\|_{L^{p_j}}\lesssim 2^{k(\frac{d}{p_j}-\frac{d}{q}-s_j)},\quad j = 1,2.$$
		Note that $(d/p_1-d/q-s_1)(d/p_2-d/q-s_2) <0$ due to $s_1-d/p_1\neq s_2-d/p_2$ and \eqref{scalingandhighfre}. Thus, we have
		$$\|u\|_{L^q}\leq \sum_{k\in \mathbb{Z}}\|\Delta_k u\|_{L^q}\lesssim \sum_{k\in \mathbb{Z}}\min_{j=1,2}\{2^{k(\frac{d}{p_j}-\frac{d}{q}-s_j)}\}\lesssim 1.$$
		
		If $p_2>q$, then we have $p_1<q$ by \eqref{scalingandhighfre}. By the H\"{o}lder inequality and the Bernstein inequality, we have
		\begin{align*}
			\|\Delta_k u\|_{L^q}\leq\|\Delta_k u\|_{L^{p_1}}^{\frac{1/q-1/p_2}{1/p_1-1/p_2}} \|\Delta_k u\|_{L^{p_2}}^{\frac{1/p_1-1/q}{1/p_1-1/p_2}}\lesssim 2^{-k\frac{s_2(1/p_1-1/q)+s_1(1/q-1/p_2)}{1/p_1-1/p_2}}.
		\end{align*}
		By \eqref{scalingandhighfre}, we have
		\begin{align*}
			\alpha:=\frac{s_2(1/p_1-1/q)+s_1(1/q-1/p_2)}{1/p_1-1/p_2} = \frac{s_1\theta_1+s_2\theta_2}{d\theta_2(1/p_1-1/p_2)}\left(\frac{d}{p_1}-\frac{d}{q}-s_1\right).
		\end{align*}
		Note that $\alpha (d/p_1-d/q-s_1)>0$. We thus have
		$$\|u\|_{L^q}\leq \sum_{k\in \mathbb{Z}}\|\Delta_ku\|_{L^q}\lesssim \min\{2^{-k\alpha},2^{k(d/p_1-d/q-s_2)}\}\lesssim 1,$$
		and the proof is complete.
	\end{proof}
	
	\begin{prop}[\cite{hajaiej2010sufficient}]\label{littlewood}
		Assume \eqref{basicred}--\eqref{scalingandhighfre}, $p_1>1$, and $1<q<\infty$. Then \eqref{gninbessel} holds.
	\end{prop}
	\begin{proof}[\textbf{Proof}]
		If $p_2\neq \infty$ or $\theta_1 = 1$, we can use Corollary 1.5 in \cite{hajaiej2010sufficient} directly. If $p_2 = \infty$ and $0<\theta_1<1$, we define $p$ by $\theta_1/p = 1/q$, $\tilde{s} = -s_2\theta_2/\theta_1$. Note that $1<p_1\leq  p<\infty$. By the Sobolev inequality, we have $ \|D^{\tilde{s}}u\|_{L^{p}}\lesssim\|D^{s_1}u\|_{L^{p_1}}$. Thus we only need to consider the case $p_1 = q\theta_1$, $s_1 = -s_2\theta_2/\theta_1$. If $s_2 = 0$, we conclude the proof by the H\"{o}lder inequality. Below we apply a similar argument as in the proof of Theorem C.2 in \cite{WangHHG}. By the Littlewood--Paley theory, we have
		\begin{align*}
			\|u\|^q_{L^q}\sim \||\Delta_ku(x)|_{l^2_k}\|^q_{L^q_x} \sim \int_0^\infty \lambda^q |\{x:|\Delta_k u(x)|_{l^2_k}>\lambda\}|\frac{d\lambda}{\lambda}.
		\end{align*}
		Let $A:=\|D^{s_1}u\|_{L^{p_1}}$, $B:=\|D^{s_2}u\|_{L^\infty}$. By the Bernstein inequality we have
		\begin{align*}
			|\Delta_ku(x)|\lesssim \min\{2^{(-s_1+d/p_1)k}A,2^{-s_2k}B\}.
		\end{align*}
		If $s_2<0$, choose $K_0(\lambda)$ such that $2^{-s_2K_0(\lambda)}B\sim c\lambda$ for some fixed constant $0<c\ll 1$. Then
		\begin{align*}
			&\quad\int_0^\infty \lambda^q |\{x:|\Delta_k u(x)|_{l^2_k}>\lambda\}|\frac{d\lambda}{\lambda}\\
			&\leq \int_0^\infty \lambda^q |\{x:|\Delta_k u(x)|_{l^2_{k\geq K_0(\lambda)}}>\lambda/2\}|\frac{d\lambda}{\lambda}\\
			&\leq \int_0^\infty \lambda^q |\{x:|2^{ks_1}\Delta_k u(x)|_{l^2_{k\geq K_0(\lambda)}}>2^{K_0(\lambda)s_1}\lambda/2\}|\frac{d\lambda}{\lambda}\\
			&\lesssim \int_0^\infty \lambda^q |\{x:|2^{ks_1}\Delta_k u(x)|_{l^2_{k}}>(B/\lambda)^{s_1/s_2}\lambda\}|\frac{d\lambda}{\lambda}
		\end{align*}
		Let $\mu = (B/\lambda)^{s_1/s_2}\lambda$. We obtain
		\begin{align*}
			\|u\|^q_{L^q}&\lesssim \int_0^\infty (\mu B^{-s_1/s_2})^{q/(1-s_1/s_2)} |\{x:|2^{ks_1}\Delta_k u(x)|_{l^2_{k}}>\mu\}|\frac{d\mu}{\mu}\\
			&\lesssim \|D^{s_1}u\|_{L^{q/(1-s_1/s_2)}}^{q/(1-s_1/s_2)}B^{qs_1/(s_1-s_2)}\sim A^{p_1}B^{q-p_1}.
		\end{align*}
		Thus $\|u\|_{L^q}\lesssim \|D^{s_1}u\|_{L^{p_1}}^{\theta_1}\|D^{s_2}u\|_{L^{p_2}}^{\theta_2}$.
		
		If $s_2>0$, choose $K_0(\lambda)$ such that $2^{-s_2K_0(\lambda)}B\sim c\lambda$ for some fixed constant $0<c\ll 1$. Then
		\begin{align*}
			&\quad\int_0^\infty \lambda^q |\{x:|\Delta_k u(x)|_{l^2_k}>\lambda\}|\frac{d\lambda}{\lambda}\\
			&\leq \int_0^\infty \lambda^q |\{x:|\Delta_k u(x)|_{l^2_{k\leq K_0(\lambda)}}>\lambda/2\}|\frac{d\lambda}{\lambda}\\
			&\leq \int_0^\infty \lambda^q |\{x:|2^{ks_1}\Delta_k u(x)|_{l^2_{k\leq K_0(\lambda)}}>2^{K_0(\lambda)s_1}\lambda/2\}|\frac{d\lambda}{\lambda}\\
			&\lesssim \int_0^\infty \lambda^q |\{x:|2^{ks_1}\Delta_k u(x)|_{l^2_{k}}>(B/\lambda)^{s_1/s_2}\lambda\}|\frac{d\lambda}{\lambda}
		\end{align*}
		By the same argument as for the case $s_2<0$, we finish the proof.
	\end{proof}
	
	\begin{prop}\label{endsobo}
		Assume \eqref{basicred}--\eqref{scalingandhighfre}, $1 = p_1<p_2<q<\infty$, $s_1-s_2 = d/p'_2$, and $\theta_1 \leq 1- p_2/q$. Then \eqref{gninbessel} holds.
	\end{prop}
	\begin{proof}[\textbf{Proof}]
		Since $\|u\|_{L^q}\lesssim \|D^{s_2}u\|_{L^{p_2}}$, we only need to show the inequality for $\theta_1 = 1- p_2/q$. By scaling, we may assume that $\|D^{d/q'}u\|_{L^1} = 1$. As in the proof of Proposition \ref{littlewood}, for any $t>0$, we choose an integer $K$ such that $2^{K d/q}\sim ct$ for some fixed constant $0<c\ll 1$. It follows from Bernstein’s inequality that
		\begin{align*}
			|\Delta_k u(x)|_{l^1_{k\leq K}}&\lesssim 2^{Kd/q}\|2^{-kd/q}\Delta_ku(x)\|_{l_k^\infty L_x^\infty}\\
			&\lesssim 2^{Kd/q}\|2^{kd/q'}\Delta_ku(x)\|_{l_k^\infty L^1}\\
			&\lesssim ct.
		\end{align*}
		Thus,
		\begin{align*}
			|\{x:|\Delta_k u(x)|_{l^2_k}>t\}|&\leq |\{x:|\Delta_k u(x)|_{l^2_{k\geq K}}>t/2\}|\\
			&\leq |\{x:|2^{kd(1/p_2-1/q)}\Delta_ku(x)|_{l^2_k}>2^{Kd(1/p_2-1/q)}t/2\}|\\
			&\leq  |\{x:|2^{ks_2}\Delta_ku(x)|_{l^2_k}>t^{q/p_2}c^{q/p_2-1}/2\}|.
		\end{align*}
		By the Littlewood--Paley theory, we obtain
		\begin{align*}
			\|u\|_{L^q}^q\sim \||\Delta_k u|_{l^2_k}\|_{L^q}^q&\sim \int_0^\infty t^{q-1}|\{x:|\Delta_k u(x)|_{l^2_k}>t\}|~dt\\
			&\lesssim \int_0^\infty t^{q-1}|\{x:|2^{ks_2}\Delta_ku(x)|_{l^2_k}>t^{q/p_2}c^{q/p_2-1}/2\}|~dt\\
			&\lesssim \int_0^\infty t^{p_2-1}|\{x:|2^{ks_2}\Delta_ku(x)|_{l^2_k}>t\}|~dt\\
			&\sim \||2^{ks_2}\Delta_k u|_{l^2_k}\|_{L^{p_2}}^{p_2}\sim \|D^{s_2}u\|_{L^{p_2}}^{p_2}.
		\end{align*}
		Therefore, $\|u\|_{L^q}\lesssim \|D^{s_2}u\|_{L^{p_2}}^{p_2/q}\sim \|D^{d/q'}u\|_{L^1}^{1-p_2/q}\|D^{s_2}u\|_{L^{p_2}}^{p_2/q}$.
	\end{proof}
	
	\begin{prop}\label{coheninter}
		Assume \eqref{basicred}--\eqref{scalingandhighfre}, $\theta_1<1$ and $1 = p_1<q<p_2\leq\infty$. If $s_1-s_2 > d/p'_2$ or $s_1-s_2\leq 0$, then \eqref{gninbessel} holds.
	\end{prop}
	\begin{proof}[\textbf{Proof}]
		Since $s_1-d/p_1\neq s_2-d/p_2$, by Proposition \ref{subcritical} we only need to consider the case $s_1\theta_1+s_2\theta_2 = 0$. It is equivalent to
		$$\|D^{-s_1}u\|_{L^q}\lesssim \|u\|_{L^1}^{1-p_2'/q'}\|D^{s_2-s_1}u\|_{L^{p_2}}^{p_2'/q'}.$$
		Let $p_2 = p$, $\alpha = (s_2-s_1)p'$. We would show a slightly stronger result. A scaling argument shows that it is enough to prove, for $\alpha>0$ or $\alpha<-d$ and $1<q<p \leq \infty$, that
		\begin{equation}\label{aim}
			\|u\|_{F^{\alpha/q'}_{q,1}}\lesssim \|u\|_{L^1}^{1-p'/q'}\|u\|_{F^{\alpha/p'}_{p,\infty}}^{p'/q'},
		\end{equation}
		with $F^s_{p,q}$ being the Triebel–Lizorkin space (we refer to \cite{triebel1983theory} for the definition). Choose  $1<\tilde{q}<q$. By Theorem 1.4 in \cite{hajaiej2010sufficient} we have
		\begin{align*}
			\|u\|_{F^{\alpha/q'}_{q,1}}\lesssim \|u\|^{1-\theta}_{F^{{\alpha}/{\tilde{q}'}}_{\tilde{q},\infty}}\|u\|_{F^{{\alpha}/{p'}}_{p,\infty}}^\theta,\quad \frac{1-\theta}{\tilde{q}}+\frac{\theta}{p} = \frac{1}{q}.
		\end{align*}
		By the interpolation result \eqref{interesult} of \cite{cohen1999ondelettes}, for $1/\tilde{q}' = \mu/q'$ and either $\alpha>0$ or $\alpha<-d$, we have
		$$\|u\|_{F^{\alpha/\tilde{q}'}_{\tilde{q},\tilde{q}}}\lesssim \|u\|_{L^1}^{1-\mu}\|u\|_{F^{\alpha/q'}_{q,q}}^\mu.$$
		Thus
		\begin{align*}
			\|u\|_{F^{\alpha/q'}_{q,1}}\lesssim\|u\|^{1-\theta}_{F^{{\alpha}/{\tilde{q}'}}_{\tilde{q},\infty}}\|u\|_{F^{{\alpha}/{p'}}_{p,\infty}}^\theta
			&\lesssim \|u\|^{1-\theta}_{F^{{\alpha}/{\tilde{q}'}}_{\tilde{q},\tilde{q}}}\|u\|_{F^{{\alpha}/{p'}}_{p,\infty}}^\theta\\
			&\lesssim(\|u\|_{L^1}^{1-\mu}\|u\|_{F^{\alpha/q'}_{q,q}}^\mu)^{1-\theta}\|u\|_{F^{{\alpha}/{p'}}_{p,\infty}}^\theta\\
			&\lesssim(\|u\|_{L^1}^{1-\mu}\|u\|_{F^{\alpha/q'}_{q,1}}^\mu)^{1-\theta}\|u\|_{F^{{\alpha}/{p'}}_{p,\infty}}^\theta.
		\end{align*}
		Noting that $p'/q' = \theta/(1-\mu(1-\theta))$, we deduce \eqref{aim}. The embeddings $F^{s}_{p,1}\hookrightarrow F^s_{p,2}\hookrightarrow F^s_{p,\infty}$, together with the Littlewood–Paley theory, complete the proof.
	\end{proof}
	
	\begin{prop}\label{doubleend}
		Assume \eqref{basicred}--\eqref{scalingandhighfre}, $\theta_1<1$, $p_1 = 1$, $p_2 = \infty$, and $q\geq 2$. If $s_1-s_2\neq d$ or $\theta_1\leq 1/q'$, then \eqref{gninbessel} holds.
	\end{prop}
	\begin{proof}[\textbf{Proof}]
		By \eqref{scalingandhighfre}, we have $\theta_1\geq 
		1/q$. If $s_1-s_2\leq 0$ or $s_1-s_2>d$, the proof follows from Proposition \ref{coheninter}. If $0<s_1-s_2<d$, $\theta_1>1/q$, then Proposition \ref{subcritical} applies. Hence it remains to consider either $0<s_1-s_2<d$ with $\theta_1 = 1/q$ or $s_1-s_2 = d$ with $1/q\leq \theta_1\leq 1/q'$. For $q = 2$, we have
		\begin{align*}
			\|u\|_{L^2}^2 &= \int_{\mathbb{R}^d}u\bar{u}~dx = \int_{\mathbb{R}^d}D^{s}uD^{-s}\bar{u}~dx\\
			&\leq \|D^{s_1}u\|_{L^1}^{1/2}\|D^{-s_1}u\|_{L^\infty}^{1/2}.
		\end{align*}
		For $q>2$, $\theta_1 = 1/q$, we have $s_2-qs_2 = s_1$. Then
		\begin{align*}
			\|u\|_{L^q}&\lesssim \|D^{-s_2(q/2-1)}u\|_{L^2}^{2/q}\|D^{s_2}u\|_{L^\infty}^{1-2/q}\\
			&\lesssim (\|D^{s_2-qs_2}u\|_{L^1}^{1/2}\|D^{s_2}u\|_{L^\infty}^{1/2})^{2/q}\|D^{s_2}u\|_{L^\infty}^{1-2/q}\\
			&\sim \|D^{s_1}u\|_{L^1}^{\theta_1}\|D^{s_2}u\|_{L^\infty}^{\theta_2}.
		\end{align*}
		For $q>2$, $s_1-s_2 = d$, we have $s_1 = d/q'$, $s_2 = -d/q$. Then
		\begin{align*}
			\|u\|_{L^q}&\lesssim \|D^{d/q'}u\|_{L^1}^{1-2/q}\|D^{d/2-d/q}u\|_{L^2}^{2/q}\\
			&\lesssim \|D^{d/q'}u\|_{L^1}^{1-2/q}(\|D^{d/q'}u\|_{L^1}^{1/2}\|D^{-d/q}u\|_{L^\infty}^{1/2})^{2/q}\\
			&\sim \|D^{s_1}u\|_{L^1}^{1/q'}\|D^{s_2}u\|_{L^\infty}^{1/q}.
		\end{align*}
		For $1/q<\theta_1\leq 1/q'$, we can choose $\theta\in (0,1]$ such that $(1-\theta)/q+\theta/q' = \theta_1$. Then
		\begin{align*}
			\|u\|_{L^q}&\lesssim (\|D^{s_1}u\|_{L^1}^{1/q}\|D^{s_2}u\|_{L^\infty}^{1/q'})^{1-\theta}(\|D^{s_1}u\|_{L^1}^{1/q'}\|D^{s_2}u\|_{L^\infty}^{1/q})^\theta\\
			&\lesssim \|D^{s_1}u\|_{L^1}^{\theta_1}\|D^{s_2}u\|_{L^\infty}^{\theta_2},
		\end{align*}
		which completes the proof.
	\end{proof}
	Combining Propositions \ref{subcritical}–\ref{doubleend}, we finish the proof of the “only if” part of Theorem \ref{firstmain}.
	
	\subsection{Necessary conditions for \texorpdfstring{\eqref{gninbessel}}{(1.1)}}
	Let $\delta_d$ be the Dirac measure on $\mathbb{R}^d$. For any $K\in \mathbb{N}$, define
	\begin{equation}\label{scalingcounter}
		u_{d,K}:=\sum_{k=1}^K \Delta_k \delta_d.
	\end{equation}
	\begin{lemma}\label{cutdirac}
		$\|u_{d,K}\|_{L^1(\mathbb{R}^d)}\sim_d 1$, $\|D^{-d}u_{d,K}\|_{L^\infty(\mathbb{R}^d)}\sim_d K$, and
		$$\|D^{-\frac{d}{p'}}u\|_{L^p(\mathbb{R}^d)}\sim_{d,p} K^{\frac{1}{p}},\quad 1<p<\infty.$$
	\end{lemma}
	\begin{proof}[\textbf{Proof}]
		By the definition of $u_{d,K}$, we have
		\begin{align*}
			D^{-\frac{d}{p'}}u_{d,K} &= \mathscr{F}^{-1}\left(|\xi|^{-\frac{d}{p'}}\sum_{k=1}^K \psi(|\xi|/2^k)\right)\\
			& = \mathscr{F}^{-1}(|\xi|^{-\frac{d}{p'}}(\varphi(|\xi|/2^K)-\varphi(|\xi|))).
		\end{align*}
		Thus $\|u_{d,K}\|_{L^1(\mathbb{R}^d)}\leq 2\|\mathscr{F}^{-1}(\varphi(|\xi|))\|_{L^1(\mathbb{R}^d)}\lesssim 1$. By the Plancherel identity, one has
		\begin{equation*}
			\|D^{-\frac{d}{2}}u_{d,K}\|_{L^2(\mathbb{R}^d)}\sim \||\xi|^{-\frac{d}{2}}\psi(|\xi|/2^k)\|_{l^2_{1\leq k\leq K}L^2(\mathbb{R}^d)} \sim K^{\frac{1}{2}}.
		\end{equation*}
		Note that $\mathscr{F}(u_{d,K})\geq 0$. Thus
		$$\|D^{-d}u_{d,K}\|_{L^\infty}\sim  \|(|\xi|^{-d}(\varphi(|\xi|/2^K)-\varphi(|\xi|))\|_{L^1(\mathbb{R}^d)}\sim K.$$
		By the Littlewood--Paley theory, for $1<p<2$ we have
		\begin{align*}
			\|D^{-\frac{d}{p'}}u_{d,K}\|_{L^{p}(\mathbb{R}^d)}\sim \| 2^{-k\frac{d}{p'}}\Delta_k\delta_d\|_{L^pl^2_{1\leq k\leq K}}\lesssim \|2^{-k\frac{d}{p'}}\Delta_k\delta_d\|_{l^p_{1\leq k\leq K}L^p}\sim K^{\frac{1}{p}},
		\end{align*}
		while for $2<p<\infty$,
		\begin{align*}
			\|D^{-\frac{d}{p'}}u_{d,K}\|_{L^{p}(\mathbb{R}^d)}&\sim \| 2^{-k\frac{d}{p'}}\Delta_k\delta_d\|_{L^pl^2_{1\leq k\leq K}} \gtrsim \|2^{-k\frac{d}{p'}}\Delta_k\delta_d\|_{l^p_{1\leq k\leq K}L^p}\sim K^{\frac{1}{p}}.
		\end{align*}
		Applying Proposition \ref{endsobo} for $2<p<\infty$, we obtain
		$$\|D^{-\frac{d}{p'}}u_{d,K}\|_{L^p(\mathbb{R}^d)}\lesssim \|u_{d,K}\|_{L^1(\mathbb{R}^d)}^{1-\frac{2}{p}}\|D^{-\frac{d}{2}}u_{d,K}\|_{L^2(\mathbb{R}^d)}^{\frac{2}{p}}\lesssim K^{\frac{1}{p}}.$$
		Applying Proposition \ref{littlewood} for $1<p<2$, we obtain
		\begin{align*}
			K^{\frac{1}{2}}\sim \|D^{-\frac{d}{2}}u_{d,K}\|_{L^2(\mathbb{R}^d)}&\lesssim \|D^{-\frac{d}{p'}}u_{d,K}\|_{L^p(\mathbb{R}^d)}^{\frac{1}{2}}\|D^{-\frac{d}{p}}u_{d,K}\|_{L^{p'}(\mathbb{R}^d)}^{\frac{1}{2}}\\
			&\lesssim \|D^{-\frac{d}{p'}}u_{d,K}\|_{L^p(\mathbb{R}^d)}^{\frac{1}{2}} K^{\frac{1}{2p'}}.
		\end{align*}
		Thus $\|D^{-d/p'}u_{d,K}\|_{L^p(\mathbb{R}^d)}\gtrsim K^{1/p}$ for $1<p<2$. Combining the above inequalities yields the conclusion.
	\end{proof}
		
	\begin{prop}\label{scalingcritical}
		Assume \eqref{basicred}--\eqref{scalingandhighfre}. If $(\theta_1,p_1) = (1,1)$, or $(\theta_1,q) = (1,\infty)$, or $q =\infty$ and $s_1-s_2 = d(1/p_1-1/p_2)$, or
		\begin{itemize}
			\item $1 = p_1<p_2<\infty$, $s_1-s_2 = d(1-1/p_2)$, $\theta_1>1-p_2/q$, or
			\item $p_1 = 1$, $p_2 = \infty$, $s_1-s_2 = d$, $\theta_1>1-1/q$,
		\end{itemize}
		then $C(d,s_1,p_1,s_2,p_2,\theta_1,q)=\infty$.
	\end{prop}
	\begin{proof}[\textbf{Proof}]
		If $\theta_1 = 1$, we can choose $s_2,p_2$ arbitrary. Thus, without loss of generality, by \eqref{scalingandhighfre} we assume that $s_1-d/p_1 = s_1-d/p_2 = s-d/q$. It follows from \eqref{basicred} that $p_1,p_2,q$ are pairwise distinct. Taking $u = D^{-s-d/q'}u_{d,K}$, where $u_{d,K}$ is given by \eqref{scalingandhighfre}, yields
		\begin{align*}
			C(d,s_1,p_1,s_2,p_2,\theta_1,q)\geq \frac{\|D^{-d/q'}u_{d,K}\|_{L^q}}{\|D^{-d/p_1'}u_{d,K}\|_{L^{p_1}}^{\theta_1}\|D^{-d/p_2'}u_{d,K}\|_{L^{p_2}}^{\theta_2}}.
		\end{align*}
		If $q = \infty$, by Lemma \ref{cutdirac} we have
		\begin{align*}
			C(d,s_1,p_1,s_2,p_2,\theta_1,q)\gtrsim K^{1-\frac{\theta_1}{p_1}-\frac{\theta_2}{p_2}}\rightarrow \infty,~ \mbox{ as }K\rightarrow\infty.
		\end{align*}
		If $1 = p_1<q<\infty$, $p_2<\infty$, and $\theta_1>1-p_2/q$, by Lemma \ref{cutdirac} we have
		\begin{align*}
			C(d,s_1,p_1,s_2,p_2,\theta_1,q)\gtrsim K^{\frac{1}{q}-\frac{1-\theta_1}{p_2}}\rightarrow \infty,~ \mbox{ as }K\rightarrow\infty.
		\end{align*}
		If $1 = p_1<q<p_2 = \infty$ and $\theta_1>1-1/q$, by Lemma \ref{cutdirac} we have
		\begin{align*}
			C(d,s_1,p_1,s_2,p_2,\theta_1,q)\gtrsim K^{\frac{1}{q}-(1-\theta_1)}\rightarrow \infty,~ \mbox{ as }K\rightarrow\infty.
		\end{align*}
		We finish the proof.
	\end{proof}
	
	Across different dimensions, the following simple relation holds.
	\begin{lemma}\label{dimensionreduction}
		Assume \eqref{basicred}--\eqref{scalingandhighfre}. If $C(d_0,s_1,p_1,s_2,p_2,\theta_1,q)=\infty$, then for any $d>d_0$ we have $C(d,s_1,p_1,s_2,p_2,\theta_1,q)=\infty$.
	\end{lemma}
	\begin{proof}[\textbf{Proof}]
		Let $x=(x',x''), \xi=(\xi',\xi'')\in \mathbb{R}^d$, with $x',\xi'\in \mathbb{R}^{d_0}$ and $x'',\xi''\in \mathbb{R}^{d-d_0}$. Denote by $\mathscr{F}_n$ the Fourier transform on $\mathbb{R}^n$. For any $f(x')\in \mathcal{S}(\mathbb{R}^{d_0})$ with $\mathrm{supp}(\mathscr{F}_{d_0}f)\subset \{\xi':|\xi'|\geq 2\}$, define $u(x) = f(x') \mathscr{F}_{d-d_0}^{-1}(\varphi(|\xi''|))(x'')$. Then for any $\alpha\in \mathbb{R}$ and $1\leq p\leq \infty$, we have
		$$\|D_x^\alpha u(x)\|_{L^p(\mathbb{R}^d)}\sim \|D^\alpha_{x'}f(x')\|_{L^p(\mathbb{R}^{d_0})}.$$
		
		If $C(d,s_1,p_1,s_2,p_2,\theta_1,q)<\infty$, we have
		\begin{align*}
			\|D^s_{x'}f(x')\|_{L^q(\mathbb{R}^{d_0})}&\sim \|D_x^s u(x)\|_{L^q(\mathbb{R}^d)}\\
			&\lesssim \|D_x^{s_1} u(x)\|_{L^{p_1}(\mathbb{R}^d)}^{\theta_1}\|D_x^{s_2} u(x)\|_{L^{p_2}(\mathbb{R}^d)}^{\theta_2}\\
			&\sim \|D^{s_1}_{x'}f(x')\|_{L^{p_1}(\mathbb{R}^{d_0})}^{\theta_1}\|D^{s_2}_{x'}f(x')\|_{L^{p_2}(\mathbb{R}^{d_0})}^{\theta_2}
		\end{align*}
		By scaling, $\|D^s_{x'}f(x')\|_{L^q(\mathbb{R}^{d_0})}\lesssim \|D^{s_1}_{x'}f(x')\|_{L^{p_1}(\mathbb{R}^{d_0})}^{\theta_1}\|D^{s_2}_{x'}f(x')\|_{L^{p_2}(\mathbb{R}^{d_0})}^{\theta_2}$ holds for any $\mathscr{F}_{d_0}f\in C_0^\infty (\mathbb{R}^{d_0}\setminus \{0\})$. 
%
		We finish the proof by contradiction.
	\end{proof}
	
	By Proposition \ref{scalingcritical} and Lemma \ref{dimensionreduction}, to show the counterexample for $0<s_1-s_2<d/p_2'$ we only need to consider the case $(d-1)/p_2'<s_1-s_2<d/p_2'$. Let $\delta = (s_1-s_2)p_2'$. By analogy, we may expect to construct a function whose support is approximately 
	$(d-\delta)$-dimensional. In fact, Brezis–Mironescu \cite{brezis2018gagliardo} gave a construction for the one-dimensional case. We extend this construction to arbitrary dimensions.
	
	Let $\alpha = d-\delta$ and $\mathfrak{v} = (1,1,\cdots,1)\in \mathbb{R}^d$. For each $k\in \mathbb{N}$, set $\epsilon_k= k^{-1/\alpha}$. $Q_k = [0,\epsilon_k]^d$, and define $E^1_k = \cup_{l=0}^{k-1} (l\epsilon_k^\alpha\mathfrak{v}+Q_k)$. For $j\geq 2$, recursively define
	$$E_k^j = \bigcup_{l=0}^{k-1} (l\epsilon_k^\alpha\mathfrak{v}+\epsilon_kE_k^{j-1}).$$
	Finally, let $u_k^j = \chi_{E_k^j}$.
	
	\begin{lemma}\label{thekeylemma}
		Let $d\in \mathbb{N}$, $d-1<\delta<d$, and $1<p<\infty$. For $k,j\in \mathbb{N}$, $\|u_k^j\|_{L^1(\mathbb{R}^d)} = \epsilon_k^{j\delta}$, $\|D^{-\delta}u_k^j\|_{L^\infty(\mathbb{R}^d)}\lesssim_{d,\delta}j\epsilon_k^{j\delta}$, and
		\begin{equation*}
			\limsup_{k\rightarrow \infty} \frac{\|D^{-\delta/p'}u_k^j\|_{L^p(\mathbb{R}^d)}}{\epsilon_k^{j\delta}}\sim_{d,\delta,p}\liminf_{k\rightarrow \infty} \frac{\|D^{-\delta/p'}u_k^j\|_{L^p(\mathbb{R}^d)}}{\epsilon_k^{j\delta}}\sim_{d,\delta,p} j^{\frac{1}{p}}.
		\end{equation*}
	\end{lemma}
	\begin{proof}[\textbf{Proof}]
		It follows directly from the definition that $|E_k^1| = k\epsilon_k^d$, $|E_k^j| = k\epsilon_k^d|E_k^{j-1}| = (k\epsilon_k^d)^j = \epsilon_k^{\delta j}$. Thus 
		$$\|u_k^j\|_{L^1(\mathbb{R}^d)} = |E_k^j| = \epsilon_k^{\delta j}.$$
		Let
		$$f_{k,\mu}^j(x) = \int_{E_k^j}\frac{dy}{|x-y|^\mu}.$$
		For $0<\beta<d$, one has $D^{-\beta}u_k^j(x)\sim f_{k,d-\beta}^j(x)$. For each $j\geq 2$, we have
		\begin{align*}
			f^j_{k,\mu}(x) =  \int_{E_k^j}\frac{dy}{|x-y|^\mu}	
			& = \sum_{l=0}^{k-1}\int_{l\epsilon_k^\alpha\mathfrak{v}+\epsilon_kE_k^{j-1}}\frac{dy}{|x-y|^\mu}
			\\
			& = \epsilon_k^{d-\mu} \sum_{l=0}^{k-1}f_{k,\mu}^{j-1}\left(\frac{x-l\epsilon_k^\alpha \mathfrak{v}}{\epsilon_k}\right).
		\end{align*}
		Thus
		\begin{align*}
			\|f_{k,\mu}^j\|_{L^r(\mathbb{R}^d)} = \epsilon_k^{d-\mu+\frac{d}{r}}\left\|\sum_{l=0}^{k-1}f_{k,\mu}^{j-1}(x-l\epsilon_k^{\alpha-1}\mathfrak{v})\right\|_{L^r(\mathbb{R}^d)}.
		\end{align*}
		For $r = \infty$, $\mu = \alpha$, we have
		\begin{equation}\label{induction}
			\|f_{k,\alpha}^j\|_{L^\infty(\mathbb{R}^d)} = \epsilon_k^{\delta}\left\|\sum_{l=0}^{k-1}f_{k,\alpha}^{j-1}(x-l\epsilon_k^{\alpha-1}\mathfrak{v})\right\|_{L^\infty(\mathbb{R}^d)}.
		\end{equation}
		If $|x|\geq 2$, we have
		\begin{equation}\label{largex}
			2^{-\mu}|x|^{-\mu}|E_k^j|\leq f_{k,\mu}^j(x)\leq 2^\mu |x|^{-\mu}|E_k^j|.
		\end{equation}
		Note that for $k>k_\alpha\gg 1$ one has $|l\epsilon_k^{\alpha-1}\mathfrak{v}-j\epsilon_k^{\alpha-1}\mathfrak{v}|>4$, $\forall~l\neq j$. Define $l_x\in \{0,1,\cdots,k-1\}$ such that $|x-l_x\epsilon_k^{\alpha-1}\mathfrak{v} |= \min_{0\leq l\leq k-1}|x-l\epsilon_k^{\alpha-1}\mathfrak{v}| $, then $|x-l\epsilon_k^{\alpha-1}\mathfrak{v}|\geq |l\epsilon_k^{\alpha-1}\mathfrak{v}-l_x\epsilon_k^{\alpha-1}\mathfrak{v}|/2$. Thus
		\begin{align*}
			\sum_{l=0}^{k-1}f_{k,\alpha}^{j-1}(x-l\epsilon_k^{\alpha-1}\mathfrak{v})&\leq f_{k,\alpha}^{j-1}(x-l_x\epsilon_k^{\alpha-1}\mathfrak{v})+\sum_{l\neq l_x}f_{k,\alpha}^{j-1}(x-l\epsilon_k^{\alpha-1}\mathfrak{v})\\
			&\leq \|f_{k,\alpha}^{j-1}\|_{L^\infty(\mathbb{R}^d)}+\sum_{l\neq l_x} 2^{\alpha}|x-l\epsilon_k^{\alpha-1}\mathfrak{v}|^{-\alpha}|E_k^{j-1}|\\
			&\leq \|f_{k,\alpha}^{j-1}\|_{L^\infty(\mathbb{R}^d)}+\sum_{l\neq l_x} 2^{2\alpha}|l-l_x|^{-\alpha}\epsilon_k^{-\alpha(\alpha-1)}|E_k^{j-1}|\\
			&\leq \|f_{k,\alpha}^{j-1}\|_{L^\infty(\mathbb{R}^d)}+2^{2\alpha-1}\epsilon_k^{\delta(j-1)}\sum_{l=1}^k l^{-\alpha}\epsilon_k^{-\alpha(\alpha-1)}\\
			&\leq \|f_{k,\alpha}^{j-1}\|_{L^\infty(\mathbb{R}^d)}+c_\alpha\epsilon_k^{\delta(j-1)}.
		\end{align*}
		By \eqref{induction}, we obtain
		$$\|f_{k,\alpha}^j\|_{L^\infty(\mathbb{R}^d)}\leq \epsilon_k^\delta\|f_{k,\alpha}^{j-1}\|_{L^\infty(\mathbb{R}^d)}+c_\alpha\epsilon_k^{j\delta}.$$
		Then $\|f_{k,\alpha}^j\|_{L^\infty(\mathbb{R}^d)}\leq \epsilon_k^{j\delta}((j-1)c_\alpha+\|f_{k,\alpha}^1\|_{L^\infty(\mathbb{R}^d)})$, $j\geq 2$.
		We also have
		\begin{equation}\label{inftystart}
			\|f_{k,\alpha}^1\|_{L^\infty(\mathbb{R}^d)}\leq c_\alpha+\sup_x\int_0^1 \frac{dy}{|x-y|^\alpha}\leq C_\alpha.
		\end{equation}
		Thus $\|f_{k,\alpha}^j\|_{L^\infty(\mathbb{R}^d)}\leq C_\alpha j \epsilon_k^{js}$ for any $j\geq 1$.
		
		Let
		\begin{equation*}
			r = p,\quad \mu = d-\frac{\delta}{p'} = \frac{d}{p}+\frac{\alpha}{p'} = \frac{\delta}{p}+\alpha
		\end{equation*}
		By \eqref{induction}, one has
		\begin{equation}\label{decom}
			\begin{aligned}
				\|f_{k,\mu}^j\|_{L^p(\mathbb{R}^d)}^p& = \epsilon_k^{(d-\alpha/p')p}\left\|\sum_{l=0}^{k-1}f_{k,\mu}^{j-1}(x-l\epsilon_k^{\alpha-1}\mathfrak{v})\right\|_{L^p(\mathbb{R}^d)}^p\\
				& = \epsilon_k^{\delta p+\alpha}\sum_{m=0}^{k-1}\left\|\sum_{l=0}^{k-1}f_{k,\mu}^{j-1}(x-l\epsilon_k^{\alpha-1}\mathfrak{v})\right\|_{L^p(B_m)}^p\\
				&\quad+ \epsilon_k^{\delta p+\alpha}\left\|\sum_{l=0}^{k-1}f_{k,\mu}^{j-1}(x-l\epsilon_k^{\alpha-1}\mathfrak{v})\right\|_{L^p(\Omega)}^p\\
				&:=\epsilon_k^{\delta p+\alpha} \sum_{l=0}^{k-1}I_m+J,
			\end{aligned}
		\end{equation}
		where
		$$B_l = \{x: |x-l\epsilon_k^{\alpha-1}\mathfrak{v}|\leq \epsilon_k^{\alpha-1}/2\},\quad \Omega = \mathbb{R}^d\setminus(\cup_{l=0}^{k-1}B_l).$$
		For any $x\in \Omega$, we have
		\begin{align}\label{thetailterm}
			\sum_{l=0}^{k-1}f_{k,\mu}^{j-1}(x-l\epsilon_k^{\alpha-1}\mathfrak{v}) \sim |E_k^{j-1}|\sum_{l=0}^{k-1}|x-l\epsilon_k^{\alpha-1}\mathfrak{v}|^{-\mu}.
		\end{align}
		By scaling, we obtain
		\begin{equation}\label{tailterm}
			\begin{aligned}
				&\quad\left\|\sum_{l=0}^{k-1}|x-l\epsilon_k^{\alpha-1}\mathfrak{v}|^{-\mu}\right\|_{L^p(\Omega)}^p \\
				&= \epsilon_k^{(d-\mu p)(\alpha-1)}\int_{(\cup_l B(l\mathfrak{v},1/2))^c}\left(\sum_{l=0}^{k-1}\frac{1}{|y-l\mathfrak{v}|^\mu}\right)^p~dy,
			\end{aligned}
		\end{equation}
		where $B(z,r)$ denotes the closed ball in $\mathbb{R}^d$ centered at $z$ with radius $r$. If $0<\mu<1$, then for any $y\in (\cup_l B(l\mathfrak{v},1/2))^c$, 
		$$\sum_{l=0}^{k-1}\frac{1}{|y-l\mathfrak{v}|^\mu}\lesssim k^{1-\mu},$$
		and hence
		\begin{align*}
			\int_{(\cup_l B(l\mathfrak{v},1/2))^c}\left(\sum_{l=0}^{k-1}\frac{1}{|y-l\mathfrak{v}|^\mu}\right)^p~dy&\lesssim \int_{|y|\geq 2k}\frac{k^p}{|y|^{\mu p}}~dy+ k^{d+(1-\mu)p}\\
			&\lesssim k^{d+(1-\mu)p}.
		\end{align*}
		If $\mu\geq 1$, we have $d\geq 2$. Let $C=\{y:|y|\geq 1/2,y = \lambda \mathfrak{v}+z, (z,\mathfrak{v}) = 0, |z|\leq 2k, |\lambda|\leq 2/3\}$. Then
		\begin{align*}
			&\quad\int_{(\cup_l B(l\mathfrak{v},1/2))^c}\left(\sum_{l=0}^{k-1}\frac{1}{|y-l\mathfrak{v}|^\mu}\right)^p~dy\\
			&\lesssim \int_{|y|\gtrsim k}\frac{k^p}{|y|^{\mu p}}~dy+k \int_{C}\left(\sum_{l=-k}^{k-1}\frac{1}{|y-l\mathfrak{v}|^\mu}\right)^p~dy\\
			&\lesssim k^{d+(1-\mu)p}+k \int_{C}\left(\sum_{l=-k}^{k-1}\frac{1}{|y-l\mathfrak{v}|^\mu}\right)^p~dy.
		\end{align*}
		Thus
		\begin{align*}
			&\quad \int_{C}\left(\sum_{l=-k}^{k-1}\frac{1}{|y-l\mathfrak{v}|^\mu}\right)^p~dy\\
			&\lesssim \int_0^{2k}\rho^{d-2}d\rho\int_{-2/3}^{2/3} \left(\sum_{l=-k}^{k-1}\frac{\chi_{\rho+|\lambda|\gtrsim 1}}{|\lambda-l|^\mu+\rho^\mu}\right)^p~d\lambda\\
			&\lesssim \sum_{n=0}^{2k}\int_{n}^{n+1}(n+1)^{d-2}~d\rho\int_{-2/3}^{2/3} \left(\sum_{l=-k}^{k-1}\frac{\chi_{\rho+|\lambda|\gtrsim 1}}{|\lambda-l|^\mu+\rho^\mu}\right)^p~d\lambda\\
			&\lesssim  \sum_{n=1}^{2k}n^{d-2}\left(\sum_{l=1}^{k}\frac{1}{l^\mu+n^\mu}\right)^p.
		\end{align*}
		If $\mu>1$, we obtain
		\begin{align*}
			\sum_{n=1}^{2k}n^{d-2}\left(\sum_{l=1}^{k}\frac{1}{l^\mu+n^\mu}\right)^p&\lesssim \sum_{n=1}^{2k}n^{d-2}n^{(1-\mu)p}\sim k^{d+(1-\mu)p-1}.
		\end{align*}
		If $\mu = 1$, we obtain
		\begin{align*}
			\sum_{n=1}^{2k}n^{d-2}\left(\sum_{l=1}^{k}\frac{1}{l+n}\right)^p&\lesssim \sum_{n=1}^{2k}n^{d-2}(1+\log (2k/n))\\
			&\sim k^{d-1}.
		\end{align*}		
		Thus
		\begin{equation}\label{tailmainesti}
			\left\|\sum_{l=0}^{k-1}|x-l\epsilon_k^{\alpha-1}\mathfrak{v}|^{-\mu}\right\|_{L^p(\Omega)}^p\lesssim \epsilon_k^{(d-\mu p)(\alpha-1)-\alpha(d+(1-\mu)p)}\sim \epsilon_k^{-\alpha}.
		\end{equation}
		By \eqref{thetailterm} and \eqref{tailmainesti}, one has
		$$J\lesssim \epsilon_k^{\delta p+\alpha} |E_k^{j-1}|^p\epsilon_k^{-\alpha} \sim \epsilon_k^{\delta jp}.$$
		We now estimate $I_m$. For $x\in B_m$, we have
		\begin{align*}
			\sum_{l=0}^{k-1}f_{k,\mu}^{j-1}(x-l\epsilon_k^{\alpha-1}\mathfrak{v})&\leq f_{k,\mu}^{j-1}(x-m\epsilon_k^{\alpha-1}\mathfrak{v})+\sum_{l\neq m}2^\mu |E_k^{j-1}||x-l\epsilon_k^{\alpha-1}\mathfrak{v}|^{-\mu}\\
			&\leq f_{k,\mu}^{j-1}(x-m\epsilon_k^{\alpha-1}\mathfrak{v})+c_\alpha\epsilon_k^{(j-1)\delta }\sum_{l=1}^kl^{-\mu}\epsilon_k^{-\mu(\alpha-1)}\\
			&:= f_{k,\mu}^{j-1}(x-m\epsilon_k^{\alpha-1}\mathfrak{v})+\epsilon_k^{(j-1)\delta }A.
		\end{align*}
		Here $A$ satisfies
		\begin{equation*}
			A\sim \left\{
			\begin{aligned}
				&\epsilon_k^{\delta/p}, &0<\mu<1;\\
				&\epsilon_k^{(1-\alpha)}\log (1+\epsilon_k^{-1}), & \mu = 1;\\
				&\epsilon_k^{\mu(1-\alpha)},& \mu>1.
			\end{aligned}
			\right.
		\end{equation*}
		Then by \eqref{decom}, $|a+b|^p\leq |a|^p+C_p|a|^{p-\varepsilon}|b|^{\varepsilon}+C_p|b|^p$ for any $0\leq \varepsilon\leq 1$, and $\epsilon_k^{\delta p+(j-1)\delta p}A^p |B_0|\lesssim \epsilon_k^{\delta j p}$, we obtain
		\begin{align*}
			&\quad\|f_{k,\mu}^j\|_{L^p(\mathbb{R}^d)}^p\\
			& \leq \epsilon_k^{\delta p+\alpha}k\|f_{k,\mu}^{j-1}(x)+\epsilon_k^{(j-1)\delta }A\|_{L^p(B_0)}^p+C\epsilon_k^{\delta jp}\\
			& \leq\epsilon_k^{\delta p}\|f_{k,\mu}^{j-1}(x)\|_{L^p(B_0)}^p+C_p\epsilon_k^{\delta p}(A\epsilon_k^{ (j-1)\delta})^\varepsilon\int_{B_0}|f_{k,\mu}^{j-1}(x)|^{p-\varepsilon}~dx+C'\epsilon_k^{\delta jp}.
		\end{align*}
		Choose $0<\varepsilon\ll 1$ such that $(p-\varepsilon)\mu>d$. Then by \eqref{largex}, one has
		\begin{align*}
			&\quad \epsilon_k^{\delta p}(A\epsilon_k^{ (j-1)\delta})^\varepsilon\int_{B_0\setminus B(0,2)}|f_{k,\mu}^{j-1}(x)|^{p-\varepsilon}~dx\\
			&\sim \epsilon_k^{\delta p}(A\epsilon_k^{ (j-1)\delta})^\varepsilon\int_2^{\epsilon_k^{\alpha-1}}\rho^{-\mu (p-\varepsilon)}|E_k^{j-1}|^{p-\varepsilon}\rho^{d-1}~d\rho\\
			&\lesssim \epsilon_k^{\delta jp} A^\varepsilon\lesssim \epsilon_k^{\delta jp}.
		\end{align*}
		By the H\"{o}lder inequality, we obtain
		\begin{align*}
			&\quad\|f_{k,\mu}^j\|_{L^p(\mathbb{R}^d)}^p \\
			&\leq\epsilon_k^{\delta p}\|f_{k,\mu}^{j-1}(x)\|_{L^p(B_0)}^p+C\epsilon_k^{\delta p}(A\epsilon_k^{ (j-1)\delta})^\varepsilon\|f_{k,\mu}^{j-1}(x)\|_{L^p(B(0,2))}^{p-\varepsilon}+C''\epsilon_k^{\delta jp}.
		\end{align*}
		Let $a_{k,j} = \|f_{k,\mu}^j\|_{L^p(\mathbb{R}^d)}^p/\epsilon_k^{\delta jp}$, then
		\begin{equation}\label{inductione}
			a_{k,j}\leq a_{k,j-1}+CA^\varepsilon a_{k,j-1}^{(p-\varepsilon)/p}+C''.
		\end{equation}
		By \eqref{largex} and \eqref{inftystart}, we have
		\begin{equation}\label{inductionfirst}
			\begin{aligned}
				a_{k,1} &= \epsilon_k^{-\delta p}\|f_{k,\mu}^1\|_{L^p(\mathbb{R}^d)}^p\\
				&= \epsilon_k^{-\delta p}(\|f_{k,\mu}^1(x)\|_{L^p_{|x|\geq 2}}^p+\|f_{k,\mu}^1(x)\|_{L^p_{|x|\leq 2}}^p)\\
				& \lesssim \epsilon_k^{-\delta p}\left(\int_{|x|\geq 2}|x|^{-\mu p}|E_k^1|^p~dx+\|f_{k,\mu}^1(x)\|_{L^\infty(\mathbb{R}^d)}^p\right)\\
				&\lesssim 1.
			\end{aligned}
		\end{equation}
		Let $b_j = \limsup_{k\rightarrow\infty}a_{k,j}$. By \eqref{inductione} and \eqref{inductionfirst}, we obtain 
		$b_j\leq b_{j-1}+C$, $b_1\lesssim 1$. Thus $b_j\lesssim j$.
		
		By \eqref{thetailterm}--\eqref{tailterm}, we also have (without loss of generality $\tilde{c} = \tilde{c}(d,\delta ,p)\ll 1$)
		\begin{align*}
			J&\geq c\epsilon_k^{\delta p+\alpha} |E_k^{j-1}|\epsilon_k^{(d-\mu p)(\alpha-1)}\int_{\mathbb{R}^d}\frac{k^p}{|y|^{\mu p}}\chi_{|y|\geq 2k}~dy\\
			&\geq \tilde{c} \epsilon_k^{\delta p+\alpha+(j-1)\delta p+(d-\mu p)(\alpha-1)-\alpha(p-\mu p+d)} = \tilde{c} \epsilon_k^{j\delta p}.
		\end{align*}
		At the same time,
		\begin{align*}
			I_m\geq \|f_{k,\mu}^{j-1}(x-m\epsilon_k^{\alpha-1}\mathfrak{v})\|_{L^p(B_m)}^p &= \|f_{k,\mu}^{j-1}(x)\|_{L^p(B_0)}^p\\
			& = \|f_{k,\mu}^{j-1}\|_{L^p(\mathbb{R}^d)}^p-\|f_{k,\mu}^{j-1}\|_{L^p(B_0^c)}^p.
		\end{align*}
		Thus
		\begin{align*}
			\|f_{k,\mu}^{j-1}\|_{L^p(B_0^c)}^p\sim \int_{\epsilon_k^{\alpha-1}}^\infty \left(\frac{|E_k^{j-1}|}{\rho^{\mu}}\right)^{p}\rho^{d-1}~d\rho \sim \epsilon_k^{(j-1)\delta p+(\alpha-1)(d-\mu p)}.
		\end{align*}
		Note that $(\alpha-1)(d-\mu p) = (1-\alpha)\alpha p/p'>0$. Thus
		\begin{align*}
			\|f_{k,\mu}^j\|_{L^p(\mathbb{R}^d)}^p &\geq \epsilon_k^{\delta p+\alpha} k(\|f_{k,\mu}^{j-1}\|_{L^p(\mathbb{R}^d)}^p-\|f_{k,\mu}^{j-1}\|_{L^p(B_0^c)}^p)+\tilde{c}\epsilon_k^{j\delta p}\\
			&\geq \epsilon_k^{\delta p}\|f_{k,\mu}^{j-1}\|_{L^p(\mathbb{R}^d)}^p+\tilde{c}\epsilon_k^{j\delta p}-\tilde{C}\epsilon_k^{j\delta p+(1-\alpha)\alpha p/p'}.
		\end{align*}
		Let $c_j = \liminf_{k\rightarrow\infty}a_{k,j}$. Then $c_j\geq c_{j-1}+\tilde{c}$. By \eqref{largex} and \eqref{inductionfirst}, we also have $a_{k,1}\gtrsim 1$. Thus $c_1\gtrsim 1$, and by induction we obtain $c_j\gtrsim j$, thereby finishing the proof.
	\end{proof}
	
	\begin{prop}\label{counterp1}
		Let $1 = p_1<q<p_2<\infty$, $0<s_1-s_2\leq d/p_2'$, $\theta_1 = 1-p_2'/q'$. Then $C(d,s_1,p_1,s_2,p_2,\theta_1,q)=\infty$.
	\end{prop}
	\begin{proof}[\textbf{Proof}]
		Let $\delta = (s_1-s_2)p_2'$. Then $s_1-s = \theta_2(s_1-s_2) = \delta/q'$. If $d-1<\delta<d$, then for any $0<\epsilon<1$, we define
		\begin{equation*}
			u_{k,\epsilon}^j = \mathscr{F}^{-1}(\varphi(\epsilon|\xi|)-\varphi(|\xi|/\epsilon))\mathscr{F}u_k^j.
		\end{equation*}
		For any $\alpha\in \mathbb{R}$, $1\leq p\leq \infty$, one has $\|D^\alpha u_{k,\epsilon}^j\|_{L^p(\mathbb{R}^d)}\lesssim \|D^\alpha u_k^j\|_{L^p(\mathbb{R}^d)}$ by the Young inequality. Due to $1<q<\infty$, we also have $\lim_{\epsilon\rightarrow 0+}\|D^\alpha u_{k,\epsilon}^j\|_{L^q(\mathbb{R}^d)}\sim \|D^\alpha u_k^j\|_{L^q(\mathbb{R}^d)}$ for any $\alpha\in \mathbb{R}$. Then by Lemma \ref{thekeylemma}, we obtain
		\begin{align*}
			C(d,s_1,p_1,s_2,p_2,\theta_1,q)&\geq\sup_k\sup_{0<\epsilon<1} \frac{\|D^{s-s_1}u_{k,\epsilon}^j\|_{L^q(\mathbb{R}^d)}}{\|u_{k,\epsilon}^j\|_{L^1(\mathbb{R}^d)}^{1-p_2'/q'}\|D^{s_2-s_1}u_{k,\epsilon}^j\|_{L^{p_2}(\mathbb{R}^d)}^{p_2'/q'}}\\
			&\gtrsim\sup_k \frac{\|D^{-\delta/q'}u_{k}^j\|_{L^q(\mathbb{R}^d)}}{\|u_{k}^j\|_{L^1(\mathbb{R}^d)}^{1-p_2'/q'}\|D^{-\delta/p_2'}u_{k}^j\|_{L^{p_2}(\mathbb{R}^d)}^{p_2'/q'}}\\
			&\gtrsim \frac{j^{1/q}}{1\cdot (j^{1/p_2})^{p_2'/q'}}\\
			&\sim j^{p_2'(\frac{1}{q}-\frac{1}{p_2})}\rightarrow \infty,~\mbox{ as }j\rightarrow\infty.
		\end{align*}
		By Proposition \ref{scalingcritical}, the conclusion holds for $s_1-s_2 = d/p_2'$. Combining with Lemma \ref{dimensionreduction}, we finish the proof.
	\end{proof}
	
	The same argument as in Proposition \ref{counterp1} yields the following conclusion.
	\begin{prop}\label{inftysubcounter}
		Let $p_1 = 1$, $p_2 = \infty$, $1<q<2$, $0<s_1-s_2<d$, $\theta_1 = 1/q$. Then $C(d,s_1,p_1,s_2,p_2,\theta_1,q)=\infty$.
	\end{prop}
	Combining Propositions \ref{scalingcritical}, \ref{counterp1}, and \ref{inftysubcounter}, we finish the proof of the “if” part of Theorem \ref{firstmain}.
	
	\section{Anisotropic Gagliardo--Nirenberg inequalities}
	If $\theta_{j_0} s_{j_0} = 0$ for some $1\leq j_0\leq n$, then by \eqref{scalingar} we have $\theta_j s_j = 0$ for any $1\leq j\leq n$. Now \eqref{aniso} is equivalent to
	\begin{equation*}
		\|u\|_{L^q(\mathbb{R}^d)}\lesssim \|u\|_{L^{p_0}(\mathbb{R}^d)}^{\theta_0}\prod_{j=1}^n\|u\|_{L^{p_j}(\mathbb{R}^d)}^{\theta_j}.
	\end{equation*}
	By the H\"{o}lder inequality, we conclude the proof. Below we would always assume $\theta_js_j>0$ for any $1\leq j\leq n$. 
	
	We define $s>0$, $1\leq p\leq \infty$, and $-\infty<r\leq \infty$ by the relations
	\begin{equation}\label{avreandin}
		\frac{d}{s} = \sum_{j=1}^n \frac{d_j}{s_j},\quad \frac{d}{p} = \sum_{j=1}^d\frac{sd_j}{s_jp_j},\quad \frac{d}{r} = \frac{d}{p}-s.
	\end{equation}
	A similar definition appears in Chapter 10 of Triebel \cite{triebel1983theory}.
	
	By \eqref{scalingar}, one has $p_0\neq r$ and
	\begin{equation}\label{thetarel}
		\theta_0 = \frac{1/q-1/r}{1/p_0-1/r},\quad \theta_j\frac{s_j}{d_j} = (1-\theta_0)\frac{s}{d},~1\leq j\leq n.
	\end{equation}
	For $k\in \mathbb{Z}^n$, define the Littlewood--Paley projection operators $P_k$ by
	$$P_k = \mathscr{F}^{-1}\left(\prod_{j=1}^n \psi(|\xi_j|/2^{k_j})\right)\mathscr{F},\quad k = (k_1,\cdots,k_n).$$
	
	\subsection{Proof of anisotropic Gagliardo--Nirenberg inequalities}\label{proofofagn}
	In this section, we show Theorem \ref{mainani}.
	\begin{prop}\label{subbern}
		Assume  \eqref{anibasicred}--\eqref{scalingar} and $\theta_0>0$. Then \eqref{aniso} holds.
	\end{prop}
	\begin{proof}[\textbf{Proof}]
		Let $s_0 = 0$, $k = (k_1,\cdots,k_n)$, and $k_0 = 0$ by convention. By scaling, we may assume that
		$$\|u\|_{L^{p_0}} = \|D_{x_1}^{s_1}u\|_{L^{p_1}}=\cdots = \|D^{s_d}_{x_d}u\|_{L^{p_d}} = 1.$$
		Let
		$$\frac{1}{\tilde{p}} = \sum_{j=0}^n\frac{\theta_j}{p_j} = \frac{\theta_0}{p_0}+\frac{1-\theta_0}{p},\quad \frac{1}{\tilde{p}}-\frac{1}{q} = \frac{\theta_j s_j}{d_j} = (1-\theta_0)\frac{s}{d}.$$
		Applying Bernstein's and Hölder's inequalities, for $0\leq \epsilon\ll 1$, we have
		\begin{align*}
			\|P_ku\|_{L^q}&\lesssim 2^{((1-\epsilon)(1/\tilde{p}-1/q)+\epsilon (1/p_l-1/q))\sum_{j=1}^nk_jd_j} \|P_ku\|_{L^{\tilde{p}}}^{1-\epsilon}\|P_ku\|_{L^{p_l}}^{\epsilon}\\
			&\lesssim 2^{-\epsilon k_ls_l+((1-\epsilon)(1/\tilde{p}-1/q)+\epsilon (1/p_l-1/q))\sum_{j=1}^nk_jd_j} \prod_{j=0}^n \|P_ku\|_{L^{p_j}}^{\theta_j(1-\epsilon)}\\
			&\lesssim 2^{-\epsilon k_ls_l-\sum_{j=1}^n(1-\epsilon)\theta_jk_js_j+((1-\epsilon)(1/\tilde{p}-1/q)+\epsilon (1/p_l-1/q))\sum_{j=1}^nk_jd_j}\\
			&\sim 2^{\epsilon(-k_ls_l+ (1/p_l-1/q)\sum_{j=1}^nk_jd_j)}.
		\end{align*}
		Then by the triangle inequality, we obtain
		\begin{align*}
			\|u\|_{L^q}\leq \sum_{k\in \mathbb{Z}^n}\|P_ku\|_{L^q} \lesssim \sum_{k\in \mathbb{Z}^n}\min_{0\leq l\leq n}\{2^{\epsilon(-s_lk_l+(\frac{1}{p_l}-\frac{1}{q})\sum_{j=1}^dk_jd_j)}\}.
		\end{align*}
		Note that $\alpha:=1/p_0-1/q \neq 0$ and
		$$ \sum_{l=1}^d\frac{d_ls}{ds_l}\left(-s_lk_l+(\frac{1}{p_l}-\frac{1}{q})\sum_{j=1}^nk_jd_j\right) = (\frac{1}{r}-\frac{1}{q})\sum_{l=1}^dk_ld_l:= -\beta\sum_{l=1}^dk_ld_l,$$
		where $\beta = 1/q-1/r\neq 0$. Note that $\theta_0 = \beta/\alpha>0$. We can choose $\gamma_{l,1}>0$, $\gamma_{l,2},\gamma_{l,3}\geq 0$ such that $\gamma_{l,1}+\gamma_{l,2}+\gamma_{l,3} = 1$ and $\gamma_{l,1}(1/p_l-1/q)+\gamma_{l,2}\alpha -\gamma_{l,3}\beta = 0$. Then
		\begin{align*}
			&\quad\gamma_{l,1}\left(-s_lk_l+(\frac{1}{p_l}-\frac{1}{q})\sum_{j=1}^nk_jd_j\right)+\gamma_{l,2}\alpha\sum_{j=1}^dk_jd_j-\gamma_{l,3}\beta\sum_{j=1}^nk_jd_j\\
			& = -\gamma_{l,1}s_lk_l:=-\alpha_l k_j,
		\end{align*}
		where $\alpha_l = \gamma_{l,1} s_l > 0$.
		\begin{align*}
			\|u\|_{L^q}&\lesssim \sum_{k\in \mathbb{Z}^n}\min\{2^{\epsilon\alpha\sum_{j=1}^n k_jd_j},2^{-\epsilon\beta\sum_{j=1}^n k_jd_j},2^{-\epsilon\alpha_1k_1},\cdots,2^{-\epsilon\alpha_d k_d}\}\\
			&\lesssim 1.
		\end{align*}
		We finish the proof.
	\end{proof}
	
	For the critical case, we use the Littlewood--Paley theory.
	
	\begin{prop}\label{crilit}
		Assume  \eqref{anibasicred}--\eqref{scalingar}, $\theta_0=0$, $1<q<\infty$, and $p_j>1$ for any $1\leq j\leq n$.
		Then \eqref{aniso} holds.
	\end{prop}
	\begin{proof}[\textbf{Proof}]
		Without loss of generality, we would assume $1<p_1\leq \cdots\leq p_n$. By \eqref{scalingar}, we know that $p_1\leq q<\infty$. The case $n = 1$ follows from Theorem \ref{firstmain}. Assume the result holds for $n-1$. If $p_n\geq q$, write $x = (x',x_n)\in \mathbb{R}^{d-d_n}\times \mathbb{R}^{d_n}$ and set
		$$\frac{1}{\gamma} = \sum_{l=1}^{n-1}\frac{\theta_l}{(1-\theta_n)p_l}\in (0,1).$$
		Then for any $u\in \mathcal{S}(\mathbb{R}^d)$, combining Proposition \ref{littlewood}, Hölder's inequality, Minkowski's inequality, and the induction hypothesis yields
		\begin{align*}
			\|u\|_{L^q}&\lesssim \|\|u\|^{1-\theta_n}_{L^{\gamma}_{x_d}}\|D_{x_n}^{s_n}u\|^{\theta_n}_{L^{p_n}_{x_n}}\|\|_{L^q_{x'}}\\
			&\lesssim \|u\|^{1-\theta_n}_{L^{\tilde{q}}_{x'}L_{x_d}^\gamma}\|D_{x_n}^{s_n}u\|^{\theta_n}_{L^{p_n}}\\
			&\lesssim \|u\|^{1-\theta_n}_{L_{x_n}^\gamma L^{\tilde{q}}_{x'} }\|D_{x_n}^{s_n}u\|^{\theta_n}_{L^{p_n}}\\
			&\lesssim \left\|\prod_{j=1}^{n-1}\|D^{s_j}_{x_j}u\|_{L^{p_j}_{x'}}^{\frac{\theta_j}{1-\theta_n}}\right\|^{1-\theta_n}_{L_{x_d}^\gamma}\|D_{x_n}^{s_n}u\|^{\theta_n}_{L^{p_n}}\\
			&\lesssim\prod_{j=1}^n\|D_{x_j}^{s_j}u\|_{L^{p_j}}^{\theta_j}.
		\end{align*}
		If $p_n<q$, then by H\"{o}lder's inequality, the Littlewood--Paley theory (see \cite{stein1970singular} or Lemma \ref{littlewoodpaley} below), Minkowski's inequality, and Sobolev's inequality, we obtain
		\begin{align*}
			\|u\|_{L^q}\sim \||P_ku|_{l^2_{k}}\|_{L^q}
			&\lesssim \left\|\prod_{j=1}^d2^{k_j\frac{sd_j}{d}-(\frac{\theta_j}{p_j}-\frac{\theta_j}{q})\sum_{l=1}^nk_ld_l}|P_ku|^{\theta_j}_{l^2_{k}}\right\|_{L^q}\\
			&\lesssim \prod_{j=1}^n \||2^{s_j k_j- (\frac{1}{p_j}-\frac{1}{q})\sum_{l=1}^nk_ld_l}P_{k} u|_{l^2_k}\|_{L^q}^{\theta_j}\\
			&\lesssim \prod_{j=1}^{n} \|D_{x_j}^{s_j}D_{x_1}^{\frac{d_1}{q}-\frac{d_1}{p_j}}\cdots D_{x_d}^{\frac{d_n}{q}-\frac{d_n}{p_j}}u\|_{L^{q}}^{\theta_j}\\
			&\lesssim \prod_{j=1}^d \|D_{x_j}^{s_j}u\|_{L^{p_j}}^{\theta_j}.
		\end{align*}
		To obtain the last inequality, we invoke the Sobolev inequality $n$ times and the Minkowski inequality $n-1$ times.
	\end{proof}
	
	\begin{prop}\label{controlinfi}
		Assume \eqref{anibasicred}--\eqref{scalingar}, $\theta_0=0$, $1 = p_1\leq \cdots\leq p_n$, $1<q<\infty$, and $(p_2,p_n)\neq (q,q)$. Then \eqref{aniso} holds.
	\end{prop}
	\begin{proof}[\textbf{Proof}]
		We first treat the case $n = 2$.
		
		If $1\leq p_2<q$, then \eqref{scalingar} implies $s_1>d_1/q'$. By Proposition \ref{coheninter} and the H\"{o}lder inequality, we obtain
		\begin{align*}
			\|D^{s_1/2}_{x_1}u\|_{L^{\tilde{p}_1}}&\lesssim \|\|D^{s_1}_{x_1}u\|_{L_{x_1}^1}^{1/2}\|u\|_{L^{q}_{x_1}}^{1/2}\|_{L^{\tilde{p}}_{x_2}},\quad \frac{1}{\tilde{p}_1} = \frac{1}{2}+\frac{1}{2q}\\
			&\lesssim \|D^{s_1}_{x_1}u\|_{L^1}^{1/2}\|u\|_{L^{q}}^{1/2}.
		\end{align*}
		Similarly we have
		$\|D^{s_2/2}_{x_2}u\|_{L^{\tilde{p}_2}}\lesssim\|D^{s_2}_{x_2}u\|_{L^{p_2}}^{1/2}\|u\|_{L^{q}}^{1/2}$, $1/\tilde{p}_2 = 1/(2p_2)+1/(2q)$. By Proposition \ref{crilit}, we have
		\begin{align*}
			\|u\|_{L^q}\lesssim \|D^{s_1/2}_{x_1}u\|_{L^{\tilde{p}_1}}^{\theta_1}\|D^{s_2/2}_{x_2}u\|_{L^{\tilde{p}_2}}^{\theta_2} \lesssim \|D^{s_1}_{x_1}u\|_{L^1}^{\theta_1/2}\|D^{s_2}_{x_2}u\|_{L^{p_2}}^{\theta_2/2}\|u\|_{L^{q}}^{1/2}.
		\end{align*}
		Thus we obtain \eqref{aniso}.
		
		If $q<p_2\leq \infty$, then for $0<\epsilon\ll 1$ by Proposition \ref{littlewood} and the H\"{o}lder inequality, we have
		\begin{equation}\label{midin}
			\|D^{s_2\epsilon}_{x_2}u\|_{L^{\tilde{p}_2}}\lesssim\|D^{s_2}_{x_2}u\|_{L^{p_2}}^{\epsilon}\|u\|_{L^{q}}^{1-\epsilon},\quad \frac{1}{\tilde{p}_2} = \frac{\epsilon}{p_2}+\frac{1-\epsilon}{q}.
		\end{equation}
		Let $\bar{p}_2 = 1/(1/\tilde{p}_2-s_2\epsilon/d_2)$, $\tilde{p}_1 = 1/(1/2+1/(2\bar{p}_2))$, $\alpha = s_1/2-d_1(1/2+1/(2\tilde{p}_2)-1/\tilde{p}_1)$. By the Sobolev inequality, we have $\|u\|_{L^{\bar{p}_2}_{x_2}}\lesssim \|D^{s_2\epsilon}_{x_2}u\|_{L^{\tilde{p}_2}}$. By Proposition \ref{subcritical}, the H\"{o}lder inequality, and the Minkowski inequality, we obtain
		\begin{equation}\label{x1mid}
			\begin{aligned}
				\|D_{x_1}^{\alpha}u\|_{L^{\tilde{p}_1}}& \lesssim \|\|D_{x_1}^{s_1}u\|_{L^1_{x_1}}^{1/2}\|u\|_{L^{\tilde{p}_2}_{x_1}}^{1/2}\|_{L^{\tilde{p}_1}_{x_2}}\\
				&\lesssim \|D_{x_1}^{s_1}u\|_{L^1}^{1/2}\|u\|_{L^{\bar{p}_2}_{x_2}L^{\tilde{p}_2}_{x_1}}^{1/2}\\
				&\lesssim \|D_{x_1}^{s_1}u\|_{L^1}^{1/2}\|u\|_{L^{\tilde{p}_2}_{x_1}L^{\bar{p}_2}_{x_2}}^{1/2}\\
				&\lesssim \|D_{x_1}^{s_1}u\|_{L^1}^{1/2}\|D_{x_2}^{s_2\epsilon}u\|_{L^{\tilde{p}_2}}^{1/2}.
			\end{aligned}
		\end{equation}
		By Proposition \ref{crilit} and \eqref{midin}--\eqref{x1mid}, for $\theta = 1/(1+d_1s_2/(d_2\alpha))$, we have
		\begin{align*}
			\|u\|_{L^q}&\lesssim \|D_{x_1}^\alpha u\|_{L^{\tilde{p}_1}}^{1-\theta}\|D_{x_2}^{s_2} u\|_{L^{p_2}}^{\theta}\\
			&\lesssim \|D_{x_1}^{s_1}u\|_{L^1}^{(1-\theta)/2}\|D_{x_2}^{s_2\epsilon}u\|_{L^{\tilde{p}_2}}^{(1-\theta)/2} \|D_{x_2}^{s_2} u\|^{\theta}_{L^{p_2}}\\
			&\lesssim \|D_{x_1}^{s_1}u\|_{L^1}^{(1-\theta)/2} \|D^{s_2}_{x_2}u\|_{L^{p_2}}^{(1-\theta)\epsilon/2+\theta}\|u\|_{L^{q}}^{(1-\epsilon)(1-\theta)/2}.
		\end{align*}
		This completes the proof for $n = 2$.
		
		For $n\geq 3$, if $p_n\leq q$, then since $(p_1,p_n)\neq (q,q)$, we have $1 = p_1\leq p_2<q$. Consequently, by \eqref{scalingar}, for any $1\leq j\leq n$, we obtain
		\begin{equation*}
			\frac{s_j}{d_j}\theta_j = \sum_{l=1}^d \frac{\theta_l}{p_l}-\frac{1}{q}>(\frac{1}{p_j}-\frac{1}{q})\theta_j.
		\end{equation*}
		Thus $s_j>(1/p_j-1/q)d_j$. For any $1\leq j\leq n$, Propositions \ref{littlewood} and \ref{coheninter} combined with H\"{o}lder's inequality yield
		\begin{equation}\label{inter}
			\|D^{s_j/2}_{x_j}u\|_{L^{\tilde{p}_j}}\lesssim\|D^{s_j}_{x_j}u\|_{L^{p_2}}^{1/2}\|u\|_{L^{q}}^{1/2},\quad \frac{1}{\tilde{p}_j} = \frac{1}{2p_j}+\frac{1}{2q}.
		\end{equation}
		By Proposition \ref{crilit}, we have $\|u\|_{L^q}\lesssim \prod_{j=1}^n \|D^{s_j/2}_{x_j}u\|_{L^{\tilde{p}_j}}^{\theta_j}$. Combining with \eqref{inter}, we conclude the proof.
		
		If $p_n > q$, let $x = (x_1,\cdots,x_{n-1},x_n):= (x',x_n)\in \mathbb{R}^{d-d_n}\times \mathbb{R}^{d_n}$ and define $\gamma$, $\tilde{q}$ by
		$$\frac{1}{\gamma} = \sum_{l=1}^{n-1}\frac{\theta_l}{(1-\theta_n)p_l}\in (0,1),\quad \frac{1}{q}= \frac{1-\theta_n}{\tilde{q}}+\frac{\theta_n}{p_n}.$$
		In the case $(p_2,p_{n-1})\neq  (\tilde{q},\tilde{q})$, applying Proposition \ref{littlewood}, Hölder's inequality, Minkowski's inequality, and the induction hypothesis yields
		\begin{align*}
			\|u\|_{L^q}\lesssim \|\|u\|^{1-\theta_n}_{L^{\gamma}_{x_n}}\|D_{x_n}^{s_n}u\|^{\theta_n}_{L^{p_n}_{x_n}}\|\|_{L^q_{x'}}
			&\lesssim \|u\|^{1-\theta_n}_{L^{\tilde{q}}_{x'}L_{x_n}^\gamma}\|D_{x_n}^{s_n}u\|^{\theta_n}_{L^{p_n}}\\
			&\lesssim \|u\|^{1-\theta_n}_{L_{x_n}^\gamma L^{\tilde{q}}_{x'} }\|D_{x_n}^{s_n}u\|^{\theta_n}_{L^{p_n}}\\
			&\lesssim \left\|\prod_{j=1}^{n-1}\|D^{s_j}_{x_j}u\|_{L^{p_j}_{x'}}^{\frac{\theta_j}{1-\theta_n}}\right\|^{1-\theta_n}_{L_{x_n}^\gamma}\|D_{x_n}^{s_n}u\|^{\theta_n}_{L^{p_n}}\\
			&\lesssim\prod_{j=1}^n\|D_{x_j}^{s_j}u\|_{L^{p_j}}^{\theta_j}.
		\end{align*}	
		If $(p_2,p_{n-1}) = (\tilde{q},\tilde{q})$, then $1/q>\theta_{n-1}/p_{n-1}$. Set $x'' = (x_1,\cdots,x_{n-2},x_n)$, and define $\tilde{\gamma}$, $r$ by
		$$\frac{1}{\tilde{\gamma}} = \frac{\theta_n}{(1-\theta_{n-1})p_n}+ \sum_{l=1}^{n-2}\frac{\theta_l}{(1-\theta_{n-1})p_l}\in (0,1),\quad \frac{1}{q}= \frac{1-\theta_{n-1}}{r}+\frac{\theta_{n-1}}{p_{n-1}}.$$
		Note that $p_2 = \tilde{q}\neq p_n$. Thus as in the previous argument, we have
		\begin{align*}
			\|u\|_{L^q}&\lesssim \|\|u\|^{1-\theta_{n-1}}_{L^{\tilde{\gamma}}_{x_{n-1}}}\|D_{x_{n-1}}^{s_{n-1}}u\|^{\theta_{n-1}}_{L^{p_{n-1}}_{x_{n-1}}}\|\|_{L^q_{x''}}\\
			&\lesssim \|u\|^{1-\theta_{n-1}}_{L^{r}_{x''}L_{x_{n-1}}^{\tilde{\gamma}}}\|D_{x_{n-1}}^{s_{n-1}}u\|^{\theta_{n-1}}_{L^{p_{n-1}}}\\
			&\lesssim \|u\|^{1-\theta_{n-1}}_{L_{x_{n-1}}^{\tilde{\gamma}} L^{r}_{x''} }\|D_{x_{n-1}}^{s_{n-1}}u\|^{\theta_{n-1}}_{L^{p_{n-1}}}\\
			&\lesssim \left\|\|D^{s_n}_{x_n}u\|_{L_{x''}^{p_n}}^{\frac{\theta_n}{1-\theta_{n-1}}}\prod_{j=1}^{n-2}\|D^{s_j}_{x_j}u\|_{L^{p_j}_{x''}}^{\frac{\theta_j}{1-\theta_{n-1}}}\right\|^{1-\theta_{n-1}}_{L_{x_{n-1}}^{\tilde{\gamma}}}\|D_{x_{n-1}}^{s_{n-1}}u\|^{\theta_{n-1}}_{L^{p_{n-1}}}\\
			&\lesssim\prod_{j=1}^n\|D_{x_j}^{s_j}u\|_{L^{p_j}}^{\theta_j}.
		\end{align*}
		We finish the proof.
	\end{proof}
		
	\subsection{Some counterexamples for \texorpdfstring{$\theta_0 = 0, q = \infty$}{theta0=0,q=infty}}\label{counter}
	In this section, we consider\eqref{aniso} with $\theta_0 = 0$, $q = \infty$. Assume \eqref{anibasicred}--\eqref{scalingar}, $(\theta_0,q) = (0,\infty)$, and $1<p_{j_0}<\infty$, $\theta_{j_0}>0$ for some $1\leq j\leq n$. For any $L\in \mathbb{N}$, we define
	\begin{equation*}
		u_L = \mathscr{F}^{-1}\left(\sum_{l=1}^L2^{-l}\prod_{j=1}^n \psi(|\xi_j|/2^{l/s_jp_j})\right).
	\end{equation*}
	\begin{lemma}\label{lpesti}
		$\|u_L\|_{L^\infty(\mathbb{R}^d)}\sim L$, $\|D^{s_j}_{x_j}u_L\|_{L^{p_j}(\mathbb{R}^d)}\lesssim L^{1/2}$ for $2\leq p_j<\infty$, and $\|D^{s_j}_{x_j}u_L\|_{L^{p_j}(\mathbb{R}^d)}\lesssim L^{1/p_j}$ for $1\leq p_j<2$.
	\end{lemma}
	\begin{proof}[\textbf{Proof}]
		Note that $\mathscr{F}u_L\geq 0$. Thus
		\begin{align*}
			\|u_L\|_{L^\infty(\mathbb{R}^d)} \sim \|\mathscr{F}u_L\|_{L^1(\mathbb{R}^d)} \sim \sum_{l=0}^L 2^{-l}\prod_{j=1}^n 2^{l\frac{d_j}{s_jp_j}} \sim L.
		\end{align*}
		If $2\leq p_j<\infty$, by the Littlewood--Paley theory, the Minkowski inequality, and the Berstein inequality, we have
		\begin{align*}
			\|D^{s_j}_{x_j}u_L\|_{L^{p_j}(\mathbb{R}^d)}
			&\sim \||P_kD^{s_j}_{x_j}u_L|_{l^2_{k\in \mathbb{Z}^n}}\|_{L^{p_j}(\mathbb{R}^d)}\\
			&\lesssim |\|P_kD^{s_j}_{x_j}u_L\|_{L^{p_j}(\mathbb{R}^d)}|_{l^2_k}\\
			&\lesssim \left|\sum_{l=1}^L2^{-l}2^{l\frac{s_j}{s_jp_j}}\prod_{m=1}^n\chi_{2^{k_m}\sim 2^{l/s_mp_m}}2^{l\frac{d_m}{s_mp_mp_j'}}\right|_{l^2_k}\\
			&\lesssim \left|\sum_{l=1}^L\prod_{m=1}^n\chi_{2^{k_m}\sim 2^{l/s_mp_m}}\right|_{l^2_k}\lesssim L^\frac{1}{2}.
		\end{align*}
		If $1\leq  p_j<2$, by the Littlewood--Paley theory and $l^{p_j}\hookrightarrow l^2$, we have
		\begin{align*}
			\|D^{s_j}_{x_j}u_L\|_{L^{p_j}(\mathbb{R}^d)}\lesssim \||P_kD^{s_j}_{x_j}u_L|_{l^2_{k\in \mathbb{Z}^n}}\|_{L^{p_j}(\mathbb{R}^d)} \lesssim |\|P_kD^{s_j}_{x_j}u_L\|_{L^{p_j}(\mathbb{R}^d)}|_{l^{p_j}_k}.
		\end{align*}
		By the same argument as for the case $2\leq p_j<\infty$, we obtain 
		$$\|D_{x_j}^{s_j}u_L\|_{L^{p_j}(\mathbb{R}^d)}\lesssim L^\frac{1}{p_j},\quad 1\leq p_j<2$$
		and finish the proof.
	\end{proof}
	We proceed to prove Theorem \ref{couninfty}. Without loss of generality, suppose $p_1\leq p_2\leq \cdots \leq p_n$. In the case $p_n = \infty$, choosing $u = f(x_1,\cdots,x_{n-1})g(x_n)$ yields
	$$\|f\|_{L^\infty(\mathbb{R}^{d-d_n})}\lesssim \prod_{j = 1}^{n-1}\|D_{x_j}^{s_j}f\|_{L^{p_j}(\mathbb{R}^{d-d_n})}^{\frac{\theta_j}{1-\theta_n}}.$$
	Repeating this procedure, we can assume $1\leq p_j<\infty$ for every $1\leq j\leq n$. Under assumption \eqref{aniso}, taking $u = u_L$ and using Lemma \ref{lpesti}, we obtain
	\begin{align*}
		L\lesssim \prod_{j=1}^nL^{\frac{\theta_j}{\min\{p_j,2\}}}\lesssim L^{\frac{\theta_{j_0}}{\min\{p_{j_0},2\}}+1-\theta_{j_0}}.
	\end{align*}
	By taking $L\rightarrow\infty$, we conclude the proof of Theorem \ref{couninfty} by contradiction.
	
	\subsection{Anisotropic Gagliardo--Nirenberg inequalities in mixed Lebesgue spaces}\label{animixed}
	Let $n\in \mathbb{N}$ and $d_j\in \mathbb{N}$ for $1\leq j\leq n$. For a vector $\bm{s} = (\bm{s}(1),\cdots,\bm{s}(n))\in \mathbb{R}^n$, we define the Fourier multiplier operator
	$$D^{\bm{s}} = \mathscr{F}^{-1}(|\xi_1|^{\bm{s}(1)}\cdots |\xi_n|^{\bm{s}(n)})\mathscr{F},\quad\xi = (\xi_1,\cdots,\xi_n)\in \mathbb{R}^{d_1}\oplus\cdots\oplus \mathbb{R}^{d_n}.$$
	
	We use the Hadamard (componentwise) product of vectors: for 
	$\bm{x}_1,\bm{x}_2\in \mathbb{R}^n$, $$\bm{x}_1\bm{x}_2 = (\bm{x}_1(1)\bm{x}_2(1),\cdots,\bm{x}_1(n)\bm{x}_2(n)).$$
	For $\bm{p} = (\bm{p}(1),\cdots,\bm{p}(n))\in [1,\infty]^n$, we write $1/\bm{p} = (1/\bm{p}(1),\cdots,1/\bm{p}(n))\in [0,1]^n$. Let $\bm{0} = (0,\cdots,0)$ and $\bm{1} = (1,\cdots,1)$ in  $\mathbb{R}^n$, and define $1/\bm{p}' = \bm{1}-1/\bm{p}$. Vector inequalities are understood componentwise: $\bm{s}_1\leq \bm{s}_2$ if and only if $ \bm{s}_1(j)\leq \bm{s}_2(j)$ for all $j$; the relations $<$, $\geq$, and $>$ are defined analogously.
	
	Let $\bm{d} = (d_1,\cdots,d_n)$ and $d = d_1+\cdots+d_n$. For a measurable function $u$ on $\mathbb{R}^d = \mathbb{R}^{d_1}\oplus\cdots\oplus \mathbb{R}^{d_n}$, the mixed Lebesgue norm is defined by
	$$\|u\|_{L^{\bm{q}}(\mathbb{R}^{\bm{d}})}:=\|\cdots\|\|u(x_1,\cdots,x_n)\|_{ L^{\bm{q}(1)}_{x_1}(\mathbb{R}^{d_1})}\|_{L^{\bm{q}(2)}_{x_2}(\mathbb{R}^{d_2})}\cdots\|_{L^{\bm{q}(n)}_{x_n}(\mathbb{R}^{d_n})}.$$
	Let $M\in \mathbb{N}$ and let $\bm{s}_m\in \mathbb{R}^n$ for $m = 1,\cdots, M$. We consider the anisotropic Gagliardo–Nirenberg inequality with mixed derivatives,
	\begin{equation}\label{anignmix}
		\|D^{\bm{s}}u\|_{L^{\bm{q}}(\mathbb{R}^{\bm{d}})}\lesssim \prod_{m=1}^M \|D^{\bm{s}_m}u\|_{L^{\bm{p}_m}(\mathbb{R}^{\bm{d}})}^{\theta_m}.
	\end{equation}
	Without loss of generality, we may assume the following standing hypotheses:
	\begin{equation}\label{basicassmix}
		\bm{s}_m,\bm{s}\geq \bm{0},~~ \bm{0}\leq \frac{1}{\bm{p}_m},\frac{1}{\bm{q}}\leq \bm{1},~~ \theta_m>0 \mbox{ for any }1\leq m\leq M.
	\end{equation}
	A scaling argument yields the necessary conditions
	\begin{equation}\label{basicconmix}
		\sum_{m=1}^M\theta_m = 1,\quad \bm{d}\left(\sum_{m=1}^M\frac{\theta_m}{\bm{p}_m}-\frac{1}{\bm{q}}\right) = \sum_{m=1}^M\theta_m\bm{s}_m-\bm{s}\geq \bm{0}.
	\end{equation}

	We shall use the following Littlewood–Paley characterization for mixed Lebesgue spaces.	
	\begin{lemma}[\cite{lizorkin1970multipliers}]\label{littlewoodpaley}
		Let $\bm{s}\geq \bm{0}$, $\bm{0}<1/\bm{p}<\bm{1}$. There exists $C_{\bm{d},\bm{q},\bm{s}}<\infty$ such that for any $u\in \mathcal{S}(\mathbb{R}^d)$, we have
		$$C_{\bm{d},\bm{q},\bm{s}}^{-1}\|D^{\bm{s}}u\|_{L^{\bm{q}}(\mathbb{R}^{\bm{d}})}\leq \||2^{k\cdot \bm{s}}P_ku|_{l^2_{k\in \mathbb{Z}^n}}\|_{L^{\bm{q}}(\mathbb{R}^{\bm{d}})} \leq C_{\bm{d},\bm{q},\bm{s}}\|D^{\bm{s}}u\|_{L^{\bm{q}}(\mathbb{R}^{\bm{d}})}$$
	\end{lemma}
	
	\begin{rem}
		Early contributions to the theory of mixed-norm Sobolev spaces and related multiplier results are due to Lizorkin \cite{lizorkin1970multipliers}, Solonnikov \cite{solonnikov1975inequalities}, and Besov–Il'in–Nikol'skii \cite{besov1978integral,besov1979integral}. Additional works on the subject include those of Triebel \cite{triebel1983theory}, Adams \cite{adams1988anisotropic}, and Guliev \cite{guliev}. For more recent progress in mixed Lebesgue spaces, we refer to Huang–Yang \cite{huang2021function} and the references therein.
	\end{rem}
	
	Let $S$ be the convex hull of the points $\{\bm{s}_m-\bm{d}/\bm{p}_m\}_{m=1}^M$, defined explicitly as
	$$S:= \left\{\sum_{m=1}^M\alpha_m\left(\bm{s}_m-\frac{\bm{d}}{\bm{p}_m}\right):\alpha_m\geq 0,\sum_{m=1}^M\alpha_m = 1\right\}.$$
	
	We have the following results:
	\begin{thm}\label{mainresmix}
		Assume \eqref{basicassmix}--\eqref{basicconmix}. Then inequality \eqref{anignmix} is valid whenever at least one of the following conditions is satisfied:
		\begin{itemize}
			\item[\hypertarget{(1)}{(1)}] $\mathrm{dim}(S) = n$ and $1/\bm{q}\leq 1/\bm{p}_m$ for any $1\leq m\leq M$;
			\item[\hypertarget{(2)}{(2)}] $\mathrm{dim}(S) = n$ and $\sum_{m=1}^M\theta_m\bm{s}_m>\bm{s}$;
			\item[\hypertarget{(3)}{(3)}] $\bm{0}<1/\bm{q}<\bm{1}$ and $\bm{0}<1/\bm{p}_m<\bm{1}$ for any $1\leq m\leq M$;
			\item[\hypertarget{(4)}{(4)}] $\bm{0}<1/\bm{q}<\bm{1}$, $\bm{0}\leq 1/\bm{p}_m<\bm{1}$, and $\sum_{m=1}^M\theta_m\bm{s}_m>\bm{s}$.
		\end{itemize}
	\end{thm}
%
%
	
	\begin{rem}
		Assume $d_j = 1$ for $j=1,\cdots,d$, $\bm{0}<1/\bm{p}_1,1/\bm{p}_2,1/\bm{q}<\bm{1}$, $0\leq \theta\leq 1$, $0\leq \sigma<s$, $\theta_1/\bm{p}_1+(1-\theta)/\bm{p}_2\geq 1/\bm{q}$, and $\bm{1}\cdot (\theta_1/\bm{p}_1+(1-\theta)/\bm{p}_2-1/\bm{q}) = (1-\theta)s-\sigma\geq 0$. Using the Marcinkiewicz–Lizorkin multiplier theorem from \cite{lizorkin1970multipliers}, thm 1.1 of Wang--Mei--Wei \cite{wang2025gagliardo} is equivalent to
		\begin{equation*}
			\sum_{j=1}^d \|D^{\sigma}_{x_j}u\|_{L^{\bm{q}}}\lesssim \|u\|_{L^{\bm{p}_1}}^{\theta}\left(\sum_{j=1}^d\|D_{x_j}^su\|_{L^{\bm{p}_2}}\right)^{1-\theta}.
		\end{equation*}
		Defining $\theta_{j,m}$ by
		\begin{align*}
			\frac{\theta}{\bm{p}_1}+\frac{1-\theta}{\bm{p}_2}-\frac{1}{\bm{q}} = \sum_{m=1}^d\theta_{j,m}se_m-\sigma e_j,~~ \mbox{ with } e_m\mbox{ the standard basis},
		\end{align*}
		and noting that $\sum_{m=1}^d \theta_{j,m} = 1-\theta$, we obtain from Theorem \ref{mainresmix} the product inequality
		\begin{align*}
			\|D^{\sigma}_{x_j}u\|_{L^{\bm{q}}}\lesssim \|u\|_{L^{\bm{p}_1}}^{\theta}\prod_{m=1}^d \|D_{x_{m}}^su\|_{L^{\bm{p}_2}}^{\theta_{j,m}},
		\end{align*}
		which contains thm 1.1 in \cite{wang2025gagliardo} as a special case.
	\end{rem}
	\begin{proof}[\textbf{Proof of Theorem \ref{mainresmix}}]
		Let $A_m = \|D^{\bm{s}_m}u\|_{L^{\bm{p}_m}(\mathbb{R}^{\bm{d}})}$.
%
		
	Assume that \hyperlink{(1)}{(1)} holds. By Bernstein's inequality, we obtain
		\begin{align*}
			\|P_k u\|_{L^{\bm{q}}(\mathbb{R}^{\bm{d}})}\lesssim 2^{k\cdot \bm{d}(1/\bm{p}_m-1/\bm{q})}\|P_k u\|_{L^{\bm{p}_m}(\mathbb{R}^{\bm{d}})}\lesssim 2^{k\cdot (\bm{d}(1/\bm{p}_m-1/\bm{q})-\bm{s}_m)} A_m.
		\end{align*}
		Thus
		\begin{align*}
			\|D^{\bm{s}}u\|_{L^{\bm{q}}(\mathbb{R}^{\bm{d}})}&\leq \sum_{k\in \mathbb{Z}^n}\|D^{\bm{s}}P_{k}u\|_{L^{\bm{q}}(\mathbb{R}^{\bm{d}})}\\ &\lesssim \sum_{k\in \mathbb{Z}^n} 2^{k\cdot \bm{s}}\|P_ku\|_{L^{\bm{q}}(\mathbb{R}^{\bm{d}})}\\
			&\lesssim \sum_{k\in \mathbb{Z}^n}2^{k\cdot \bm{s}}\min_{1\leq m\leq M}\{2^{k\cdot (\bm{d}(1/\bm{p}_m-1/\bm{q})-\bm{s}_m)} A_m\}.
		\end{align*}
		Since $\mathrm{diam}(S) = n$, for any $0<\epsilon\ll 1$, there exist $\epsilon_{m,l}\in \mathbb{R}$ with $|\epsilon_{m,l}|<\theta_m$ for any $1\leq m\leq M$, $1\leq l\leq n$ such that
		$$\sum_{m=1}^M \epsilon_{m,l} = 0,\quad\sum_{m=1}^M \epsilon_{m,l}(\bm{s}_m-\bm{d}/\bm{p}_m) = \epsilon e_{l},\quad 1\leq l\leq n,$$
		where $e_l$ denotes the $l$-th standard basis vector in $\mathbb{R}^n$. Thus
		\begin{align*}
			\|D^{\bm{s}}u\|_{L^{\bm{q}}(\mathbb{R}^{\bm{d}})}&\lesssim \sum_{k\in \mathbb{Z}^n}2^{k\cdot \bm{s}}\min_{1\leq l\leq n,\pm}\left\{\prod_{m=1}^M\{2^{k\cdot (\bm{d}(1/\bm{p}_m-1/\bm{q})-\bm{s}_m)} A_m\}^{\theta_m\pm\epsilon_{m,l}}\right\}\\
			&\lesssim \sum_{k\in \mathbb{Z}^n}\min_{1\leq l\leq n,\pm}\left\{2^{k\cdot (\mp\epsilon e_l)}\prod_{m=1}^M A_m^{\theta_m\pm\epsilon_{m,l}}\right\}\\
			&\lesssim \prod_{m=1}^M A_m^{\theta_m}.
		\end{align*}
		
		Assume that \hyperlink{(2)}{(2)} holds. For each $l$, define $\tilde{\bm{p}}^{+}_l$ and $\tilde{\bm{p}}^{-}_l$ by
		$$\frac{1}{\tilde{\bm{p}}^{\pm}_l} = \sum_{m=1}^M 
		\frac{\theta_m\pm\epsilon_{m,l}}{\bm{p}_m}.$$
		Note that for $\sum_{m=1}^M \theta_m \bm{s}_m > \bm{s}$, $|\epsilon_{m,l}|\ll 1$, we have $1/\bm{q}\leq 1/\tilde{\bm{p}}^{\pm}_{ l}$. Then by Bernstein's inequality and H\"{o}lder's inequality, we have
		\begin{align*}
			\|P_ku\|_{L^{\bm{q}}(\mathbb{R}^{\bm{d}})}&\lesssim 2^{k\cdot \bm{d}(1/\tilde{\bm{p}}_{\pm l}-1/\bm{q})}\|P_k u\|_{L^{\tilde{\bm{p}}_l}(\mathbb{R}^{\bm{d}})}\\
			&\lesssim 2^{k\cdot \bm{d}(1/\tilde{\bm{p}}_{\pm l}-1/\bm{q})}\prod_{m=1}^M\|P_ku\|_{L^{\bm{p}_m}(\mathbb{R}^{\bm{d}})}^{\theta_m\pm\epsilon_{m,l}}\\
			&\lesssim \prod_{m=1}^M (2^{k\cdot (\bm{d}(1/\bm{p}_m-1/\bm{q})-\bm{s}_m)}A_m)^{\theta_m\pm\epsilon_{m,l}}.
		\end{align*}
		By the same argument as for case \hyperlink{(1)}{(1)}, the proof of case \hyperlink{(2)}{(2)} is complete.
%

		Assume that \hyperlink{(3)}{(3)} holds. We first establish the following Sobolev inequality: if $\bm{0}<1/\bm{p},1/\bm{q}<\bm{1}$, $\bm{d}(1/\bm{p}-1/\bm{q}) = \bm{s}\geq \bm{0}$, then
		\begin{equation}\label{sobo}
			\|u\|_{L^{\bm{q}}(\mathbb{R}^{\bm{d}})}\lesssim \|D^{\bm{s}}u\|_{L^{\bm{p}}(\mathbb{R}^{\bm{d}})}.
		\end{equation}
		We only need to consider $\bm{s} = (0,\cdots,0,s)$, $s>0$, $\bm{p}(j) = \bm{q}(j)$, $\forall~1\leq j\leq n-1$. Let $\varDelta_{j,l} = \mathscr{F}^{-1}\psi (|\xi_j|/2^l\mathscr{F})$. Then $P_k = \varDelta_{1,k_1}\cdots \varDelta_{n,k_n}$. 
		By Lemma \ref{littlewoodpaley} and Minkowski's inequality, we obtain
		\begin{align*}
			\|u\|_{L^{\bm{q}}(\mathbb{R}^{\bm{d}})} &\sim \||P_k u|_{l^2_{k\in \mathbb{Z}^n}}\|_{L^{\bm{q}}(\mathbb{R}^{\bm{d}})}\\
			&\lesssim \||P_k u|_{l^1_{k_n}l^{2}_{k_1,\cdots,k_{n-1}}}\|_{L^{\bm{q}}(\mathbb{R}^{\bm{d}})}\\
			&\lesssim \||\||P_k u|_{l^{2}_{k_1,\cdots,k_{n-1}}}\|_{_{L^{\bm{q}(n-1)}_{x_{n-1}}\cdots L^{\bm{q}(1)}_{x_1}}}|_{l^1_{k_n}}\|_{L^{\bm{q}(n)}_{x_n}}\\
			&\lesssim \||\|\varDelta_{n,k_n}u\|_{_{L^{\bm{p}(n-1)}_{x_{n-1}}\cdots L^{\bm{p}(1)}_{x_1}}}|_{l^1_{k_n}}\|_{L^{\bm{q}(n)}_{x_n}}.
		\end{align*}
		Below, we use an argument similar to that in the proof of Proposition \ref{littlewood}.
		\begin{align*}
			\|u\|_{L^{\bm{q}}(\mathbb{R}^{\bm{d}})}^{\bm{q}(n)}\lesssim \int_0^\infty \lambda^{\bm{q}(n)}|\{x_n:|\|\varDelta_{n,k_n}u\|_{{L^{\bm{p}(n-1)}_{x_{n-1}}\cdots L^{\bm{p}(1)}_{x_1}}}|_{l^1_{k_n}}>\lambda\}|\frac{d\lambda}{\lambda}.
		\end{align*}		
		Let $f_{k_n}(x_n) = \|\varDelta_{n,k_n}u\|_{_{L^{\bm{p}(n-1)}_{x_{n-1}}\cdots L^{\bm{p}(1)}_{x_1}}}$ and $A = \||2^{sk_n}f_{k_n}(x_n)|_{l^\infty_{k_n}}\|_{L^{\bm{p}(n)}_{x_n}}$. By Bernstein's inequality and Minkowski's inequality, we have
		\begin{align*}
			f_{k_n}(x_n)\lesssim 2^{k_n\frac{d_n}{\bm{p}(n)}}\|\varDelta_{n,k_n}u\|_{L^{\bm{p}}(\mathbb{R}^{\bm{d}})} \lesssim 2^{k_n(\frac{d_n}{\bm{p}(n)}-\bm{s}(n))}A\sim 2^{k_n\frac{d_n}{\bm{q}(n)}}A.
		\end{align*}
		Choose $K(\lambda)$ such that $2^{K(\lambda)\frac{d_n}{\bm{q}(n)}}A \sim c\lambda$ for $0<c\ll 1$.
		Then we obtain
		\begin{align*}
			\|u\|_{L^{\bm{q}}(\mathbb{R}^{\bm{d}})}^{\bm{q}(n)}
			&\lesssim \int_0^\infty \lambda^{\bm{q}(n)}|\{x_n:|f_{k_n}(x_n)|_{l^1_{k_n\geq K(\lambda)}}>\lambda/2\}|\frac{d\lambda}{\lambda}\\
			&\lesssim \int_0^\infty \lambda^{\bm{q}(n)}|\{x_n:|2^{sk_n}f_{k_n}(x_n)|_{l^\infty_{k_n}}>2^{sK(\lambda)}\lambda\}|\frac{d\lambda}{\lambda}\\
			&\lesssim A^{\frac{\bm{s}(n)\bm{q}(n)\bm{p}(n)}{d_n}}\int_0^\infty \mu^{\bm{p}(n)}|\{x_n:|2^{sk_n}f_{k_n}(x_n)|_{l^\infty_{k_n}}>\mu\}|\frac{d\mu}{\mu}\\
			&\sim A^{\frac{\bm{s}(n)\bm{q}(n)\bm{p}(n)}{d_n}}\||2^{sk_n}f_{k_n}(x_n)|_{l^\infty_{k_n}}\|_{L^{\bm{p}(n)}_{x_n}}\sim A^{\bm{q}(n)} 
		\end{align*}
		Thus by Lemma \ref{littlewoodpaley} and Minkowski's inequality, we obtain
		\begin{align*}
			\|u\|_{L^{\bm{q}}(\mathbb{R}^{\bm{d}})}&\lesssim \||2^{sk_n}\|\varDelta_{n,k_n}u\|_{_{L^{\bm{p}(n-1)}_{x_{n-1}}\cdots L^{\bm{p}(1)}_{x_1}}}|_{l^\infty_{k_n}}\|_{L^{\bm{p}(n)}_{x_n}}\\
			&\lesssim \||2^{sk_n}\||P_k u|_{l^{2}_{k_1,\cdots,k_{n-1}}}\|_{_{L^{\bm{p}(n-1)}_{x_{n-1}}\cdots L^{\bm{p}(1)}_{x_1}}}|_{l^\infty_{k_n}}\|_{L^{\bm{p}(n)}_{x_n}}\\
			&\lesssim \||2^{k\cdot \bm{s}}P_k u|_{l^\infty_{k_n}l^{2}_{k_1,\cdots,k_{n-1}}}\|_{L^{\bm{p}}(\mathbb{R}^{\bm{d}})}\\
			&\lesssim \||2^{k\cdot \bm{s}}P_k u|_{l^{2}_{k}}\|_{L^{\bm{p}}(\mathbb{R}^{\bm{d}})}\sim \|D^{\bm{s}}u\|_{L^{\bm{p}}(\mathbb{R}^{\bm{d}})}.
		\end{align*}
		Thus we finish the proof of \eqref{sobo}.
		
		Now set $1/\bm{r} = \sum_{m=1}^M \theta_m/\bm{p}_m$, $\bm{\alpha} = \sum_{m=1}^M \theta_m\bm{s}_m-\bm{s}$. Then we have $\bm{0}<1/\bm{q}\leq 1/\bm{r}<\bm{1}$. Applying \eqref{sobo}, Lemma \ref{littlewoodpaley}, and Hölder's inequality yields
		\begin{align*}
			\|D^{\bm{s}}u\|_{L^{\bm{q}}(\mathbb{R}^{\bm{d}})}&\lesssim \|D^{\bm{s}+\bm{\alpha}}u\|_{L^{\bm{r}}(\mathbb{R}^{\bm{d}})}\\
			&\sim \||2^{k\cdot (\bm{s}+\bm{\alpha})}P_ku|_{l^2_{k\in \mathbb{Z}^n}}\|_{L^{\bm{r}}(\mathbb{R}^{\bm{d}})}\\
			&\lesssim \left\|\prod_{m=1}^M|2^{k\cdot \bm{s}_m}P_ku|_{l^2_{k\in \mathbb{Z}^n}}^{\theta_m}\right\|_{L^{\bm{r}}(\mathbb{R}^{\bm{d}})}\\
			&\lesssim \prod_{m=1}^M\||2^{k\cdot \bm{s}_m}P_ku|_{l^2_{k\in \mathbb{Z}^n}}\|_{L^{\bm{p}_m}(\mathbb{R}^{\bm{d}})}^{\theta_m}\\
			&\sim \prod_{m=1}^M \|D^{\bm{s}_m}u\|_{L^{\bm{p}_m}(\mathbb{R}^{\bm{d}})}^{\theta_m}.
		\end{align*}
		
		Assume that \hyperlink{(4)}{(4)} holds. We first prove the result for $M = 2$. By Lemma \ref{littlewoodpaley} and Minkowski's inequality, we have
		\begin{align*}
			\|D^{\bm{s}}u\|_{L^{\bm{q}}(\mathbb{R}^{\bm{d}})}&\sim \||2^{k\cdot \bm{s}}P_ku|_{l^2_{k\in \mathbb{Z}^n}}\|_{L^{\bm{q}}(\mathbb{R}^{\bm{d}})}\\
			&\lesssim \||2^{k\cdot \bm{s}}P_ku|_{l^1_{k\in \mathbb{Z}^n}}\|_{L^{\bm{q}}(\mathbb{R}^{\bm{d}})}\\
			&\lesssim \||2^{k_n\bm{s}(n)}\cdots\||2^{k_1\bm{s}(1)}P_ku|_{l^1_{k_1}}\|_{L^{\bm{q}(1)}_{x_1}}\cdots |_{l^1_{k_n}}\|_{L_{x_n}^{\bm{q}(n)}}.
		\end{align*}
		We define the following quantities inductively:
		\begin{align*}
			&f_1 = |2^{k_1\bm{s}(1)}P_ku|_{l^1_{k_1}},~F_l = \|f_l\|_{L^{\bm{q}(l)}_{x_l}},~f_{l+1} = |2^{k_{l+1}\bm{s}(l+1)}F_l|_{l^1_{k_{l+1}}}\\
			&g_1 = |2^{k_1\bm{s}_1(1)}P_ku|_{l^\infty_{k_1}},~G_l = \|g_l\|_{L^{\bm{p}_1(l)}_{x_l}},~g_{l+1} =|2^{k_{l+1}\bm{s}_1(l+1)}G_l|_{l^\infty_{k_{l+1}}}\\
			&h_1 = |2^{k_1\bm{s}_2(1)}P_ku|_{l^\infty_{k_1}},~H_l = \|h_l\|_{L^{\bm{p}_2(l)}_{x_l}},~h_{l+1} = |2^{k_{l+1}\bm{s}_2(l+1)}H_l|_{l^\infty_{k_{l+1}}}.
		\end{align*}		
		We prove $F_l\lesssim G_l^{\theta_1}H_l^{\theta_2}$ by induction. Set $1/\bm{r} = \theta_1/\bm{p}_1+\theta_2/\bm{p}_2$. By the Bernstein inequality, we have
		\begin{align*}
			\|P_ku\|_{L_{x_1}^{\infty}} &\lesssim \min\{2^{k_1d_1/\bm{p}_1(1)}	\|P_ku\|_{L_{x_1}^{\bm{p}_1(1)}},2^{k_1d_1/\bm{p}_2(1)}	\|P_ku\|_{L_{x_1}^{\bm{p}_2(1)}}\}\\
			&\lesssim \min\{2^{k_1(d_1/\bm{p}_1(1)-\bm{s}_1(1))}G_1,2^{k_1(d_1/\bm{p}_2(1)-\bm{s}_2(1))}	H_1\}\\
			&\lesssim 2^{k_1(d_1/\bm{q}(1)-\bm{s}(1))} G_1^{\theta_1}H_1^{\theta_2}.
		\end{align*}
		Let $K_\lambda$ such that $2^{K_\lambda d_1/\bm{q}(1)} G_1^{\theta_1}H_1^{\theta_2}\sim c\lambda$ for some $0<c\ll 1$. Then
		\begin{align*}
			F_1^{\bm{q}(1)}&\sim \int_0^\infty \lambda^{\bm{q}(1)}\left|\left\{x_1:\sum_{k_1\in \mathbb{Z}} 2^{k_1\bm{s}(1)}|P_ku(x)|>\lambda\right\}\right|\frac{d\lambda}{\lambda}\\
			&\sim \int_0^\infty \lambda^{\bm{q}(1)}\left|\left\{x_1:\sum_{k_1\geq K_\lambda} 2^{k_1\bm{s}(1)}|P_ku(x)|>\frac{\lambda}{2}\right\}\right|\frac{d\lambda}{\lambda}\\
			&\lesssim \int_0^\infty \lambda^{\bm{q}(1)}\left|\left\{x_1: 2^{K_\lambda(\bm{s}(1)-\theta_1 \bm{s}_1(1)-\theta_2\bm{s}_2(1))}g_1^{\theta_1}h_1^{\theta_2}>\frac{\lambda}{2}\right\}\right|\frac{d\lambda}{\lambda}\\
			& \lesssim \int_0^\infty \lambda^{\bm{q}(1)}\left|\left\{x_1: (\lambda/G_1^{\theta_1} H_1^{\theta_2})^{\frac{1/\bm{q}(1)-1/\bm{r}(1)}{1/\bm{q}(1)}}g_1^{\theta_1}h_1^{\theta_2}\gtrsim \lambda\right\}\right|\frac{d\lambda}{\lambda}\\
			&\lesssim (G_1^{\theta_1}H_1^{\theta_2})^{\bm{q}(1)-\bm{r}(1)}\int_0^\infty \mu^{\bm{r}(1)}\{x_1:g_1^{\theta_1} h_1^{\theta_2}> \mu\}\frac{d\mu}{\mu}\\
			&\sim (G_1^{\theta_1}H_1^{\theta_2})^{\bm{q}(1)-\bm{r}(1)}\|g_1^{\theta_1}h_1^{\theta_2}\|_{L^{\bm{r}(1)}_{x_1}}^{\bm{r}(1)}\\
			&\lesssim (G_1^{\theta_1}H_1^{\theta_2})^{\bm{q}(1)}.
		\end{align*}
		Thus $F_1\lesssim G_1^{\theta_1}H_1^{\theta_2}$. Likewise, applying Bernstein's inequality gives
		\begin{align*}
			&\quad\|F_l\|_{L_{x_{l+1}}^{\infty}}\\
			&\lesssim \|G_l^{\theta_1}H_l^{\theta_2}\|_{L^\infty_{x_{l+1}}}\\
			&\lesssim (2^{k_{l+1}(d_{l+1}/\bm{p}_1(l+1)-\bm{s}_1(l+1))}G_{l+1})^{\theta_1}(2^{k_{l+1}(d_{l+1}/\bm{p}_2(l+1)-\bm{s}_2(l+1))}H_{l+1})^{\theta_2}\\
			&\lesssim 2^{k_{l+1}(d_{l+1}/\bm{q}(l+1)-\bm{s}(l+1))}G_{l+1}^{\theta_1}H_{l+1}^{\theta_2}.
		\end{align*}
		Let $K_\lambda$ such that $2^{k_{l+1}d_{l+1}/\bm{q}(l+1)}G_{l+1}^{\theta_1}H_{l+1}^{\theta_2}\sim c\lambda$ for some $0<c\ll 1$. Then
		\begin{align*}
			F_{l+1}^{\bm{q}(l+1)}&\sim \int_0^\infty \lambda^{\bm{q}(1)}\left|\left\{x_{l+1}:\sum_{k_{l+1}\in \mathbb{Z}} 2^{k_{l+1}\bm{s}(l+1)}F_l>\lambda\right\}\right|\frac{d\lambda}{\lambda}\\
			&\sim \int_0^\infty \lambda^{\bm{q}(l+1)}\left|\left\{x_1:\sum_{k_{l+1}\geq K_\lambda} 2^{k_{l+1}\bm{s}(l+1)}F_l>\frac{\lambda}{2}\right\}\right|\frac{d\lambda}{\lambda}\\
			&\lesssim \int_0^\infty \lambda^{\bm{q}(l+1)}\left|\left\{x_1: 2^{K_\lambda(\bm{s}(l+1)-\theta_1 \bm{s}_1(l+1)-\theta_2\bm{s}_2(l+1))}G_l^{\theta_1}H_l^{\theta_2}>\frac{\lambda}{2}\right\}\right|\frac{d\lambda}{\lambda}\\
			& \lesssim \int_0^\infty \lambda^{\bm{q}(l+1)}\left|\left\{x_1: (\lambda/G_{l+1}^{\theta_1} H_{l+1}^{\theta_2})^{\frac{1/\bm{q}(l+1)-1/\bm{r}(l+1)}{1/\bm{q}(l+1)}}G_l^{\theta_1}H_l^{\theta_2}\gtrsim \lambda\right\}\right|\frac{d\lambda}{\lambda}\\
			&\lesssim (G_{l+1}^{\theta_1}H_{l+1}^{\theta_2})^{\bm{q}(l+1)-\bm{r}(l+1)}\|G_l^{\theta_1}H_l^{\theta_2}\|_{L^{\bm{r}(1)}_{x_1}}^{\bm{r}(l+1)}\lesssim (G_{l+1}^{\theta_1}H_{l+1}^{\theta_2})^{\bm{q}(l+1)}.
		\end{align*}
		We obtain $F_{l+1}\lesssim G_{l+1}^{\theta_1}H_{l+1}^{\theta_2}$. An induction argument yields $F_n\lesssim G_n^{\theta_1}H_n^{\theta_2}$.
		
		We now define the Hardy–Littlewood maximal operator by
		\begin{align*}
			M_jf(x_j) = \sup_{R>0}\frac{1}{R^{d_j}}\int_{\mathbb{R}^{d_j}}|f(y_j)|\chi_{|x_j-y_j|\leq R}~dy_j.
		\end{align*}
		Then
		\begin{align*}
			&\quad\|D^{\bm{s}}u\|_{L^{\bm{q}}(\mathbb{R}^{\bm{d}})}\\
			&\lesssim G_n^{\theta_1}H_n^{\theta_2}\\
			&\sim \|2^{k\cdot \bm{s}_1}P_ku\|_{L^{\bm{p}_1(n)}_{x_n}l^\infty_{k_n}\cdots L^{\bm{p}_1(1)}_{x_1}l^\infty_{k_1}}^{\theta_1}\|2^{k\cdot \bm{s}_2}P_ku\|_{L^{\bm{p}_2(n)}_{x_n}l^\infty_{k_n}\cdots L^{\bm{p}_2(1)}_{x_1}l^\infty_{k_1}}^{\theta_2}\\
			&\lesssim \|M_n\cdots\|M_1D^{\bm{s}_1}u\|_{L^{\bm{p}_1(1)}_{x_1}}\cdots\|_{L_{x_n}^{\bm{p}_1(n)}}^{\theta_1}\|M_n\cdots\|M_1D^{\bm{s}_2}u\|_{L^{\bm{p}_2(1)}_{x_1}}\cdots\|_{L_{x_n}^{\bm{p}_2(n)}}^{\theta_2}\\
			&\lesssim \|D^{\bm{s}_1}u\|_{L^{\bm{p}_1}(\mathbb{R}^{\bm{d}})}^{\theta_1}\|D^{\bm{s}_2}u\|_{L^{\bm{p}_2}(\mathbb{R}^{\bm{d}})}^{\theta_2}.
		\end{align*}		
		For $K\geq 3$, let $0<\epsilon\ll 1$, set
		$$\frac{1}{\tilde{\bm{p}}_m} = \frac{1}{2}\left(\frac{1}{\bm{q}}+\frac{1}{\bm{p}_m}\right)-\epsilon \bm{1},\quad \tilde{\bm{s}}_m = \frac{1}{2}(\bm{s}+\bm{s}_m)-\epsilon \bm{d}.$$
		Then $\bm{0}<1/\tilde{\bm{p}}_m<\bm{1}$ and		
		\begin{equation*}
			\|D^{\tilde{\bm{s}}_m}u\|_{L^{\tilde{\bm{p}}_m}}\lesssim \|D^{\bm{s}}u\|_{L^{\bm{q}}(\mathbb{R}^{\bm{d}})}^{\frac{1}{2}}\|D^{\bm{s}_m}u\|_{L^{\bm{p}_m}(\mathbb{R}^{\bm{d}})}^{\frac{1}{2}}.
		\end{equation*}
		Applying the result for case \hyperlink{(3)}{(3)}, we obtain
		\begin{align*}
			\|D^{\bm{s}}u\|_{L^{\bm{q}}(\mathbb{R}^{\bm{d}})}&\lesssim \prod_{m=1}^M \|D^{\tilde{\bm{s}}_m}u\|_{L^{\tilde{\bm{p}}_m}}^{\theta_m}\\
			&\lesssim \prod_{m=1}^M  \left( \|D^{\bm{s}}u\|_{L^{\bm{q}}(\mathbb{R}^{\bm{d}})}^{\frac{1}{2}}\|D^{\bm{s}_m}u\|_{L^{\bm{p}_m}(\mathbb{R}^{\bm{d}})}^{\frac{1}{2}}\right)^{\theta_m}\\
			&\lesssim \|D^{\bm{s}}u\|_{L^{\bm{q}}(\mathbb{R}^{\bm{d}})}^{\frac{1}{2}}\left(\prod_{m=1}^M   \|D^{\bm{s}_m}u\|_{L^{\bm{p}_m}(\mathbb{R}^{\bm{d}})}^{\theta_m}\right)^{\frac{1}{2}}.
		\end{align*}
		The proof is complete.
	\end{proof}
	
	\section{The profile decomposition}\label{profiledeco}
	The concentration-compactness argument is useful for establishing Theorem \ref{existextri}; see Lions \cite{lions1984concentration,lions1984concentration2,lions1985concentration,lions1985concentration2}. In this section, we present the profile decomposition to describe the concentration-compactness behavior relative to the embeddings $H^{\bm{s}}_{\bm{d}} \hookrightarrow L^q$ and $\dot{H}^{\bm{s}}_{\bm{d}} \hookrightarrow L^q$. This decomposition is particularly useful for proving the existence of maximizers of certain functionals, and is also a powerful tool in recent studies of the long-time behavior of dispersive equations. The method was introduced by G\'{e}rard \cite{gerard1998}, Bahouri--G\'{e}rard \cite{bahourigerard}, and our presentation mainly follows that of \cite{koch2014}.
	
	From now on, we assume the following: $$n\in \mathbb{N},~~ d_j\in \mathbb{N},~~ d = \sum_{j=1}^nd_j,~~\mathbb{R}^d= \oplus_{j=1}^n \mathbb{R}^{d_j},$$
	with $x = (x_1,\cdots,x_n)$, $\xi = (\xi_1,\cdots,\xi_n)$, $x_j,\xi_j\in \mathbb{R}^{d_j}$. Let $2<q<\infty$, $\bm{d} = (d_1,\cdots,d_n)$, $\bm{s} = (s_1,\cdots,s_n)$ with $s_j>0$ for each $1\leq j\leq n$.
	
	Recall the definitions of $H^{\bm{s}}_{\bm{d}}$ and $\dot{H}^{\bm{s}}_{\bm{d}}$ from \eqref{defiHds} and \eqref{defidotHds}, respectively, and that of $s$ from \eqref{avreandin}.
	\begin{prop}\label{linearprofilesubcritical}
		Assume that $1/q>1/2-s/d$. Let $\{u_l\}$ be a bounded sequence in $H^{\bm{s}}_{\bm{d}}$. Then after replacing it by a subsequence, there exist a bounded sequence $\{U^j\}$ in $H^{\bm{s}}_{\bm{d}}$, and sequences $\{x_{l}^j\}\subset \mathbb{R}^d$ such that for $r_l^k$ defined by
		$$u_l(x) = \sum_{j=1}^k U^j(x+x_l^j)+r_l^k(x),$$
		we have $r_l^k(\cdot-x_l^j)\rightharpoonup 0$ weakly in $H^{\bm{s}}_{\bm{d}}$ as $l\rightarrow \infty$ for any $1\leq j\leq k$ and
		\begin{equation}\label{almostorth}
			\lim_{k\rightarrow \infty} \lim_{l\rightarrow \infty}\|r_l^k\|_{L^q} = 0,\quad \lim_{l\rightarrow \infty}|x_l^j-x_l^k| = \infty,~~ \forall~j\neq k.
		\end{equation}
		Also,
		\begin{equation}\label{lqdec}
			\lim_{l\rightarrow \infty}\|u_l\|_{L^q}^q =\lim_{l\rightarrow \infty}\|r_l^k\|_{L^q}^q+ \sum_{j=1}^k\|U^j\|_{L^q}^q,
		\end{equation}
		and
		\begin{equation}\label{l2dec}
			\limsup_{l\rightarrow \infty} \|T_m u_l\|_{L^2}^2 = \limsup_{l\rightarrow \infty} \|T_m r_l^k\|_{L^2}^2+\sum_{j=1}^k \|T_m U^j\|_{L^2}^2
		\end{equation}
		for any $T_m = \mathscr{F}^{-1}m(\xi)\mathscr{F}$, $|m(\xi)|\lesssim 1+|\xi_1|^{s_1}+\cdots +|\xi_n|^{s_n}$.
	\end{prop}
	\begin{proof}[\textbf{Proof}]
		Let $A = \lim_{l}\|u_l\|_{H^{\bm{s}}_{\bm{d}}}$, $\epsilon = \lim_l \|u_l\|_{L^q}$. If $\epsilon = 0$, we may set $U^j = 0$ and choose the $x_l^j$ arbitrarily so that \eqref{almostorth} holds. If $\epsilon>0$, by choosing a subsequence (still denoted by $u_l$), we may find $k\in \mathbb{Z}^n$ with $|k|\lesssim_{A,\epsilon} 1$ such that $\|P_ku_l\|_{L^q}\gtrsim _{A,\epsilon} 1$ for each $l$. By the H\"{o}lder inequality and $\|u_n\|_{L^2}\lesssim 1$, we obtain $\|P_ku_l\|_{L^\infty}\gtrsim_{A,\epsilon} 1$. We choose $x_l^1$ such that $|P_k u_l(-x_l^1)| = \|P_ku_l\|_{L^\infty}$. By choosing a subsequence, there exists $U^1$ such that $u_l(\cdot-x_l^1)$ converges to $U^1$ weakly in $H^{\bm{s}}_{\bm{d}}$ as $l\rightarrow \infty$. Since $\|\partial^\alpha_x P_k u_l\|_{L^\infty}\lesssim_{A,\epsilon,\alpha} 1$, we would also assume $P_k u_l(-x_l^1)$ converges to $P_k U^1$ in $L^\infty$. Thus $\|P_k U^1\|_{L^\infty}\gtrsim_{A,\epsilon} 1$. By the Bernstein inequality and the H\"{o}lder inequality, we have $\|U^1\|_{L^r} \sim_{r,A,\epsilon} 1$ for any $2\leq r\leq \infty$. Let $s_* = \min_{1\leq j\leq n}\{s_j\}$. We have $H^{\bm{s}}_{\bm{d}}\hookrightarrow H^{s_*}(\mathbb{R}^d)$. Thus, for any bounded sequence in $H^{\bm{s}}_{\bm{d}}$, we can extract a subsequence that converges almost everywhere. Consequently, we may assume that $u_l(\cdot-x_l^1)\rightarrow U^1(\cdot)$, $a.e.$ as $l\rightarrow \infty$. By the refined Fatou lemma (see, e.g., \cite{lieb2001analysis}), we have
		\begin{align*}
			\lim_{l\rightarrow \infty}\|r_l^1\|^q_{L^q} = \lim_{l\rightarrow \infty}\|u_l(\cdot-x_l^1)-U^1(\cdot)\|_{L^q}^q & = \lim_{l\rightarrow \infty} \|u_l\|_{L^q}^q-\|U^1\|_{L^q}^q\\
			& = \epsilon^q - \|U^1\|_{L^q}^q.
		\end{align*}
		For the sequence $\{r_l^1\}$, we can repeat the former procedure for $\{u_l\}$. Then, we obtain $\{U^j\}$, $\{x_l^j\}$, and
		$$\lim_{l\rightarrow \infty}\|r_l^k\|_{L^q}^q = \epsilon^q-\sum_{j=1}^k \|U^j\|_{L^q}^q.$$
		Note that if $\lim_{l\rightarrow \infty}\|r_l^k\|_{L^q}^q = \delta$, we have $\|U^j\|_{L^q}\gtrsim_{A,\delta} 1$ for any $1\leq j\leq k$. Thus
		$$\lim_{k\rightarrow \infty}\lim_{l\rightarrow \infty}\|r_l^k\|_{L^q}^q = 0.$$
		We now prove by induction that $r_l^k(\cdot-x_l^j)\rightharpoonup 0$ weakly in $H^{\bm{s}}_{\bm{d}}$ as $l\rightarrow \infty$ for every $1\leq j\leq k$. Assume this holds for all $k\leq k_0$ (with $1\leq j\leq k$). The base case $k_0$ = 1 follows from the construction. For the induction step, note that
		$$r_l^{k_0+1}(x-x_{l}^{k_0+1}) = r_l^{k_0}(x-x_{l}^{k_0+1})-U^{k_0+1}(x).$$
		By the definition of $U^{k_0+1}$, the left-hand side tends weakly to $0$ as $l\rightarrow \infty$, which proves the claim for $j = k_0+1$. For $1\leq j\leq k_0$, we have
		\begin{equation}\label{beforetrs}
			r_l^{k_0+1}(x-x_l^j) =  r_l^{k_0}(x-x_{l}^j)-U^{k_0+1}(x+x_l^{k_0+1}-x_l^j).
		\end{equation}
		If $x_l^{k_0+1}-x_l^j$ converges to some $x_*\in \mathbb{R}^d$ as $l\rightarrow \infty$, we have
		$$r_l^{k_0+1}(x-x_{l}^{k_0+1}) =r_l^{k_0+1}(x-x_l^j+x_l^j-x_{l}^{k_0+1}) \rightharpoonup -U^{k_0+1}(x).$$
		Thus $U^{k_0+1} = 0$, which implies $U^j = 0$, $\lim_{l\rightarrow\infty}\|r_l^j\|_{L^q} = 0$ for any $j\geq k_0+1$. For such indices, the sequences $\{x_l^j\}$ can be selected arbitrarily. If $|x_l^{k_0+1}-x_l^j|\rightarrow \infty$ as $l\rightarrow \infty$, by \eqref{beforetrs} and the induction hypothesis we obtain $r_l^{k_0+1}(x-x_l^j)\rightharpoonup 0$ weakly in $H^{\bm{s}}_{\bm{d}}$ as $l\rightarrow \infty$.
		
		The proof also shows that we may choose $\{x_l^j\}$ such that $\lim_{l\rightarrow \infty}|x_l^j-x_l^k| = \infty$ for all $j\neq k$. Applying this orthogonality yields \eqref{lqdec}–\eqref{l2dec}, completing the proof.
	\end{proof}
	
	The critical case is more subtle than the subcritical one. We follow the argument presented in Tao \cite{tao2006nonlinear}, Appendix A. To proceed, we first introduce some notation.
	
	For $y\in \mathbb{R}^d$, define the translation operator by
	$$T_y u(x) = u(x+y).$$
	
	For $\mu>0$, define the scaling operator by
	$$S_\mu u(x) =\mu^{\frac{d}{sq}} u(\mu^{\frac{1}{s_1}} x_1,\cdots,\mu^{\frac{1}{s_n}}x_n).$$
	These operators satisfy
	$$\|S_\mu u\|_{L^q} = \|u\|_{L^q},~~\|S_\mu D^{s_j}_{x_j}u\|_{L^2} = \| D^{s_j}_{x_j}u\|_{L^2},~~ S_{\mu_1}S_{\mu_2} = S_{\mu_1\mu_2},~~ S_1 = Id,$$
	and
	\begin{equation}\label{defimmu}
		S_\mu T_y = T_{m_\mu y}S_\mu,~~\mbox{ where }m_\mu y := (\mu^{-\frac{1}{s_1}}y_1,\cdots,\mu^{-\frac{1}{s_n}}y_n).
	\end{equation}
	Moreover, we have
	\begin{equation}\label{comscfreqc}
		\begin{aligned}
			&\quad P_k S_\mu u(x)\\
			& = \mathscr{F}^{-1}\left(\prod_{j=1}^n \psi (2^{-k_j}|\xi_j|)\mu^{\frac{d}{sq}-\frac{d}{s}}\hat{u}(\mu^{-\frac{1}{s_1}} \xi_1,\cdots,\mu^{-\frac{1}{s_n}}\xi_n)\right)(x)\\
			& = S_\mu\mathscr{F}^{-1}\left(\prod_{j=1}^n \psi (2^{-k_j}\mu^{\frac{1}{s_j}}|\xi_j|)\hat{u}(\xi)\right)(x).
		\end{aligned}
	\end{equation}
	%
	
	
	We prove the following refined anisotropic Gagliardo–Nirenberg inequality. For the fractional Gagliardo–Nirenberg inequality in Lorentz spaces, see, for example, \cite{wei2023gagliardo} and the references therein.
	\begin{prop}\label{lorentzrefined}
		For $s<d/2$ and $1/q = 1/2-s/d$, the estimate $\|u\|_{L^{q,2}}\lesssim \|u\|_{\dot{H}^{\bm{s}}_{\bm{d}}}$ holds.
	\end{prop}
	\begin{proof}[\textbf{Proof}]
		We use the argument in \cite{tao2006nonlinear}, Appendix A. Let $\psi_{\bm{s}}:=\mathscr{F}^{-1}(1/(|\xi_1|^{s_1}+\cdots+|\xi_n|^{s_n}))$. For any measurable $E\subset \mathbb{R}^d$ satisfying $|E|\leq 1$, we claim
		\begin{equation*}
			\int_{E}|\psi_{\bm{s}}(x)|~dx\lesssim 1.
		\end{equation*}
		Typically, we need to control the term
		\begin{align*}
			&\quad\int_{E}\left|\int_{\mathbb{R}^d}\frac{(\prod_{k=1}^l\varphi(|x_{j_k}||\xi_{j_k}|))(\prod_{k=l+1}^n(1-\varphi(|x_{j_k}||\xi_{j_k}|))) e^{ix\cdot \xi}}{|\xi_1|^{s_1}+\cdots + |\xi_n|^{s_n}}~d\xi\right|~dx
		\end{align*}
		where $\{j_1,\cdots,j_n\} = \{1,\cdots, n\}$. Integrating by parts with respect to the variables $\xi_{j_k}$ ($k = l+1, \cdots, n$), we can bound this term by
		\begin{align*}
			&\quad\int_{E}\left|\int_{\mathbb{R}^d}\frac{(\prod_{k=1}^l\chi_{|\xi_{j_k}||x_{j_k}|\lesssim 1})(\prod_{k=l+1}^n\chi_{|\xi_{j_k}||x_{j_k}|\gtrsim 1}) }{(|\xi_1|^{s_1}+\cdots + |\xi_n|^{s_n})\prod_{k=l+1}^n(|x_{j_k}||\xi_{j_k}|)^L}~d\xi\right|~dx.
		\end{align*}
		By choosing $L>d$, we have
		\begin{align*}
			&\quad\int_{E}\left|\int_{\mathbb{R}^d}\frac{(\prod_{k=1}^l\chi_{|\xi_{j_k}||x_{j_k}|\lesssim 1})(\prod_{k=l+1}^n\chi_{|\xi_{j_k}||x_{j_k}|\gtrsim 1}) }{(|\xi_1|^{s_1}+\cdots + |\xi_n|^{s_n})\prod_{k=l+1}^n(|x_{j_k}||\xi_{j_k}|)^L}~d\xi\right|~dx\\
			& = \int_{\mathbb{R}^d}\int_{E}\frac{(\prod_{k=1}^l\chi_{|\xi_{j_k}||x_{j_k}|\lesssim 1})(\prod_{k=l+1}^n\chi_{|\xi_{j_k}||x_{j_k}|\gtrsim 1}) }{(|\xi_1|^{s_1}+\cdots + |\xi_n|^{s_n})\prod_{k=l+1}^n(|x_{j_k}||\xi_{j_k}|)^L}~dxd\xi\\
			& \lesssim \int_{\mathbb{R}^d}\frac{\min\{1,\prod_{j=1}^n |\xi_j|^{-d_j}\}}{|\xi_1|^{s_1}+\cdots + |\xi_n|^{s_n}} ~d\xi\\
			& \lesssim \int_0^\infty \cdots \int_0^\infty \frac{\min\{1,\prod_{j=1}^{n}\rho_j^{-d_j}\}}{\rho_1^{s_1}+\cdots+\rho_n^{s_n}}~\rho_1^{d_1-1}\cdots \rho_n^{d_n-1}d\rho_1\cdots d\rho_n.
		\end{align*}
		Let $\rho_{j_0}^{s_{j_0}} = \max\{\rho_1^{s_1},\cdots,\rho_n^{s_n}\}$. If $\rho_{j_0}\leq 1$, we can control this part by
		\begin{equation*}
			\int_0^1 \rho_{j_0}^{d_{j_0}-1-s_{j_0}+\sum_{j\neq j_0} \frac{d_js_{j_0}}{s_{j}}}~d\rho_{j_0} = \int_0^1 \rho^{-1-s_{j_0}+ \frac{ds_{j_0}}{s}}~d\rho<\infty.
		\end{equation*}
		If $\rho_{j_0}> 1$, by choosing $0<\epsilon\ll 1$ we can control this part by
		\begin{equation*}
			\int_1^\infty \rho_{j_0}^{-1-s_{j_0}+ \frac{ds_{j_0}\epsilon}{s}}~d\rho_{j_0} <\infty.
		\end{equation*}
		
		If $|E| = \lambda$, we set $\mu = \lambda^{-s/d}$ and $\lambda_j = \mu^{1/s_j}$, and define $E_\lambda = \{(\lambda_jx_j)_{j=1}^n: x\in E\}$. Then $|E_\lambda| = 1$, and hence
		\begin{align*}
			\int_{E}|\psi_{\bm{s}}(x)|~dx =  \frac{1}{\mu}\int_{E_\lambda}|\psi_{\bm{s}}(x)|~dx\lesssim \mu^{-1}\sim \lambda^{\frac{s}{d}}.
		\end{align*}
		Let $E_k\subset \mathbb{R}^d$ with measure $2^k$, $\chi_k$ be a function bounded in magnitude by $1$ and supported on the set $S_k$ with measure at most $2^k$. Then
		\begin{align*}
			\int_{E_{k'}}|\chi_k *\psi_{\bm{s}}(x)|~dx
			&\leq \int_{E_{k'}}\int_{S_k}|\psi_{\bm{s}}(x-y)|~dydx\\
			&\lesssim \min\{|E_{k'}||S_k|^{\frac{s}{d}}, |S_k||E_{k'}|^{\frac{s}{d}}\}\\
			&\lesssim \min\{2^{k'+\frac{s}{d}k},2^{k+\frac{s}{d}k'}\}.
		\end{align*}
		Using the equivalent norm on Lorentz spaces (see \cite{tao2006nonlinear}, Appendix A), we have
		\begin{align*}
			\|f*\psi_{\bm{s}}\|_{L^{q,2}} \sim \sup\left\|2^{-\frac{k}{q'}}\int_{E_k}f*\psi_{\bm{s}}(x)~dx\right\|_{l^2_k}.
		\end{align*}
		where the supremum is taken over all sets $E_k$ with $|E_k| = 2^k$ for each $k\in \mathbb{Z}$. Let $f\in L^2$. Then $f$ admits a decomposition $f = \sum_{k\in \mathbb{Z}}c_k\chi_k$ with $c_k\geq 0$, $|\chi_k|\leq 1$, $|\mathrm{supp}(\chi_k)|\leq 2^k$, and $\|f\|_{L^2}\sim \|2^{k/2}c_k\|_{l^2_k}$ (see \cite{keel1998endpoint}). Then
		\begin{align*}
			\left|\int_{E_k}f*\psi_{\bm{s}}(x)~dx\right| &\leq \sum_l c_l\left|\int_{E_k}\chi_{l}*\psi_{\bm{s}}(x)~dx\right|\\
			&\lesssim \sum_l c_l\min\{2^{l+\frac{s}{d}k},2^{k+\frac{s}{d}l}\}.
		\end{align*}
		Thus
		\begin{align*}
			\|f*\psi_{\bm{s}}\|_{L^{q,2}}& \lesssim \left\|2^{-\frac{k}{q'}}\sum_l c_l\min\{2^{l+\frac{s}{d}k},2^{k+\frac{s}{d}l}\}\right\|_{l^2_k}\\
			& \sim \left\|\sum_l 2^{\frac{l}{2}}c_l\min\{2^{\frac{l-k}{2}},2^{\frac{k-l}{q}}\}\right\|_{l^2_k} \lesssim \|2^{\frac{k}{2}}c_k\|_{l^2_k}\sim \|f\|_{L^2}.
		\end{align*}
		The proof is finished by virtue of the identity $u = ((\sum_{j=1}^n D^{s_j}_{x_j})u)*\psi_{\bm{s}}$.
	\end{proof}
	\begin{lemma}\label{localconcentration}
		Assume $s<d/2$ and $1/q = 1/2-s/d$. Let $0<\eta\leq 1$. If $\eta\lesssim \|u\|_{L^q}\lesssim \|u\|_{\dot{H}^{\bm{s}}_{\bm{d}}}\lesssim 1$, then there exists $k\in \mathbb{Z}^n$ and $x_0\in \mathbb{R}^d$ such that $|P_k u(x_0)|\sim_{\eta} 2^{(k_1 d_1+\cdots+k_nd_n)/q}$.
	\end{lemma}
	\begin{proof}[\textbf{Proof}]
		First, by the Lorentz-space characterization of $L^q$ (see Exercise A.4 in \cite{tao2006nonlinear}), there exist sets $E_l$ with $|E_l| = 2^l$ such that
		$$\sum_l 2^{l(1-q)}\left|\int_{E_l}u(x)~dx \right|^q \sim \eta^q.$$
		By Proposition \ref{lorentzrefined}, we have
		$$\sum_l 2^{l(\frac{2}{q}-2)}\left|\int_{E_l}u(x)~dx \right|^2 \lesssim \|u\|_{L^{q,2}}^2\lesssim \|u\|_{\dot{H}^{\bm{s}}_{\bm{d}}}^2\lesssim 1.$$
		By the H\"{o}lder inequality, there exist $l_0$ such that
		\begin{equation}\label{lorentz}
			2^{l_0(\frac{1}{q}-1)}\left|\int_{E_{l_0}}u(x)~dx \right|\gtrsim \eta^{\frac{q}{q-2}}.
		\end{equation}
		On the other hand, by the triangle inequality and the H\"{o}lder inequality, we have
		\begin{align*}
			\left|\int_{E_{l_0}}u(x)~dx \right| &\leq \sum_k \left|\int_{E_{l_0}}P_k u(x)~dx \right|\\
			&\leq \sum_k \min\{\|P_k u\|_{L^2}|E_{l_0}|^{\frac{1}{2}},\|P_k u\|_{L^\infty}|E_{l_0}|\}\\
			&\lesssim \sum_k \min\{2^{-k_1s_1+\frac{l_0}{2}},\cdots,2^{-k_ns_n+\frac{l_0}{2}},\|P_k u\|_{L^\infty}2^{l_0}\}\\
			&\lesssim 2^{l_0-\frac{l_0}{q}}\left(\sup_k \|P_k u\|_{L^\infty}2^{-\frac{k_1 d_1+\cdots+k_nd_n}{q}}\right)^{1-\frac{2}{q}}.
		\end{align*}
		Combining with \eqref{lorentz}, we obtain $\sup_k \|P_k u\|_{L^\infty}2^{-\frac{k_1 d_1+\cdots+k_nd_n}{q}}\gtrsim_\eta 1$. Applying Bernstein’s inequality gives $\sup_{k}\|P_k u\|_{L^\infty}2^{-\frac{k_1 d_1+\cdots+k_nd_n}{q}}\lesssim \|u\|_{L^q}\lesssim \eta$, finishing the proof.
	\end{proof}
	
	\begin{prop}\label{criticallinearprofiledec}
		Assume $s<d/2$ and $1/q = 1/2-s/d$. Let $\{u_l\}$ be a bounded sequence in $\dot{H}^{\bm{s}}_{\bm{d}}$. Then after replacing it by a subsequence, there exist a bounded sequence $\{U^j\}$ in $\dot{H}^{\bm{s}}_{\bm{d}}$, and sequences $\mu_l^j \in (0,\infty)$, 
		$x_{l}^j\in  \mathbb{R}^d$ such that for $r_l^k$ defined by
		$$u_l = \sum_{j=1}^k T_{x_l^j}S_{\mu_l^j}U^j+r_l^k,$$
		we have $S_{\mu_l^j}^{-1}T^{-1}_{x_l^j}r_l^k \rightharpoonup 0$ weakly in $\dot{H}^{\bm{s}}_{\bm{d}}$ as $l\rightarrow \infty$ for any $1\leq j\leq k$ and
		\begin{equation}\label{almostorthc}
			\lim_{k\rightarrow \infty} \lim_{l\rightarrow \infty}\|r_l^k\|_{L^q} = 0,\quad \lim_{l\rightarrow \infty} \left|\log\frac{\mu_{l}^j}{\mu_l^k}\right|+|m_{\frac{1}{\mu_l^j}}(x_l^j-x_l^k)| = \infty,~~ \forall~j\neq k.
		\end{equation}
		Also,
		\begin{equation}\label{lqdecc}
			\lim_{l\rightarrow \infty}\|u_l\|_{L^q}^q =\lim_{l\rightarrow \infty}\|r_l^k\|_{L^q}^q+ \sum_{j=1}^k\|U^j\|_{L^q}^q,
		\end{equation}
		and
		\begin{equation}\label{l2decc}
			\limsup_{l\rightarrow \infty} \|T_m u_l\|_{L^2}^2 = \limsup_{l\rightarrow \infty} \|T_m r_l^k\|_{L^2}^2+\sum_{j=1}^k \|T_m U^j\|_{L^2}^2
		\end{equation}
		for any $T_m = \mathscr{F}^{-1}m(\xi)\mathscr{F}$, $m(\mu^{\frac{1}{s_1}}\xi_1,\cdots,\mu^{\frac{1}{s_n}}\xi_n) = \mu m(\xi)$ for any $\mu>0$, $\xi\in \mathbb{R}^d$.
	\end{prop}
	\begin{proof}[\textbf{Proof}]
		Let $A = \lim_{l}\|u_l\|_{\dot{H}^{\bm{s}}_{\bm{d}}}$, $\epsilon = \lim_l \|u_l\|_{L^q}$. If $\epsilon = 0$, we define $U^j = 0$ and $x_l^j$, $\mu_l^j$ arbitrarily. If $\epsilon>0$, by Lemma \ref{localconcentration} there exist $k_l^1=(k^1_{l,1},\cdots,k^1_{l,n})\in \mathbb{Z}^n$, $x_l^1\in \mathbb{R}^d$ such that $|P_{k_l^1}u_l(-x_l^1)|\gtrsim _{A,\epsilon} 2^{(k^1_{l,1}d_1+\cdots+k^1_{l,n}d_n)/q}$ for any $l$. Let $\mu_l^1 = 2^{(k_{l,1}^1d_1+\cdots+k_{l,n}^1d_n)s/d}$. By choosing a subsequence, there exists $U^1$ such that $S_{\mu_l^1}^{-1}T^{-1}_{x_l^1}u_l$ converges to $U^1$ weakly in $\dot{H}^{\bm{s}}_{\bm{d}}$ as $l\rightarrow \infty$. For any $R>0$, $\varphi(|x|/R)S_{\mu_l^1}^{-1}T^{-1}_{x_l^1}u_l(x)$ is uniformly bounded in $H^{\bm{s}}_{\bm{d}}$. We can choose a subsequence such that $S_{\mu_l^1}^{-1}T^{-1}_{x_l^1}u_l$ converges to $U^1$ almost everywhere.
		
		Since
		$$\|P_{k_l^1} u_l\|_{L^\infty}\lesssim 2^{-k_{l,j}^1s_j}\|D^{s_j}_{x_j}P_{k_l^1} u_l\|_{L^\infty}\lesssim 2^{-k_{l,j}^1s_j+\frac{1}{2}\sum_{m=1}^n k_{l,m}^1d_m},$$
		it follows that for every $1\leq j\leq n$,
		$$2^{\frac{s}{d}\sum_{m=1}^n k_{l,m}^1d_m}=2^{(\frac{1}{2}-\frac{1}{q})\sum_{m=1}^n k_{l,m}^1d_m}\gtrsim 2^{k_{l,j}^1 s_j}.$$
		Recall that $\sum_{m=1}^n d_m/s_m = d/s$. We obtain
		\begin{equation}\label{scalingrela}
			\mu_l^1=2^{\frac{s}{d}\sum_{m=1}^n k_{1,m}^1d_m} \sim 2^{k_{1,j}^1 s_j},\quad \forall~1\leq j\leq n.
		\end{equation}
		By \eqref{comscfreqc} and \eqref{scalingrela}, we obtain
		\begin{align*}
			\left\|\sum_{ \langle k\rangle\sim 1}P_kS_{\mu_l^1}^{-1}T^{-1}_{x_l^1}u_l\right\|_{L^\infty}&\gtrsim \left\|S_{\mu_l^1}^{-1}T^{-1}_{x_l^1}\sum_{2^{k_js_j}\sim \mu_l^1}P_ku_l\right\|_{L^\infty}\\
			& \gtrsim (\mu_l^1)^{-\frac{d}{sq}}\|P_{k_l^1}u_l\|_{L^\infty}\\
			&\gtrsim 1.
		\end{align*}
		Since $\|\partial_x^\alpha P_k S_{\mu_l^1}^{-1}T^{-1}_{x_l^1}u_l\|_{L^\infty}\lesssim 1$ for any $\langle k\rangle\lesssim 1$, we would also assume $P_k S_{\mu_l^1}^{-1}T^{-1}_{x_l^1}u_l$ converges to $P_k U^1$ in $L^\infty$. Then, by Bernstein's inequality, we obtain
		$$\|U^1\|_{L^q}\gtrsim \sum_{\langle k\rangle \lesssim 1}\|P_k U^1\|_{L^q}\gtrsim \sum_{\langle k\rangle \lesssim 1}\|P_k U^1\|_{L^\infty}\gtrsim 1.$$
		Then by the refined Fatou lemma we have
		\begin{align*}
			\lim_{l\rightarrow \infty}\|r_l^1\|^q_{L^q} & = \lim_{l\rightarrow \infty}\|S_{\mu_l^1}^{-1}T^{-1}_{x_l^1}u_l-U^1\|_{L^q}^q\\
			& = \lim_{l\rightarrow \infty} \|u_l\|_{L^q}^q-\|U^1\|_{L^q}^q = \epsilon - \|U^1\|_{L^q}^q.
		\end{align*}
		For the sequence $r_l^1$, we can repeat the former procedure. Then 
		we obtain $\{U^j\}$, $\{\mu_l^j\}$, $\{x_l^j\}$, and
		$$\lim_{l\rightarrow \infty}\|r_l^k\|_{L^q}^q = \epsilon-\sum_{j=1}^k \|U^j\|_{L^q}^q.$$
		Note that if $\lim_{l\rightarrow \infty}\|r_l^k\|_{L^q}^q\neq 0$, we have $\|U^j\|_{L^q}\gtrsim_{A,\epsilon} 1$ for any $1\leq j\leq k$. Thus
		$$\lim_{k\rightarrow \infty}\lim_{l\rightarrow \infty}\|r_l^k\|_{L^q}^q = 0.$$
		Assume $S_{\mu_l^j}^{-1}T^{-1}_{x_l^j}r_l^{k}\rightharpoonup 0$ weakly in $\dot{H}^{\bm{s}}_{\bm{d}}$ as $l\rightarrow \infty$ for any $1\leq j\leq k$, $k\leq k_0$. The case $k_0 = 1$ follows from the construction.
		By the definition of $U^{k_0+1}$, we know $S_{\mu_l^{k_0+1}}^{-1}T^{-1}_{x_l^{k_0+1}}r_l^{k_0+1}\rightharpoonup 0$ weakly in $\dot{H}^{\bm{s}}_{\bm{d}}$ as $l\rightarrow \infty$. By \eqref{defimmu}, for $1\leq j\leq k_0$, we have
		\begin{equation}\label{orthrelation}
			\begin{aligned}
				S_{\mu_l^{j}}^{-1}T^{-1}_{x_l^{j}}r_l^{k_0+1} & =  S_{\mu_l^{j}}^{-1}T^{-1}_{x_l^{j}}r_l^{k_0}-S_{\mu_l^{j}}^{-1}T^{-1}_{x_l^{j}}T_{x_l^{k_0+1}}S_{\mu_l^{k_0+1}}U^{k_0+1}\\
				&=  S_{\mu_l^{j}}^{-1}T^{-1}_{x_l^{j}}r_l^{k_0}-T_{m_{\frac{1}{\mu_l^j}}(x_l^{k_0+1}-x_l^{j})}S_{\mu_l^{k_0+1}/\mu_l^{j}}U^{k_0+1}.
			\end{aligned}
		\end{equation}
		
		If $\mu_l^{k_0+1}/\mu_l^{j}\rightarrow \mu\in (0,\infty)$ and $m_{{1}/{\mu_l^j}}(x_l^{k_0+1}-x_l^{j})\rightarrow x\in \mathbb{R}^d$ as $l\rightarrow\infty$, we have, weakly in $\dot{H}^{\bm{s}}_{\bm{d}}$,
		\begin{align*}
			\lim_{l\rightarrow\infty}S_{\mu_l^{k_0+1}}^{-1}T^{-1}_{x_l^{k_0+1}}r_l^{k_0+1} = \lim_{l\rightarrow\infty}S_\mu^{-1}T_{x_*}^{-1}S_{\mu_l^{j}}^{-1}T^{-1}_{x_l^{j}}r_l^{k_0+1} = -U^{k_0+1}.
		\end{align*}
		Thus $U^{k_0+1} = 0$, which implies $U^j = 0$ for all $j\geq k_0+1$, we are free to choose $\{x_l^j\}$ and $\{\mu_l^j\}$ arbitrarily for such $j$.
		
		If
		$$\lim_{l\rightarrow \infty} \left|\log\frac{\mu_{l}^j}{\mu_l^k}\right|+|m_{\frac{1}{\mu_l^j}}(x_l^j-x_l^k)| = \infty,$$ by \eqref{orthrelation}, we obtain that $S_{\mu_l^{j}}^{-1}T^{-1}_{x_l^{j}}r_l^{k_0+1}\rightharpoonup 0$ converges weakly in $\dot{H}^{\bm{s}}_{\bm{d}}$ as $l\rightarrow \infty$.
		
		In the proof, we may choose sequences $\{x_l^j\}$ and $\{\mu_l^j\}$ such that \eqref{almostorthc} holds. Then, by \eqref{almostorthc}, we obtain \eqref{lqdecc}--\eqref{l2decc}.
	\end{proof}
	
	To avoid redundancy, we omit the routine derivation of Theorem \ref{existextri} from Lemmas \ref{linearprofilesubcritical} and \ref{criticallinearprofiledec}; the same argument will be presented in a slightly more general setting in Propositions \ref{generalso} and \ref{crigeneralso} in the next section.
	
	\section{Construction of soliton solutions for anisotropic fractional Schr\"{o}dinger equations}\label{application}
	In this section, we consider the anisotropic fractional Schr\"{o}dinger equation
	\begin{equation}\label{generalnon}
		i\partial_t u+\sum_{j=1}^nD_{x_j}^{2s_j}u = |u|^{q-2}u,
	\end{equation}
	where $u(t,x) = u(t,x_1,\cdots,x_n)$ is a complex-valued unknown function of $t\in \mathbb{R}$ and $x = (x_1,\cdots,x_n)x_j\in \mathbb{R}^{d_1}\oplus\cdots \oplus \mathbb{R}^{d_j}$. We impose the following standing assumptions: $n\in \mathbb{N}$, $d_j\in \mathbb{N}$, $s_j>0$ for any $1\leq j\leq n$,
	$$d = \sum_{j=1}^nd_j,\quad 2<q<\infty,\quad \omega\geq 0,\quad v\in \mathbb{R}^d,$$
	and
	$$\bm{d} = (d_1,\cdots,d_n),\quad\bm{s} = (s_1,\cdots,s_n),\quad \frac{d}{s} = \sum_{j=1}^n\frac{d_j}{s_j}.$$
	
	If $s_j = 1$ for all $j$, then \eqref{generalnon} reduces to the classical nonlinear Schr\"{o}dinger equation. For the construction of soliton solutions in this setting, we refer to Weinstein \cite{weinstein1982nonlinear}; for the fractional case, see Guo–Huang \cite{guo2012existence}; and for the anisotropic case where $0<s_j\leq 1$ for each $j$, we refer to Esfahani \cite{esfahani2015anisotropic}.
	
	The equation \eqref{generalnon} has the following conservation laws:
	\begin{align*}
		&\mathrm{Mass}: &M(u)&=\int_{\mathbb{R}^d}|u|^2~dx,\\
		&\mathrm{Momentum}:&P(u)& = \int_{\mathbb{R}^d}\mathrm{Im}(\bar{u}\nabla u)~dx,\\
		&\mathrm{Energy}:&E(u)& = \int_{\mathbb{R}^d}\frac{1}{2}\sum_{j=1}^n |D_{x_j}^{s_j}u|^2-\frac{1}{q}|u|^q~dx.
	\end{align*}	
	
	We refer to $H_{\bm{d}}^{\bm{s}}$ as the energy space of the equation \eqref{generalnon}. Recall the definition of $s$ by \eqref{avreandin}. If $1/q>1/2-s/d$, we call the equation \eqref{generalnon} \textit{energy subcritical}; and if $1/q = 1/2-s/d$, we call it \textit{energy critical}.	
	
	\subsection{Soliton solutions via the variational method}\label{solitonsolu}	
	Suppose there exist constants $v\in \mathbb{R}^d$, $\omega\in \mathbb{R}$, and a nonzero function $Q$ such that $e^{-i\omega t}Q(x-vt)$ solves \eqref{generalnon} in some sense. Then we call $Q$ a soliton solution of \eqref{generalnon} with velocity $c$. When $v = 0$, we refer to $Q$ as a solitary solution. Substituting this ansatz into \eqref{generalnon} formally yields the elliptic equation
	\begin{equation}\label{ellip}
		\sum_{j=1}^n D_{x_j}^{2s_j}Q+\omega Q-iv\cdot\nabla Q = |Q|^{q-2}Q.
	\end{equation}
	
	As we shall see, the values $s_j = 1/2$ and $s_j = 1$ play a distinguished role in the subsequent analysis. To this end, we introduce the index sets
	\begin{equation}\label{someindexsets}
		\begin{aligned}
			J_-:=\{1\leq j\leq n:s_j<1/2\},\quad J_+:=\{1\leq j\leq n:s_j>1/2\},\\
			J:=\{1\leq j\leq n:s_j=1/2\},\quad
			\tilde{J}:=\{j\in J_+:1/2<s_j<1\}.
		\end{aligned}
	\end{equation}
	We adopt the standard convention that sums and products over the empty set are $0$ and $1$, respectively.
	
	Let $v = (v_1,\cdots,v_n)\in \mathbb{R}^d$ with $v_j\in \mathbb{R}^{d_j}$. For the subcritical case, we assume
	\begin{equation}\label{subcriass}
		v_j = 0,~j\in J_-;\quad |v_j|<1,~j\in J;\quad \omega>\sum_{j\in J_+} (2s_j-1)\left(\frac{|v_j|}{2s_j}\right)^{\frac{2s_j}{2s_j-1}}.
	\end{equation}
	For the critical case, we assume
	\begin{equation}\label{criass}
		v_j = 0,~j\in J_-\cup J_+;\quad |v_j|<1,~j\in J;\quad \omega = 0.
	\end{equation}

	Let $w(\xi)$ denote the symbol of \eqref{ellip}, defined by
	\begin{equation}\label{symbol}
		w(\xi) = \omega+v\cdot \xi+\sum_{j=1}^n|\xi_j|^{2s_j}.
	\end{equation}
	Observe that, under assumption \eqref{subcriass}, there is a constant $c(\omega,\bm{s},v)>0$ such that $w(\xi)\geq c(\omega,\bm{s},v)$ for all $\xi\in \mathbb{R}^d$.
	
	\begin{thm}\label{generalso}
		Assume $1/q>1/2-s/d$ together with \eqref{subcriass}. Then there exists a nonzero solution $Q\in H^{\bm{s}}_{\bm{d}}$ of \eqref{ellip}, which is a maximizer of the inequality 
		$$\|u\|_{L^q(\mathbb{R}^d)}\leq C_{\omega,v}\|(w(\xi))^{1/2}\hat{u}(\xi)\|_{L^2_\xi},$$
		where $w(\xi)$ is defined by \eqref{symbol}. Moreover, $Q$ satisfies the equality $\|u\|_{L^q(\mathbb{R}^d)}^q = \|(w(\xi))^{1/2}\hat{u}(\xi)\|_{L^2_\xi}^2$.
	\end{thm}
	\begin{proof}[\textbf{Proof}]
		Define the Weinstein functional by
		\begin{equation*}
			W[u]:=  \frac{\|u\|_{L^q}}{\|(w(\xi))^{1/2}\hat{u}(\xi)\|_{L^2}},\quad u\in H^{\bm{s}}_{\bm{d}}\setminus\{0\}.
		\end{equation*}
		and set
		\begin{equation}\label{wein}
			C_{\omega,v}:= \sup_{0\neq u\in H^{\bm{s}}_{\bm{d}}} W[u].
		\end{equation}
		Let $\{u_l\}$ be a maximizing sequence for the Weinstein functional. We normalize it so that $\|u_l\|_{L^q}^q =  \|(w(\xi))^{\frac{1}{2}}\widehat{u_l}(\xi)\|_{L^2}^2$. Then we have 
		$$\lim_{l\rightarrow \infty}\|u_l\|_{L^q}^{2-q} = C_{\omega,v}^2,\quad\|u_l\|_{H^{\bm{s}}_{\bm{d}}}\lesssim 1.$$
		Applying Proposition \ref{linearprofilesubcritical}, we obtain the corresponding profiles $\{U^j\}$, centers $\{x_l^j\}$, and remainders $r_l^k$. By \eqref{lqdec}--\eqref{l2dec} and the H\"{o}lder inequality, we have
		\begin{align*}
			C_{\omega,v}^\frac{2q}{2-q} = \sum_{j=1}^\infty \|U^j\|_{L^q}^q&\leq C_{\omega,v}^q \sum_{j=1}^\infty \|(w(\xi))^{\frac{1}{2}}\widehat{U^j}(\xi)\|_{L^2}^q\\
			& \leq C_{\omega,v}^q  \left(\sum_{j=1}^\infty \|(w(\xi))^{\frac{1}{2}}\widehat{U^j}(\xi)\|_{L^2}^2\right)^{\frac{q}{2}}\\
			& \leq C_{\omega,v}^q\limsup_{l\rightarrow\infty}\|(w(\xi))^{\frac{1}{2}}\widehat{u_l}(\xi)\|_{L^2}^q\\
			& = C_{\omega,v}^\frac{2q}{2-q}.
		\end{align*}
		Thus $U^j = 0$ for any $j\geq 2$. Applying \eqref{l2dec} once more yields
		$$\lim_{l\rightarrow\infty}\|r_l^1\|_{H^{\bm{s}}_{\bm{d}}} = 0.$$
		Consequently, $u_l(\cdot-x_l^1)\rightarrow U^1$ in $H^{\bm{s}}_{\bm{d}}$, so $U^1$ is a maximizer of the functional $W$, which we denote by $Q$. The Euler–Lagrange equation for $Q$ is precisely \eqref{ellip}, finishing the proof.
	\end{proof}

	\begin{thm}\label{crigeneralso}
		Assume $s<d/2$, $1/q=1/2-s/d$, and \eqref{criass}.	Then there exists a nonzero solution $Q\in \dot{H}^{\bm{s}}_{\bm{d}}$ of \eqref{ellip}, which is a maximizer of the inequality
		\begin{equation*}
			\|u\|_{L^q(\mathbb{R}^d)} \leq B_v\prod_{j=1}^n\|D^{s_j}_{x_j,v_j}u\|_{L^2(\mathbb{R}^d)}^{\frac{sd_j}{ds_j}},
		\end{equation*}
		where $D_{x_j,v_j}^{s_j} = \mathscr{F}^{-1}(|\xi_j|+v_j\cdot \xi_j)^{s_j}\mathscr{F}$. Moreover, for each $j = 1,\cdots,n$, $Q$ satisfies $sd_j\|Q\|_{L^q}^q = s_jd\|D_{x_j,v_j}^{s_j}Q\|_{L^2}^2$.
	\end{thm}
	\begin{proof}[\textbf{Proof}]
		Let $\theta_j = sd_j/(ds_j)$ and
		\begin{equation}\label{defiBv}
			W[u] = \frac{\|u\|_{L^q}}{\prod_{j=1}^n\|D^{s_j}_{x_j,v_j}u\|_{L^2}^{\theta_j}}, \quad B_v:=\sup_{0\neq u\in  \dot{H}^{\bm{s}}_{\bm{d}}}W[u].
		\end{equation}
		By Theorem \ref{mainani}, we know that $W[u]$ is bounded. Consider the maximizing sequence $\{u_l\}$ of the upper functional. By scaling, we would assume
		\begin{equation}\label{scalingnormal}
			\|u_l\|_{L^q}^q = \frac{\|D_{x_j,v_j}^{s_j}u_l\|_{L^2}^2}{\theta_j},\quad j=1,\cdots,n.
		\end{equation}
		Then we have $\lim_{l\rightarrow \infty}\|u_l\|_{L^q}^{2-q} = B_v^2\prod_{j=1}^n \theta_j^{\theta_j}$ and $\|u_l\|_{\dot{H}^{\bm{s}}_{\bm{d}}}\lesssim 1$. Proposition \ref{criticallinearprofiledec} yields the profiles $\{U^j\}$, centers  $\{x_l^j\}$,  scaling parameters $\{\mu_l^j\}$, and remainders $r_l^k$. Combining \eqref{lqdecc}--\eqref{l2decc} with H\"{o}lder's inequality, we obtain
		\begin{align*}
			\theta_j \left(B_v^2\prod_{j=1}^n \theta_j^{\theta_j}\right)^\frac{q}{2-q} = \limsup_{l\rightarrow\infty}\|D^{s_j}_{x_j,v_j}u_l\|_{L^2}^2\geq \sum_{j=1}^\infty \|D^{s_j}_{x_j,v_j}U^j\|_{L^2}^2
		\end{align*}
		and
		\begin{align*}
			\left(B_v^2\prod_{l=1}^n \theta_l^{\theta_l}\right)^\frac{q}{2-q}  = \sum_{j=1}^\infty \|U^j\|_{L^q}^q&\leq B_v^q \sum_{j=1}^\infty \left(\prod_{l=1}^n\|D_{x_l,v_l}^{s_l}U^j\|_{L^2}^{\theta_l}\right)^q\\
			& \leq B_v^q \prod_{l=1}^n \left(\sum_{j=1}^\infty \|D_{x_l,v_l}^{s_l}U^j\|_{L^2}^q\right)^{\theta_l}\\
			& \leq B_v^q \prod_{l=1}^n \left(\sum_{j=1}^\infty \|D_{x_l,v_l}^{s_l}U^j\|_{L^2}^2\right)^{\frac{q\theta_l}{2}}\\
			&\leq B_v^q \prod_{l=1}^n\left(\theta_l \left(B_v^2\prod_{j=1}^n \theta_j^{\theta_j}\right)^\frac{q}{2-q}\right)^{\frac{q\theta_l}{2}}\\
			& = \left(B_v^2\prod_{l=1}^n \theta_l^{\theta_l}\right)^\frac{q}{2-q}.
		\end{align*}
		For $q>2$, $|a_j|_{l^q_j} = |a_j|_{l^2_j}$ implies at most one nonzero $j$. Hence $U^j = 0$ for $j\geq 2$. Applying \eqref{l2decc} again, we have 
		$$\lim_{l\rightarrow\infty}\|r_l^1\|_{\dot{H}^{\bm{s}}_{\bm{d}}} = 0.$$
		Thus, $S_{\mu_l^1}^{-1}T^{-1}_{x_l^1}u_l$ converges to $U^1$ in $\dot{H}^{\bm{s}}_{\bm{d}}$. Then $U^1$ is a maximizer of $W$. We denote it by $Q$. By \eqref{scalingnormal}, we obtain $sd_j\|Q\|_{L^q}^q = s_jd\|D_{x_j,v_j}^{s_j}Q\|_{L^2}^2$ for each $j$. The Euler–Lagrange equation for $Q$ is \eqref{ellip}, which completes the proof.
	\end{proof}
	

	\subsection{Estimates for soliton solutions in the energy space}\label{solitonestimate}
	In this subsection, we present a nonscattering result generalizing the result of Bellazzini–Georgiev–Lenzmann–Visciglia \cite{bellazzini2019on}. We first recall the definition of scattering (see Tao \cite{tao2006nonlinear}).
	\begin{defi}
		Consider the general dispersive equation
		$$i\partial_tu + Lu = N(u),$$
		where $L = \mathscr{F}^{-1}L(\xi)\mathscr{F}$ is a Fourier multiplier operator, and where the initial data satisfies $u(0,x) = u_0(x)\in X$ for some Banach space $X$. Given a solution $u\in C([0,\infty);X)$ of this equation, we say that $u$ scatters forward if there exists $u_0^+\in X$ such that the corresponding linear solution $u^{+}(t) = e^{itL}u_0^+$ satisfies
		$$\lim_{t\rightarrow +\infty}\|u(t)-u^+(t)\|_{X} = 0.$$
		Scattering backward is defined analogously.
	\end{defi}
	Recalling the index sets $J$ and $\tilde{J}$ from \eqref{someindexsets}, we have the following result.
	\begin{thm}\label{subcrnonscattering}
		Let $1/q>1/2-s/d$. If $J\neq \varnothing$, or if $$2<q<2+\frac{4}{\sum_{j\in \tilde{J}}d_j+\sum_{j\notin \tilde{J}}d_j/s_j},$$
		then small-data scattering for \eqref{generalnon} does not hold in $H^{\bm{s}}_{\bm{d}}$.
	\end{thm}
	\begin{thm}\label{crnonscattering}
		Let $s<d/2$ and $1/q = 1/2-s/d$. If $\sum_{j\in J}d_j\geq 2$, then small initial data scattering for \eqref{generalnon} does not hold in $\dot{H}^{\bm{s}}_{\bm{d}}$.
	\end{thm}

	First, we estimate the best constants in Theorems \ref{generalso} and \ref{crigeneralso}, and then prove the nonscattering result for \eqref{generalnon}.
	
	\begin{lemma}\label{theesitmateforC}
		Let $C_{\omega,v}$ be as in \eqref{wein}. Set
		\begin{equation}\label{mX}
			m := \left( \omega - \sum_{j\in J_+} (2s_j-1)\left(\frac{|v_j|}{2s_j}\right)^{\frac{2s_j}{2s_j-1}} \right)^{1/2},\quad
			X_j := \frac{|v_j|^{s_j/(2s_j-1)}}{m}.
		\end{equation}
		Let $E := \{0\}\cup\{j:s_j>1/2,\ X_j\ge1\}$, $X_0:=1$, and define
		$$X := \max_{l\in E}\left( X_l^{\frac{d}{s}(\frac12-\frac1q)-1}
		\prod_{j\in E,\ X_j\ge X_l}
		\left(\frac{X_l}{X_j}\right)^{(s_j-1)\frac{d_j}{s_j}(\frac12-\frac1q)} \right).$$
		Then
		\begin{equation*}
			C_{\omega,v}\sim_{\bm{d},\bm{s},q} m^{\frac{d}{s}(\frac{1}{2}-\frac{1}{q})-1}X\prod_{j\in J}(1-|v_j|)^{-\frac{d_j+1}{2}(\frac{1}{2}-\frac{1}{q})}.
		\end{equation*}
	\end{lemma}
	\begin{proof}[\textbf{Proof}]
		From the definitions of $m$ and $X_j$, it follows that
		\begin{equation}\label{wxiesti}
			\begin{aligned}
				&\quad w(\xi)\\
				&= m^2\Bigg(1+\sum_{j\in J_-}|\xi_jm^{-\frac{1}{s_j}}|^{2s_j}+\sum_{j=J}(|m^{-2}\xi_j|+v_j\cdot m^{-2}\xi_j)\\
				&\quad +\sum_{j\in J_+}(|m^{-\frac{1}{s_j}}\xi_j|^{2s_j}+m^{-2+\frac{1}{s_j}}v_j\cdot m^{-\frac{1}{s_j}}\xi_j+\frac{2s_j-1}{(2s_j)^{2s_j/(2s_j-1)}}X_j^2\Bigg).
			\end{aligned}
		\end{equation}		
		By scaling, we have $C_{\omega,v} = m^{d/s(1/2-1/q)-1}C_{\tilde{w},\tilde{v}}$, where
		$$\tilde{v}_j = m^{-2+\frac{1}{s_j}}v_j,\quad \tilde{w} = 1+\sum_{j\in J_+}(2s_j-1)\left(\frac{|\tilde{v}_j|}{2s_j}\right)^{\frac{2s_j}{2s_j-1}}.$$
		We may therefore set $m = 1$. For $\xi_j = (\xi_{j,1},\xi_j')\in \mathbb{R}\times\mathbb{R}^{d_j-1}$, rotational invariance allows us to assume, without loss of generality, that $v_j = (-|v_j|,0,\cdots,0)$. Consequently,
		\begin{align*}
			w(\xi) & = 1+\sum_{j\in J_-}|\xi_j|^{2s_j}+\sum_{j\in J}(|\xi_j|-|v_j|\xi_{j,1})\\
			&\quad + \sum_{j\in J_+}\left(|\xi_j|^{2s_j}-X_j^{\frac{2s_j-1}{s_j}}\xi_{j,1}+\frac{2s_j-1}{(2s_j)^{2s_j/(2s_j-1)}}X_j^2\right)
		\end{align*}
		
		Define
		$$J_+':=\{j\in J_+:X_j<1\},\quad S:=\{j\in J_+:X_j\geq 1\}.$$
		Let
		$$|\xi_j|\sim 2^{k_j}~(j\in J_-\cup J_+');\quad \xi_{j,1}\sim 2^{k_j}~(j\in J\cup S);\quad |\xi_j'|\sim 2^{k_j'}~(j\in J\cup S).$$ 
		Fix $j\in S$.
		\begin{itemize}
			\item If $2^{k_j'}\gtrsim 2^{k_j}\sim X_j^{1/s_j}$, then
			$$|\xi_j|^{2s_j}-X_j^{\frac{2s_j-1}{s_j}}\xi_{j,1}+\frac{2s_j-1}{(2s_j)^{2s_j/(2s_j-1)}}X_j^2\sim 2^{2s_jk_j'}.$$
			\item If $2^{k_j'}\ll 2^{k_j}\sim X_j^{1/s_j}$, then
			\begin{align*}
				&\quad|\xi_j|^{2s_j}-X_j^{\frac{2s_j-1}{s_j}}\xi_{j,1}+\frac{2s_j-1}{(2s_j)^{2s_j/(2s_j-1)}}X_j^2\\
				&\sim \xi_{j,1}^{2s_j}-X_j^{\frac{2s_j-1}{s_j}}\xi_{j,1}+\frac{2s_j-1}{(2s_j)^{2s_j/(2s_j-1)}}X_j^2+2^{2k_j'+2s_jk_j-2k_j}\\
				&\sim X_j^{2-{2}/{s_j}} (\xi_{j,1}-c_jX^{1/s_j})^2+2^{2k_j'+2s_jk_j-2k_j},
			\end{align*}
			where $c_j>0$ is the constant satisfying
			$$c_j^{2s_j}-c_j+(2s_j-1)/(2s_j)^{2s_j/(2s_j-1)} = 0.$$
			If $|\xi_{j,1}-c_jX^{1/s_j}|\sim 2^{l_j}$ with $2^{l_j}\lesssim X_j^{1/s_j}$, then
			$$|\xi_j|^{2s_j}-X_j^{\frac{2s_j-1}{s_j}}\xi_{j,1}+\frac{2s_j-1}{(2s_j)^{2s_j/(2s_j-1)}}X_j^2\sim X_j^{2-2/s_j}(2^{2l_j}+2^{2k_j'}).$$
			
			\item If $2^{k_j}\gg X_j^{1/s_j}$, then
			$$|\xi_j|^{2s_j}-X_j^{\frac{2s_j-1}{s_j}}\xi_{j,1}+\frac{2s_j-1}{(2s_j)^{2s_j/(2s_j-1)}}X_j^2\sim 2^{2s_jk_j}+2^{2s_jk_j'}.$$
			\item If $2^{k_j}\ll X_j^{1/s_j}$, then
			$$|\xi_j|^{2s_j}-X_j^{\frac{2s_j-1}{s_j}}\xi_{j,1}+\frac{2s_j-1}{(2s_j)^{2s_j/(2s_j-1)}}X_j^2\sim X_j^2+2^{2s_jk_j'}.$$
		\end{itemize}
		Let $\tilde{S}:=\{j\in S: 2^{k_j'}\ll 2^{k_j}\sim X_j^{1/s_j}\}$. 
		Then for any $u\neq 0$, $\mathrm{supp}(\mathscr{F}u)\subset \{\xi:|\xi_j|\sim 2^{k_j}(j\in J_-\cup J_+');~\xi_{j,1}\sim 2^{k_j}(j\in J\cup S);~|\xi_j'|\sim 2^{k_j'}(j\in J\cup S);~|\xi_{j,1}-c_jX^{1/s_j}|\sim 2^{l_j}(j\in \tilde{S})\}$. We have
		\begin{align*}
			W[u] =  \frac{\|u\|_{L^q}}{\|(w(\xi))^{1/2}\hat{u}(\xi)\|_{L^2}}\sim \frac{1}{1+\sum_{j\in J_-\cup J_+'}2^{s_jk_j}+A+B+C} \frac{\|u\|_{L^q}}{\|u\|_{L^2}}
		\end{align*}
		where
		\begin{align*}
			A&:=\sum_{j\in J}\frac{(1-|v_j|)^{1/2}2^{k_j}+2^{k_j'}}{2^{k_j/2}+2^{k_j'/2}},\\
			B&:= \sum_{j\in S\setminus \tilde{S}}m(k_j,k_j'),\\
			C&:=\sum_{j\in \tilde{S}}X_j^{1-1/s_j}(2^{l_j}+2^{k_j'}),
		\end{align*}
		and
		\begin{align*}
			m(k_j,k_j'):=\left\{
			\begin{aligned}
				&2^{s_jk_j'}, && 2^{k_j'}\gtrsim 2^{k_j}\sim X_j^{1/s_j},\\
				&2^{s_jk_j}+2^{s_jk_j'}, && 2^{k_j}\gg X_j^{1/s_j},\\
				&X_j+2^{s_jk_j'}, && 2^{k_j}\ll X_j^{1/s_j}.
			\end{aligned}
			\right.
		\end{align*}
		By the Littlewood--Paley theory (Lemma \ref{littlewoodpaley}) and Bernstein's inequality, we obtain
		\begin{align*}
			C_{\omega,v}\sim \sup_{k_j\in \mathbb{Z},1\leq j\leq n}\sup_{k_j'\in \mathbb{Z},j\in J\cup S}\sup_{2^{l_j}\leq X_j^{1/s_j}, j\in \tilde{S}}\frac{H}{1+\sum_{j\in J_-\cup J_+'}2^{s_jk_j}+A+B+C},
		\end{align*}
		where
		$$H = 2^{(\sum_{j\in J_-\cup J_+'}k_jd_j+\sum_{j\in J\cup S\setminus \tilde{S}}(k_j+k_j'(d_j-1))+\sum_{j\in \tilde{S}}(l_j+k_j'(d_j-1)))(\frac{1}{2}-\frac{1}{q})}.$$
		To achieve the maximum, we have $2^{s_jk_j}\sim 1+A+B+C$ for $j\in J_-\cup J_+'$, $2^{l_j}\gtrsim 2^{k_j'}$ for $j\in \tilde{S}$, and $2^{k_j}\gtrsim 2^{k_j'}$ for $j\in J$. Thus
		\begin{align*}
			C_{\omega,v}\sim \sup_{k_j,k_j'\in \mathbb{Z},j\in J\cup S}\sup_{2^{l_j}\leq X_j^{1/s_j},j\in \tilde{S}}\frac{\tilde{H}}{(1+\tilde{A}+B+\tilde{C})^{1-\sum_{j\in J_-\cup J_+'}\frac{d_j}{s_j}(\frac{1}{2}-\frac{1}{q})}}
		\end{align*}
		where $\tilde{H}=2^{(\sum_{j\in J\cup S\setminus \tilde{S}}(k_j+k_j'(d_j-1))+\sum_{j\in \tilde{S}}(l_j+k_j'(d_j-1)))(\frac{1}{2}-\frac{1}{q})}$ and
		\begin{align*}
			\tilde{A}:=\tilde{A}(k_j,k_j')&:=\sum_{j\in J}((1-|v_j|)^{1/2}2^{k_j/2}+2^{k_j'-k_j/2})\chi_{2^{k_j}\gtrsim 2^{k_j'}},\\
			\tilde{C}&:=\sum_{j\in \tilde{S}}X_j^{1-1/s_j}2^{l_j}.
		\end{align*}
		For any $j\in J$, let $2^{k_j} = (1+B+\tilde{C})^22^{\tilde{k}_j}$, $2^{k_j'} = (1+B+\tilde{C})^22^{\tilde{k}_j'}$. Then we have
		\begin{align*}
			C_{\omega,v}&\sim \sup_{k_j,k_{j}'\in  \mathbb{Z},j\in S}\sup_{2^{l_j}\leq X_j^{1/s_j},j\in \tilde{S}}\frac{2^{(\sum_{j\in  S\setminus \tilde{S}}(k_j+k_j'(d_j-1))+\sum_{j\in \tilde{S}}(l_j+k_j'(d_j-1)))(\frac{1}{2}-\frac{1}{q})}}{(1+B+\tilde{C})^{1-\sum_{j\in J_-\cup J\cup J_+'}\frac{d_j}{s_j}(\frac{1}{2}-\frac{1}{q})}}\\
			&\quad \cdot \sup_{\tilde{k}_j,\tilde{k}_j'\in \mathbb{Z},2^{\tilde{k}_j}\gtrsim 2^{\tilde{k}_j'},j\in J}\frac{2^{(\sum_{j\in J}(\tilde{k}_j+\tilde{k}_j'(d_j-1)))(\frac{1}{2}-\frac{1}{q})}}{(1+\tilde{A}(k_j,k_j'))^{1-\sum_{j\in J_-\cup J_+'}\frac{d_j}{s_j}(\frac{1}{2}-\frac{1}{q})}}\\
			&:=UV.
		\end{align*}
		$V$ is maximized when $2^{k_j}\sim (1-|v_j|)^{-1}$ and $2^{k_j'}\sim (1-|v_j|)^{-1/2}$ for $j\in J$. Thus
		\begin{align*}
			V \sim \prod_{j\in J}(1-|v_j|)^{-\frac{d_j+1}{2}(\frac{1}{2}-\frac{1}{q})}
		\end{align*}
		Assume that $U$ attains its maximum at $k_{j,0}$ and $k_{j,0}'$ for $j\in S$, and at $l_{j,0}$ for $j\in \tilde{S}$. If $2^{k_{j,0}}\gg X_j^{1/s_j}$, we have $k_{j,0} = k_{j,0}'$. Neither the case $2^{k_{j,0}}\ll X_j^{1/s_j}$ nor the case $2^{k_{j,0}}\sim X_j^{1/s_j}\ll 2^{k_{j,0}'}$ can occur. Thus $2^{k_{j,0}}\sim 2^{k_{j,0}'}\gtrsim X_j^{1/s_{j}}$ for $j\in S\setminus \tilde{S}$. Let $x_j = 2^{k_{j,0}' s_j}$ for $j\in S\setminus \tilde{S}$ and $x_j = X_j^{1-1/s_j}2^{l_{j,0}}$ for $j\in \tilde{S}$. Thus we obtain
		\begin{align*}
			U&\sim \max_{\tilde{S}\subset S}\sup_{x_j\geq X_j, j\in S\setminus \tilde{S}}\sup_{0\leq x_j\leq X_j, j\in \tilde{S}}\frac{(\prod_{j\in S\setminus \tilde{S}}x_j^{1/s_j}\prod_{j\in \tilde{S}}(x_jX_j^{1/s_j-1}))^{d_j(\frac{1}{2}-\frac{1}{q})}}{(1+\sum_{j\in S}x_j)^{1-\sum_{j\in J_-\cup J\cup J_+'}\frac{d_j}{s_j}(\frac{1}{2}-\frac{1}{q})}}\\
			&\sim \sup_{\tilde{S}\subset S}\sup_{x\geq X_ {\tilde{S}}}\sup_{0\leq x_j\leq X_j, j\in \tilde{S}} \frac{x^{\sum_{j\in S\setminus \tilde{S}}\frac{d_j}{s_j}(\frac{1}{2}-\frac{1}{q})}\prod_{j\in \tilde{S}}(x_jX_j^{1/s_j-1})^{d_j(\frac{1}{2}-\frac{1}{q})}}{(x+\sum_{j\in \tilde{S}}x_j)^{1-\sum_{j\in J_-\cup J\cup J_+'}\frac{d_j}{s_j}(\frac{1}{2}-\frac{1}{q})}}\\
			&:= \sup_{\tilde{S}\subset S}f_{\tilde{S}}\prod_{j\in J}(1-|v_j|)^{-\frac{d_j+1}{2}(\frac{1}{2}-\frac{1}{q})}
		\end{align*}
		where $X_{\tilde{S}}:=\max_{j\in S\setminus\tilde{S}}X_j$ (If $S\setminus \tilde{S} = \varnothing$, $X_{\tilde{S}} = 1$). Let $\mathfrak{S}:=\{\tilde{S}\subset S:X_j\geq X_{\tilde{S}}\mbox{ for any }j\in \tilde{S}\}$. It suffices to consider the maximum over $S\in \mathfrak{S}$.
		
		If $\tilde{S} = \varnothing$, we have
		$$f_{\tilde{S}}\sim \frac{X_{\varnothing}^{\sum_{j\in S}\frac{d_j}{s_j}(\frac{1}{2}-\frac{1}{q})}}{X_\varnothing^{1-\sum_{j\in J_-\cup J\cup J_+'}\frac{d_j}{s_j}(\frac{1}{2}-\frac{1}{q})}}\sim X_{\varnothing}^{\frac{d}{s}(\frac{1}{2}-\frac{1}{q})-1}.$$
		
		If $\tilde{S}\neq \varnothing$, let $K = \#\tilde{S}$, and let $j_1,\cdots,j_{K}\in \tilde{S}$ be ordered so that $X_{j_1}\leq \cdots\leq X_{j_K}$.
		
		If
		$$\sum_{j\in \tilde{S}}d_j(\frac{1}{2}-\frac{1}{q})+\sum_{j\in J_-\cup J\cup J_+'}\frac{d_j}{s_j}(\frac{1}{2}-\frac{1}{q})\leq 1,$$
		then
		\begin{align*}
			f_{\tilde{S}}&\sim\frac{X_{\tilde{S}}^{\sum_{j\in S\setminus \tilde{S}}\frac{d_j}{s_j}(\frac{1}{2}-\frac{1}{q})}\prod_{j\in \tilde{S}}(X_{\tilde{S}}X_j^{\frac{1}{s_j}-1})^{d_j(\frac{1}{2}-\frac{1}{q})}}{X_{\tilde{S}}^{1-\sum_{j\in J_-\cup J\cup J_+'}\frac{d_j}{s_j}(\frac{1}{2}-\frac{1}{q})}}\\
			&\sim X_{\tilde{S}}^{\frac{d}{s}(\frac{1}{2}-\frac{1}{q})-1}\prod_{j\in \tilde{S}}\left(\frac{X_{\tilde{S}}}{X_j}\right)^{\frac{d_j}{s_j}(\frac{1}{2}-\frac{1}{q})(s_j-1)}.
		\end{align*}
		If instead
		$$\sum_{j\in \tilde{S}}d_j(\frac{1}{2}-\frac{1}{q})+\sum_{j\in J_-\cup J\cup J_+'}\frac{d_j}{s_j}(\frac{1}{2}-\frac{1}{q})> 1,$$
		then there exists $1\leq k\leq K$ such that
		$$\sum_{l = k+1}^Kd_{j_l}(\frac{1}{2}-\frac{1}{q})+\sum_{j\notin \tilde{S}}\frac{d_j}{s_j}(\frac{1}{2}-\frac{1}{q}) \leq 1$$
		and
		$$\sum_{l= k}^Kd_{j_l}(\frac{1}{2}-\frac{1}{q})+\sum_{j\notin \tilde{S}}\frac{d_j}{s_j}(\frac{1}{2}-\frac{1}{q}) > 1.$$
		In this case,
		\begin{align*}
			f_{\tilde{S}}& \sim \frac{X_{j_k}^{\sum_{j\in S\setminus \tilde{S}}\frac{d_j}{s_j}(\frac{1}{2}-\frac{1}{q})}\prod_{l=1}^kX_{j_l}^{\frac{d_{j_l}}{s_{j_l}}(\frac{1}{2}-\frac{1}{q})}}{X_{j_k}^{1-\sum_{j\in J_-\cup J\cup J_+'}\frac{d_j}{s_j}(\frac{1}{2}-\frac{1}{q})}}\prod_{l=k+1}^K(X_{j_k}X_{j_l}^{\frac{1}{s_{j_l}}-1})^{d_{j_l}(\frac{1}{2}-\frac{1}{q})}\\
			&\lesssim X_{j_k}^{\frac{d}{s}(\frac{1}{2}-\frac{1}{q})-1}\prod_{l=k+1}^{K}\left(\frac{X_{j_k}}{X_j}\right)^{\frac{d_j}{s_j}(\frac{1}{2}-\frac{1}{q})(s_j-1)}.
		\end{align*}
		Finally, we conclude that this term is equivalent to
		$$U\sim \max_{\tilde{S}\in \mathfrak{S}}\left(X_{\tilde{S}}^{\frac{d}{s}(\frac{1}{2}-\frac{1}{q})-1}\prod_{j\in \tilde{S}}\left(\frac{X_{\tilde{S}}}{X_j}\right)^{\frac{d_j}{s_j}(\frac{1}{2}-\frac{1}{q})(s_j-1)}\right)\sim X.$$
		The proof is complete.
	\end{proof}
	\begin{rem}\label{sleq1}
		If $s_j\geq 1$ for all $j\in S$, then $X = 1$; if instead $s_j\leq 1$ for all $j\in S$, then
		$$X = \prod_{j\in S}X_j^{(1-s_j)\frac{d_j}{s_j}(\frac{1}{2}-\frac{1}{q})}.$$
	\end{rem}
	
	\begin{lemma}\label{thehomoesti}
		Let $B_v$ be defined by \eqref{defiBv}. Then there exists $0<c<C$, which depend on $\bm{d}$ and $\bm{s}$, such that
		$$c\prod_{j\in J}(1-|v_j|)^{-\frac{d_j+1}{2}(\frac{1}{2}-\frac{1}{q})}\leq B_v \leq C\prod_{j\in J}(1-|v_j|)^{-\frac{d_j+1}{2}(\frac{1}{2}-\frac{1}{q})}.$$
	\end{lemma}
	\begin{proof}[\textbf{Proof}]
		Let $u(x) = \prod_{j=1}^n u_j(x_j)$. By the results in \cite{bellazzini2019on,bellazzini2019oncorr}, we obtain
		\begin{align*}
			B_v &\geq \prod_{j=1}^n\sup_{0\neq u_j\in H^{s_j}(\mathbb{R}^{d_j})}\frac{\|u_j\|_{L^q(\mathbb{R}^{d_j})}}{\|u_j\|^{1-\theta_j}_{L^2(\mathbb{R}^{d_j})}\|D_{x_j,v_j}^{s_j}u_j\|_{L^2(\mathbb{R}^{d_j})}^{\theta_j}}\\
			&\gtrsim \prod_{j\in J} (1-|v_j|)^{-\frac{(d_j+1)(q-2)}{4q}}.
		\end{align*}
		For the opposite direction, we may assume, without loss of generality, that $v_j = (|v_j|,0,\cdots,0)$. Consider $u\neq 0$ whose Fourier support satisfies 
		$$\mathrm{supp}(\mathscr{F}u)\subset \{\xi:|\xi_j|\sim 2^{k_j}(j\in J_-\cup J_+);~|\xi_{j,1}|\sim 2^{k_j},|\xi_j'|\sim 2^{k_j'}(j\in J)\}. $$By the Bernstein inequality, we have
		\begin{align*}
			&\quad\frac{\|u\|_{L^q}}{\prod_{j=1}^n \|D_{x_j,v_j}^{s_j}u\|_{L^2}^{\theta_j}}\\
			&\sim \frac{1}{2^{\sum_{j\in J_-\cup J_+}k_js_j\theta_j}\prod_{j\in J}(\frac{(1-|v_j|)^{1/2}2^{k_j}+2^{k_j'}}{2^{k_j/2}+2^{k_j'/2}})^{\theta_j}}\frac{\|u\|_{L^q}}{ \|u\|_{L^2}}\\
			&\lesssim \prod_{j\in J} \frac{2^{\frac{s}{d}(k_j+k_j'(d_j-1))}(2^{k_j/2}+2^{k_j'/2})^{\theta_j}}{((1-|v_j|)^{1/2}2^{k_j}+2^{k_j'})^{\theta_j}}.
		\end{align*}
		Then by the Littlewood--Paley theory (Lemma \ref{littlewoodpaley}), we obtain
		\begin{align*}
			B_v&\lesssim \sup_{k_j,k_j'\in \mathbb{Z},j\in J}\prod_{j\in J} \frac{2^{\frac{s}{d}(k_j+k_j'(d_j-1))}(2^{k_j/2}+2^{k_j'/2})^{\theta_j}}{((1-|v_j|)^{1/2}2^{k_j}+2^{k_j'})^{\theta_j}}\\
			&\sim \sup_{k_j,k_j'\in \mathbb{Z},2^{k_j}\geq 2^{k_j'},j\in J}\prod_{j\in J} \left(\frac{2^{(k_j+k_j'(d_j-1))/(2d_j)}2^{k_j/2}}{(1-|v_j|)^{1/2}2^{k_j}+2^{k_j'}}\right)^{\theta_j}\\
			&\sim \sup_{0<y_j\leq 1, j\in J}\prod_{j\in J} \left(\frac{y_j^{(d_j-1)/(2d_j)}}{(1-|v_j|)^{1/2}+y_j}\right)^{\theta_j}\\
			&\lesssim \prod_{j\in J}(1-|v_j|)^{-\frac{d_j+1}{4d_j}\theta_j}.
		\end{align*}
		Since $\theta_j = sd_j/(ds_j) = 2sd_j/d = 2(1/2-1/q)d_j$ for $j\in J$, the proof is complete.
	\end{proof}
	
	\begin{proof}[\textbf{Proof of Theorem \ref{subcrnonscattering}}]
		Let $Q_{\omega,v}$ be the solution of \eqref{ellip} constructed in Theorem \ref{generalso}. Then $u(t,x) = e^{-i\omega t}Q_{\omega,v}(x-vt)$ is a solution of \eqref{generalnon}. 
		 Moreover, it is immediate that $u$ does not scatter and that 
		 $$\|u\|_{C(\mathbb{R};H^{\bm{s}}_{\bm{d}})} = \|Q_{\omega,v}\|_{H^{\bm{s}}_{\bm{d}}}.$$
		
		By \eqref{wxiesti}, we have
		\begin{align*}
			w(\xi)^{\frac{1}{2}}&\gtrsim m+\sum_{j\in J_-\cup J_+'}|\xi_j|^{s_j}+\sum_{j\in J}(1-|v_j|)^{\frac{1}{2}}|\xi_j|^{\frac{1}{2}}\\
			&\quad+\sum_{j\in S}|\xi_j|^{s_j}\chi_{|\xi_j|\nsim |v_j|^{1/(2s_j-1)}}.
		\end{align*}
		Recall the definition of $X_j$ from \eqref{mX}. Then,
		\begin{align*}
			\sup_\xi \frac{1+\sum_{j=1}^n|\xi_j|^{s_j}}{w(\xi)^{1/2}}&\lesssim \frac{1}{m}+1+\max_{j\in J}(1-|v_j|)^{-\frac{1}{2}}+\sum_{j\in S}\frac{|v_j|^{s_j/(2s_j-1)}}{m}\\
			&\lesssim \frac{1}{m}+(1-\max_{j\in J}|v_j|)^{-\frac{1}{2}}+\max_{j\in S}X_j.
		\end{align*}
		Combining with Lemma \ref{theesitmateforC}, we obtain
		\begin{align*}
			&\quad\|Q_{\omega,v}\|_{H^{\bm{s}}_{\bm{d}}}\\
			&\lesssim \|w(\xi)^{\frac{1}{2}}\widehat{Q_{\omega,v}}(\xi)\|_{L^2}\sup_\xi \frac{1+\sum_{j=1}^n|\xi_j|^{s_j}}{w(\xi)^{1/2}}\\
			&\sim \left(\frac{1}{m}+\left(1-\max_{j\in J}|v_j|\right)^{-\frac{1}{2}}+\max_{j\in S} X_j\right) m^{\frac{q}{q-2}-\frac{d}{2s}}X^{-\frac{q}{q-2}}\prod_{j\in J}(1-|v_j|)^{\frac{d_j+1}{4}}.
		\end{align*}
		If $J\neq \varnothing$, we choose $|v_j| = \nu\in (0,1)$ for $j\in J$, $v_j = 0$ for $j\in J_+$, and $m = (1-\nu)^{1/2}$. Since $1/q>1/2-s/d$, we have
		\begin{align*}
			\|Q_{\omega,v}\|_{H^{\bm{s}}_{\bm{d}}}&\lesssim (m^{-1}+(1-\nu)^{-\frac{1}{2}}) m^{\frac{q}{q-2}-\frac{d}{2s}}(1-\nu)^{\sum_{j\in J}\frac{d_j+1}{4}}\\
			& \sim (1-\nu)^{-\frac{1}{2}+\frac{q}{2(q-2)}-\frac{d}{4s}+\sum_{j\in J}\frac{d_j+1}{4}}\\
			&\lesssim (1-\nu)^{\frac{q}{2(q-2)}-\frac{d}{4s}}\rightarrow 0,\quad \mbox{as } \nu\rightarrow 1-.
		\end{align*}
		If $J = \varnothing$, we set $v_j = 0$ for $j\in J_+\setminus \tilde{J}$, $X_j = z>1$ for $j\in \tilde{J}$, and $m = z^{-1}$. By Remark \ref{sleq1}, we have
		\begin{align*}
			\|Q_{\omega,v}\|_{H^{\bm{s}}_{\bm{d}}} &\lesssim (m^{-1}+z) m^{\frac{q}{q-2}-\frac{d}{2s}}z^{-\sum_{j\in \tilde{J}}(1-s_j)\frac{d_j}{s_j}(\frac{1}{2}-\frac{1}{q})\frac{q}{q-2}}\\
			& \sim z^{1-\frac{q}{q-2}+\frac{d}{2s}-\sum_{j\in \tilde{J}}(1-s_j)\frac{d_j}{s_j}(\frac{1}{2}-\frac{1}{q})\frac{q}{q-2}}\\
			&\sim z^{-\frac{2}{q-2}+\sum_{j\notin \tilde{J}}\frac{d}{2s_j}+\sum_{j\in \tilde{J}}\frac{d_j}{2}}.
		\end{align*}
		Thus, if $-\frac{2}{q-2}+\sum_{j\notin \tilde{J}}\frac{d}{2s_j}+\sum_{j\in \tilde{J}}\frac{d_j}{2}<0$, which is equivalent to $q<2+4/(\sum_{j\notin \tilde{J}}d_j/s_j+\sum_{j\in \tilde{J}}d_j)$, then $\|Q_{\omega,v}\|_{H^{\bm{s}}_{\bm{d}}}\rightarrow 0$ as $z\rightarrow \infty$. We finish the proof.
	\end{proof}
	\begin{proof}[\textbf{Proof of Theorem \ref{crnonscattering}}]
		Let $Q_v$ be the solution of \eqref{ellip} constructed in Theorem \ref{crigeneralso}. 
		By the definition of $D_{x_j,v_j}^{s_j}$, we have
		\begin{align*}
			\|Q_v\|_{\dot{H}^{\bm{s}}_{\bm{d}}}\sim \sum_{j=1}^n\|D_{x_j}^{s_j}Q_v\|_{L^2}
			&\lesssim \sum_{j=1}^n (1-|v_j|)^{-s_j}\|D_{x_j,v_j}^{s_j}Q_v\|_{L^2}\\
			&\lesssim \max_{j\in J}(1-|v_j|)^{-\frac{1}{2}}\|Q_v\|_{L^q}^{\frac{q}{2}}\\
			&\sim \max_{j\in J}(1-|v_j|)^{-\frac{1}{2}}B_v^{\frac{q}{2-q}}.
		\end{align*}
		Combining with Lemma \ref{thehomoesti}, we obtain
		\begin{align*}
			\|Q_{v}\|_{\dot{H}^{\bm{s}}_{\bm{d}}} \lesssim \max_{j\in J}(1-|v_j|)^{-\frac{1}{2}}\left(\prod_{j\in J}(1-|v_j|)^{-\frac{d_j+1}{2}(\frac{1}{2}-\frac{1}{q})}\right)^{\frac{q}{2-q}}.
		\end{align*}
		Let $|v_j| = \nu\in (0,1)$ for $j\in J$. Then
		\begin{align*}
			\|Q_{v}\|_{\dot{H}^{\bm{s}}_{\bm{d}}}\lesssim (1-\nu)^{-\frac{1}{2}+\sum_{j\in J}\frac{d_j+1}{4}}.
		\end{align*}
		Note that for $\sum_{j\in J}d_j\geq 2$, we have $-1/2+\sum_{j\in J}(d_j+1)/4>0$. Thus, $\|Q_v\|_{\dot{H}^{\bm{s}}_{\bm{d}}}\rightarrow 0$ as $\nu\rightarrow 1-$. The same reasoning as in Theorem \ref{subcrnonscattering} finishes the proof.
	\end{proof}

	\section*{Acknowledgment}
	Jie Chen acknowledges the support of NSFC grants 12301116. The author thanks Professor Baoxiang Wang, from whose PDE course he learned a great deal. The author also thanks Ying Zhang, Xiaojie Wang, and Danyi Li for their assistance in finding and translating literature.
	
	\phantomsection 
	\bibliographystyle{amsplain}
	\addcontentsline{toc}{section}{References}
	\bibliography{reference}	
	
	\begin{itemize}[leftmargin=5pt]
		\item[] \scriptsize\textsc{Jie Chen: School of Science, Jimei University, Xiamen 361021, P.R. China}
		
		\textit{E-mail address}: \textbf{jiechern@jmu.edu.cn}
	\end{itemize}
	
\end{document}